\documentclass{scrartcl}
\usepackage{mathtools}
\usepackage{tikz-cd}
\usepackage[backend=biber,style=alphabetic,maxbibnames=9,backref=true,maxalphanames=4]{biblatex}
\usepackage{amsmath}
\usepackage{amsthm}
\usepackage{thmtools}
\usepackage{amssymb}
\usepackage{enumitem}
\usepackage[scr=boondox,  
            cal=esstix]   
           {mathalpha}
\usepackage{quiver}
\usepackage{csquotes}
\usepackage{xurl}
\usepackage{adjustbox}
\usepackage{microtype}
\usepackage{xspace}

\usepackage{hyperref}
\usepackage[capitalise]{cleveref}

\hypersetup{hypertexnames=false} 

\title{Normed idempotence of spectra and distributivity in equivariant higher algebra}
\date{}
\author{
    \textsc{Kaif Hilman}
    \thanks{
        Mathematisches Institut, Universität Bonn, Germany, \;\texttt{kaif@math.uni-bonn.de}
    } 
\and 
    \textsc{Branko Juran}
    \thanks{
        Mathematisches Institut, Universität Bonn, Germany, \;\texttt{juran@math.uni-bonn.de}
    } 
\and 
    \textsc{Natalie Stewart}
    \thanks{
        Max-Planck-Institut für Mathematik, Bonn, Germany, \;\texttt{stewart@mpim-bonn.mpg.de}
    }
}

\def\emptyset{\varnothing}

\def\bMap{\mathbf{Map}}

\newcommand{\Acc}[1]{\mathrm{Acc}_{\kappa}}
\newcommand{\GAcc}[1]{\udl{\mathrm{Acc}}_{\kappa}}

\newcommand{\Arrow}{\udl{\mathrm{Ar}}}

\newcommand{\WCat}{\underline{\mathrm{FSub}\Cat}}
\newcommand{\wWCat}{\underline{\mathrm{FSub}\wCat}}

\newcommand{\GCatterminal}{\GCat[\ast]}

\newcommand{\Cocone}[1]{\underline{\mathrm{Cocone}}\left( #1 \right)}

\theoremstyle{plain}
\newtheorem{lemma}{Lemma}[section]
\newtheorem{corollary}[lemma]{Corollary}
\newtheorem{theorem}[lemma]{Theorem}
\newtheorem{proposition}[lemma]{Proposition}

\def\sm{\mathrm{small}}

\def\cocart{\mathrm{cocart}}
\def\XX{\mathcal{X}}

\newtheorem*{mainthm}{Main Theorem}

\newtheorem{alphThm}{Theorem}

\newcommand{\neutralize}[1]{\expandafter\let\csname c@#1\endcsname\count@}
\makeatother

\makeatletter

\makeatletter

\crefname{Condition}{Condition}{Condition}

\theoremstyle{definition}
\newtheorem{definition}[lemma]{Definition}
\newtheorem{remark}[lemma]{Remark}
\newtheorem{notation}[lemma]{Notation}
\newtheorem{construction}[lemma]{Construction}
\newtheorem{example}[lemma]{Example}

\newtheorem{nota}[lemma]{Notation}

\crefname{proposition}{Proposition}{Propositions}
\Crefname{lemma}{Lemma}{Lemmas}
\Crefname{cor}{Corollary}{Corollaries}
\Crefname{thm}{Theorem}{Theorems}
\Crefname{alphThm}{Theorem}{Theorems}
\Crefname{alphCor}{Corollary}{Corollaries}
\Crefname{defn}{Definition}{Definitions}
\Crefname{nota}{Notation}{Notations}
\Crefname{notation}{Notation}{Notations}
\Crefname{cons}{Construction}{Constructions}
\Crefname{rmk}{Remark}{Remarks}
\Crefname{obs}{Observation}{Observations}
\Crefname{warning}{Warning}{Warnings}
\Crefname{conj}{Conjecture}{Conjectures}
\Crefname{recollect}{Recollection}{Recollections}
\Crefname{terminology}{Terminology}{Terminologies}
\Crefname{setting}{Setting}{Settings}
\Crefname{example}{Example}{Examples}

\newcommand{\Card}[1][]
{ \ifthenelse{\equal{#1}{}}%
{\mathrm{Card}}
{\mathrm{Card}^{#1\mathrm{-norm}}}
}

\newcommand{\eval}{\mathrm{ev}}
\newcommand{\RR}{\mathcal R}
\renewcommand{\AA}{\underline{\mathcal A}}
\newcommand{\Alg}[3][]
{ \ifthenelse{\equal{#1}{}}%
{\myuline{\mathrm{Alg}}_{#2} \left( #3 \right) }
{\myuline{\mathrm{Alg}}_{#2} \left( #3 \right)^{#1}}
}
\newcommand{\BB}{\underline{\mathcal{B}}}           
\newcommand{\BSpaces}[1][]{\Spaces[#1]_\ast}
\newcommand{\GBSpaces}[1][]{\underline{\mathrm{An}}_{\ast}^{#1}}
\newcommand{\can}{\mathrm{can}}
\newcommand{\Cat}[1][]{\mathrm{Cat}^{#1}}           
\newcommand{\wCat}[1][]{\mathrm{CAT}^{#1}}         

\newcommand{\cbr}[1]{{\left\{#1\right\}}}      
\newcommand{\C}{\mathcal C}    
\newcommand{\CC}{\underline{\mathcal C}}    
\newcommand{\cllstab}[2]{{\cite[\href{https://arxiv.org/pdf/2301.08240v2\#theorem.#2}{#1~#2}]{cll23}}}              
\DeclareMathOperator*{\colim}{colim}            
\newcommand{\DD}{\underline{\mathcal D}}    
\DeclareMathOperator{\Fun}{Fun}                 
\newcommand{\EEE}[2][]{
    \ifthenelse{\equal{#1}{}}%
    {\mathbb{E}_{#2}}{\mathbb{E}_{#2,#1}}
}
\newcommand{\EE}{\underline{\mathcal E}}
\newcommand{\ev}{\mathrm{ev}}                   
\newcommand{\FinTN}[2][]{\mathbb{F}^{#1}_{#2}}
\newcommand{\GCat}[1][]{\underline{\Cat}^{#1}}
\newcommand{\wGCat}[1][]{{\udl{\mathrm{CAT}}}^{#1}}

\newcommand{\GPr}[1][]{\underline{\Pr}^{#1}}            
\newcommand{\GSpaces}[1][]{\underline{\mathrm{An}}^{#1}}     
\newcommand{\wGSpaces}[1][]{\underline{\mathrm{AN}}^{#1}}  
\DeclareMathOperator{\id}{id}
\newcommand{\KK}{\underline{\mathcal{K}}}           
\newcommand{\JJ}{\underline{\mathcal{J}}}           
\newcommand{\op}{^{\mathrm{op}}}                       
\newcommand{\mdef}[1]{\textcolor{violet}{#1}}           
\newcommand{\MM}{\underline{\mathcal{M}}}
\newcommand{\Mod}[2]{\underline{\operatorname{Mod}}_{#1}\prn{#2}}

\newcommand{\NAlg}[2][]
{ \ifthenelse{\equal{#1}{}}%
{\myuline{\mathrm{NAlg}}_{\mathcal{N}} \left( #2 \right) }
{\myuline{\mathrm{NAlg}}_{\mathcal{N}} \left( #2 \right)^{#1}}
}
\newcommand{\NPNAlg}[2][]
{ \ifthenelse{\equal{#1}{}}%
{\mathrm{NAlg}_{\mathcal{N}} \left( #2 \right) }
{\mathrm{NAlg}_{\mathcal{N}} \left( #2 \right)^{#1}}
}
\newcommand{\NAlgidem}[2][]
{ \ifthenelse{\equal{#1}{}}%
{\myuline{\mathrm{NAlg}}^{\mathrm{idem}}_{\mathcal{N}} \left( #2 \right) }
{\mathrm{NAlg}^{\mathrm{idem}}_{\mathcal{N}} \left( #2 \right)^{#1}}
}
\newcommand{\NN}{\mathcal{N}}                       
\newcommand{\Orbit}[1]{\mathrm{Orb}_{#1}}         
\newcommand{\overslice}[2]{{#1}_{/#2}}             
\newcommand{\prn}[1]{{\left(#1 \right)}}            
\newcommand{\brk}[1]{{\left[#1 \right]}}            
\newcommand{\Spaces}[1][]{\mathrm{An}^{#1}}                  
\DeclareMathOperator{\Span}{Span} %
\newcommand{\TotFP}{\int}                 
\newcommand{\TT}{\mathcal{T}}        
\newcommand{\underslice}[2]{#1_{#2 / }}           

\DeclareFontFamily{U}{dmjhira}{}
\DeclareFontShape{U}{dmjhira}{m}{n}{ <-> dmjhira }{}
\newcommand{\yoneda}[1][]{\text{\usefont{U}{dmjhira}{m}{n}\symbol{"48}}_{#1}}                 

\newcommand{\PP}{\udl{\mathcal{P}}}

\newcommand{\SFTN}[2]{{\Span_{#2}\prn{\FinTN[#1]{}}}}

\newcommand{\CAlg}{\mathrm{CAlg}}
\newcommand{\CMon}{\mathrm{CMon}}

\newcommand{\tCat}[1][]{{\mathbb{C}\mathrm{at}}^{#1}}
\newcommand{\tGCat}[1][]{{\udl{\mathbb{C}\mathrm{at}}}^{#1}}
\newcommand{\Catcocone}{\wGCat_{\triangleright}}
\newcommand{\Catcoconebig}{\wGCat[\mathrm{big}]_{\triangleright}}
\newcommand{\GCatcoco}[1]{\wGCat\left(#1\right)}

\newcommand{\Catcoconeproj}{\wGCat[\mathrm{proj}]_{\triangleright}}
\newcommand{\smallCatcoconeproj}{\GCat[\mathrm{proj}]_{\triangleright}}

\def\animae{{animæ}}

\newcommand{\stoptocwriting}{%
  \addtocontents{toc}{\protect\setcounter{tocdepth}{-5}}}
\newcommand{\resumetocwriting}{%
  \addtocontents{toc}{\protect\setcounter{tocdepth}{\arabic{tocdepth}}}}

\usetikzlibrary{spath3}
\tikzset{between/.style n args={2}{/tikz/spath/at end path construction={
    \tikzset{spath/split at keep middle={current}{#1}{#2}}
}}}

\newcommand{\filt}{\mathrm{filt}}
\newcommand{\allcolim}{\mathrm{all}}
\newcommand{\cofree}{\underline{\mathrm{Coeff}}}
\newcommand{\fgt}{\mathrm{frgt}}

\newcommand{\constant}{\mathrm{const}}
\newcommand{\repsphere}{\mathrm{RepSph}}
\newcommand{\hosphere}{\mathrm{HoSph}}
\newcommand{\finind}{\NN_{\mathrm{fin}}}

\newcommand{\Map}{\mathrm{Map}}
\newcommand{\invObj}{\udl{\mathcal{I}}}
\newcommand{\finsmall}{\mathrm{fin}}

\newcommand{\irex}{^{\mathrm{irex}}}
\newcommand{\cat}{\mathrm{Cat}}
\newcommand{\udl}[1]{\underline{#1}}

\newcommand{\perfect}{^{\mathrm{perf}}}
\newcommand{\A}{\mathcal{A}}
\newcommand{\B}{\mathcal{B}}
\newcommand{\sC}{\mathcal{C}}
\newcommand{\D}{\mathcal{D}}
\newcommand{\E}{\mathcal{E}}
\newcommand{\R}{\mathcal{R}}
\newcommand{\K}{\mathcal{K}}
\newcommand{\canonical}{\mathrm{can}}
\newcommand{\tcocone}{^{\triangleright}}
\newcommand{\func}{\mathrm{Fun}}

\newcommand{\distributiveLocalisation}[3]{_{{#3}}^{{#2}-{\delta}_{#1}}}

\newcommand{\distributeAll}[1]{^{\delta_{#1}}}
\newcommand{\distributeSmall}[1]{^{\mathrm{fin}-\delta_{#1}}}
\newcommand{\presheaf}{\mathrm{PSh}}
\newcommand{\wpresheaf}{\mathrm{PSH}}
\newcommand{\presentable}{\mathrm{Pr}}
\newcommand{\baseCat}{\mathcal{T}}
\newcommand{\rexact}{^{{\mathrm{rex}}}}
\newcommand{\nattrans}{\mathrm{Nat}}

\newcommand{\spectra}{\mathrm{Sp}}

\newcommand{\module}{\mathrm{Mod}}
\newcommand{\calg}{\mathrm{CAlg}}
\newcommand{\unit}{\mathbf{1}}
\newcommand{\cofib}{\mathrm{cofib}}

\newcommand{\ind}{\mathrm{Ind}}
\newcommand{\topcat}{\mathrm{Top}}

\def\LargeCat{\wCat}

\usepackage{contour}
\usepackage{ulem}

\contourlength{0.5pt}

\newcommand{\myuline}[1]{%
  \uline{\phantom{#1}}%
  \llap{\contour{white}{#1}}%
}

\renewcommand{\udl}[1]{
  \myuline{#1}
}

\makeatletter
\newcommand*{\saved@myuline}{}
\let\saved@myuline\myuline

\newcommand*{\mathuline}{%
  \mathpalette{\math@myuline\saved@myuline}%
}
\newcommand*{\math@myuline}[3]{%
  \mbox{#1{$#2#3\m@th$}}%
}

\renewcommand*{\myuline}{%
  \relax  
  \ifmmode
    \expandafter\mathuline
  \else
    \expandafter\saved@myuline
  \fi
}
\makeatother

\newcommand{\llp}[2]{{\cite[\href{https://arxiv.org/pdf/2503.02839v1\#theorem.#2}{#1~#2}]{llp25}}}
\newcommand{\Shaone}[2]{{\cite[\href{https://arxiv.org/pdf/1809.05892v3\#nul.#2}{#1~#2}]{Sha23}}}
\newcommand{\Shatwo}[2]{{\cite[\href{https://arxiv.org/pdf/2109.11954v2\#nul.#2}{#1~#2}]{shahPaperII}}}
\newcommand{\Crefenum}[2]{\Cref{#1}.\ref{#2}}

\newcommand{\ha}[2]{{\cite[\href{https://people.math.harvard.edu/~lurie/papers/HA.pdf\#theorem.#2}{#1~#2}]{lur17}}}

\newcommand{\kerodon}[1]{{\cite[\href{https://kerodon.net/tag/#1}{Tag #1}]{kerodon}}}

\newcommand{\infl}[2]{\mathrm{infl}^{#2}_{#1}}
\newcommand{\coind}[2]{\mathrm{coind}_{#1}^{#2}}
\newcommand{\res}[2]{\mathrm{res}^{#1}_{#2}}

\newcommand{\Orb}{\mathrm{Orb}}
\newcommand{\LL}{\underline{\mathcal{L}}}

\newcommand{\bres}[2]{\mathbf{res}^{#1}_{#2}}
\newcommand{\yo}{\text{\usefont{U}{min}{m}{n}\symbol{'107}}}
\DeclareFontFamily{U}{min}{}
\DeclareFontShape{U}{min}{m}{n}{<-> dmjhira}{}
\newcommand{\yonedaMap}{\yo}

\newcommand*\cocolon{%
        \nobreak
        \mskip6mu plus1mu
        \mathpunct{}%
        \nonscript
        \mkern-\thinmuskip
        {:}%
        \mskip2mu
        \relax
}

\newcommand{\introheader}[1]{ 
  \vspace{\topsep} 
  \noindent \textbf{#1}
}

\let\originalleft\left
\let\originalright\right
\renewcommand{\left}{\mathopen{}\mathclose\bgroup\originalleft}
\renewcommand{\right}{\aftergroup\egroup\originalright}

\bibliography{references.bib}

\begin{document}
  
\maketitle

\vspace{-40pt}

\begin{abstract}
    We introduce and develop the theory of distributivity for base categories  equipped with so-called atomic orbital subcategories, with an eye towards applications in equivariant homotopy theory. This extends the important framework in Nardin's thesis and supplies several crucial details on distributivity missing from the treatment therein. As our main technical result, we provide a parametrised distributive version of the universal property of ``partial cocompletions with relations''. We then employ this to construct a distributively normed symmetric monoidal structure on presentable categories. In turn, by specialising to the case of equivariant homotopy theory for compact Lie groups $G$, we show that the category of genuine $G$-spectra is normed idempotent, so that we obtain the structure of multiplicative norms on $G$-spectra along inclusions of finite index subgroups.
\end{abstract}
\tableofcontents

\section{Introduction}  
Since landmark work of Hill-Hopkins-Ravenel \cite{HHR} famously leveraged multiplicative norms in finite group-equivariant homotopy theory to resolve long-standing questions about the Adams spectral sequence, such norm operations have received intense attention. 
In fact, analogous structures are of interest in the more general context of compact Lie groups.
For instance, in \cite{greenleesMayMU} (which preceded even \cite{HHR}) Greenlees-May used norm operations along finite-index inclusions of compact Lie groups to prove a completion theorem for equivariant bordism.

However, the norms are a tricky structure to construct, due in no small part to the confluence of their nonlinear and equivariant natures.
Indeed,
as these are  power operations incorporating the symmetry of the multiplied objects, standard inductive procedures to deal with each tensor factor ``one at a time'' do not apply, and one is forced to deal with ``distributivity'' in a different and coordinate-free way;
norms are thereby rendered doubly hard to treat coherently in the $\infty$-categorical setting.

To this end, Nardin introduced an elegant formalism of parametrised distributivity and symmetric monoidal structures in his thesis \cite{nardinThesis}, capturing the norm structure for finite groups in the $\infty$-categorical setting;
this provided a crucial part of the (only partially completed) Barwick-Dotto-Nardin-Glasman-Shah program concerning parametrised higher algebra \cite{Barwick-intro}.
Since then, this perspective has proved to be a powerful one and serves as one of the standard paradigms for applications, cf. e.g. \cite{carrickSlice,horevKlangZou,levyThomSpectra,quigleyShahParametrisedTate,QuinnZhu}. Unfortunately, Nardin's setup does not cover the case of compact Lie groups;
even in the finite group case, they only consider multiplicative structures which have norms along arbitrary inclusions. Recently, a lot of attention has been paid to (incomplete) transfer systems \cite{Blumberg-op,Hill_chromatic} which only incorporate norms along a distinguished class of inclusions of subgroups.
Moreover, some of the subtler details about distributivity crucial to their theory are missing, leading to what we believe are potentially genuine mistakes. 

The aims of the present work are threefold:
\begin{enumerate}[label=(\alph*)]
   \item fill in the aforementioned crucial gaps on distributivity by invoking the cubical analyses from the first author's work \cite{kaifNoncommMotives},
    \item extend Nardin's framework to one rich enough to capture finite-index norms for compact Lie groups and incomplete transfer systems, and
    \item establish a universal property and the structure of  finite-index norms on $G$-spectra in the distributive context for compact Lie groups $G$.
\end{enumerate}
Additionally, apart from supplying the details missing from the literature needed for the first author's work on equivariant algebraic K-theory for finite groups \cite{kaifNoncommMotives,hilmanRamzi} , the last two authors will use these norms for finite-index inclusions as an input to construct norms along arbitrary inclusions of compact Lie groups in future work, incarnating the Blumberg-Hill-Mandell conjectures \cite{bhm22} in the $\infty$-categorical context \cite{norms}.

To focus the rest of the introduction, let us now state our main result:

\begin{mainthm}
    Let $G$ be a compact Lie group. The category of presentable $G$-categories $\udl{\presentable}^L_G$ admits the structure of distributive finite-index norms. Moreover, the category of genuine $G$-spectra $\myuline{\spectra}_G$ is a normed idempotent in $\udl{\presentable}^L_G$;
    thus it uniquely admits the structure of coherently-compatible and distributive finite-index norms.
\end{mainthm}

We now give a more in-depth overview of the contents of the article. To begin the discussion, let us touch upon the notions of distributivity and  symmetric monoidal structures in the generality we  develop in this work. Since everything will be  $\infty$-categorical, we shall henceforth drop the prefix ``$\infty$-'' wherever applicable.

\introheader{Distributivity and symmetric monoidal structures.}
Let $\baseCat$ be a small category which will serve as the ``base'' category over which we will parametrise. Following the language of parametrised category theory as introduced by Barwick, Dotto, Glasman, Nardin and Shah in \cite{bdgns16}, we will write $\cat^{\baseCat}$ for the category $\func(\baseCat\op,\cat)$ of \textit{$\baseCat$-categories} and objects therein will be denoted with an underline. In his thesis \cite{nardinThesis}, Nardin introduced the notions of \textit{$\baseCat$-distributivity} and \textit{$\baseCat$-normed categories} for so-called \textit{atomic orbital} base categories $\baseCat$ (cf. \Cref{subsection:atomic_orbital}). By setting $\baseCat$ to be the orbit category $\Orb_G$ of a finite group $G$, these are meant to capture the distributive multiplicative norms of Hill-Hopkins-Ravenel in a  fully coherent and natively $\infty$-categorical manner. 

Unfortunately, atomic orbital categories are not rich enough to encapsulate our motivating example, the equivariant homotopy theory of compact Lie groups.
Nevertheless, a straightforward extension of Nardin's theory suffices for our purposes, namely the framework of a base category $\baseCat$ equipped with a distinguished \textit{atomic orbital subcategory $\NN\subset \baseCat$} (see \Cref{subsection:atomic_orbital}). The main example to keep in mind is that of the orbit category $\Orb_G$ for a compact Lie group $G$ and the wide subcategory $\finind\subset \Orb_G$ of finite-index injections;
further atomic orbital subcategories  $\NN \subset \Orb_G$ give a notion of \textit{incomplete systems of finite-index transfers}, specialising to the program of Blumberg-Hill et. al. \cite{Blumberg-op} when $G$ is finite.

Associated to the data $\NN\subset \baseCat$ are then the notions of \textit{$\NN$-normed structures} and \textit{$\NN$-distributivity}, generalising Nardin's aformentioned framework.
For the former notion, we say that a map of disjoint unions of elements $W = \coprod_i W_i \rightarrow \coprod_j V_j = V$ in $\baseCat$ is \textit{$\NN$-admissible} if the constituent maps $W_i \rightarrow V_j$ lie in $\NN$;
$\NN$-normed $\baseCat$-categories $\CC$ have ``$w$-tensor functors'' 
\[
    w_{\otimes}\colon \CC^W \coloneqq \prod_{i} \CC^{W_i} \longrightarrow \prod_i\CC^{V_i} = \CC^V.
\]
For $w$ a fold map, these give traditional tensor functors, and for $w$ a map lying directly in $\NN$, these are a categorical context intended for finite-index Hill-Hopkins-Ravenel's norms.

Let us now explain the notion of distributivity.  Vaguely speaking, distributivity is a general term which refers to how certain tensor products ``distribute'' across arbitrary colimits, categorifying  the axiom that multiplications distribute  sums in the definition of a classical ring.

Now, as amply highlighted in \cite[\textsection A.3.3 and A.3.4]{HHR}, the multiplicative norms 
-- being an equivariant version of tensor products -- 
also satisfy a certain \textit{equivariant} distributivity of fundamental importance. 
However, this distributivity is much more complicated to even state. 
Notwithstanding, Nardin has found a compact formulation which directly generalises to our setting. 

As our first contribution, we introduce the notion of \textit{categories of cocones} to organise this theory in an efficient and conceptual manner;
the specialisation of this to non-parametrised  higher category theory already presents a more conceptual framework for distributivity and the symmetric monoidal structure on presentable categories than the standard approach.

\introheader{Categories of cocones.}
We define a cocone in a $\baseCat$-category $\CC$ simply as a functor $\udl{K}^{\triangleright} \to \CC$.
Let $\Catcocone$ be the $\baseCat$-category of  $\baseCat$-categories whose objects are a $\baseCat$-category $\CC$ together with a $\baseCat$-class $\udl{\R}$ of cocones in $\CC$ and morphisms are functors of categories which send cocones to cocones. 
Now, $\Catcocone$ comes with an $\NN$-normed structure: given an $\NN$-admissible map $w \colon W \to V$, a $W$-category (i.e. $\baseCat_{/W}$-category) $\CC$, and a cocone $\udl{K}^{\triangleright} \to \CC$, we obtain a cocone in $w_\ast \CC$: 
\[ (w_\ast \udl{K})^{\triangleright} \longrightarrow w_\ast(\udl{K}^{\triangleright}) \longrightarrow w_\ast \CC; \]
here, the first arrow adjoins the appropriate cocone point to the inclusion $w_* \iota\colon w_*\udl{K} \hookrightarrow w_* \prn{\udl{K}^{\triangleright}}$.
Applying this procedure to every cocone in a class of cocones $\udl{\R}$ in $\CC$ yields a class of cocones $w_\ast \udl{\R}$ in $w_\ast \CC$.
In the special case of non-parametrised category theory $(\TT=\ast)$ and $w$ being the fold map, this specialises to the fact that, given cocones $K^{\triangleright} \to \mathcal C$ and $L^{\triangleright} \to \mathcal D$, we obtain a cocone in the product $(K \times L)^{\triangleright} \hookrightarrow K^{\triangleright} \times L^{\triangleright} \to \mathcal C \times \mathcal D$.

Let $\KK \subset \GCat$ be a $\baseCat$-subcategory of indexing diagrams, closed under $\NN$-indexed products. 
The $\baseCat$-category $\GCatcoco{\KK}$ of $\KK$-cocomplete $\baseCat$-categories and $\KK$-colimit preserving $\baseCat$-functors is a full subcategory of $\Catcocone$: the inclusion $\GCatcoco{\KK} \subset \Catcocone$ equips a $\KK$-cocomplete category $\CC$ with the class of $\KK$-colimit cocones $\udl{\RR}_{\KK}(\CC)$, i.e. the classes of colimit cocones $\udl{K}^\triangleright \to \CC$ with $\udl{K} \in \KK$. 

In this framework it is easy to state what a $\KK$-distributive functor is: given $\NN$-admissible $w \colon W \to V$, $\KK$-cocomplete categories $\CC$ in the context of $W$ and $\DD$ in the context of $V$, a $V$-functor $w_\ast \CC \to \DD$ is said to be $\KK$-distributive if it sends the class of cocones $w_\ast \udl{\RR}_{\KK}(\CC)$ to $\udl{\RR}_{\KK}(\DD)$. 
In the special case $\baseCat = \ast$ and $w$ the fold map, this specialises to the classical notion of functors commuting with $\mathcal K$-colimits in each variable separately. 

With this in hand, we are able to formulate and construct the $\NN$-normed structure on $\GCatcoco{\KK}$: we show that the inclusion $\GCatcoco{\KK} \subset \Catcocone$ admits a partial left adjoint. 
In this work, we show that this partial left adjoint is compatible with the $\NN$-normed structure, inducing an $\NN$-normed structure. 

\begin{alphThm}[Precise and full version in \Cref{thm:catK_symMon}]\label{alphThm:CATK_Symmon}
    Let $\KK \subset \GCat$ be a collection of categories, closed under $\NN$-indexed products. 
    The inclusion $\GCatcoco{\KK} \subset \Catcocone$ of the category of $\KK$-cocomplete categories into the category of cocones factors through an $\NN$-normed subcategory $\Catcoconeproj(\KK) \subset \Catcocone$ such that the inclusion $\GCatcoco{\KK} \subset \Catcoconeproj(\KK)$ admits an $\NN$-norm compatible left adjoint, inducing a unique $\NN$-normed structure on $\GCatcoco{\KK}$. The norm functors in $\GCatcoco{\KK}$ corepresent $\KK$-distributive functors.

    Moreover, the norm functors in $\GCatcoco{\KK}$ are themselves distributive.
\end{alphThm}

This makes the existence of the $\NN$-normed structure on $\GCatcoco{\KK}$, whose $\NN$-norms corepresent $\KK$-distributive functors, a property of its inclusion into $\Catcocone$.
The left adjoint can be thought of as ``$\KK$-cocompletion with relations;'' our crucial technical contribution is to establish a \textit{distributive} universal property of this relative cocompletion.

As in the non-parametrised situation, some work is still needed to obtain from this an $\NN$-normed structure on $\baseCat$-presentable categories: we perform a set-theoretic argument to ensure that the norm of a presentable category is still presentable.
For arguments in algebraic $K$-theory, we must furthermore know that these norm functors on $\udl{\Pr}^L$ are themselves distributive, but this is significantly different from the non-parametrised setting;
whereas distributivity of tensor products is implied by existence of a tensor-hom adjunction, no such argument works for norms.
To establish distributivity of $\udl{\Pr}^L$ norms, we instead use $\wGCat$ to embed it \textit{nicely} into $\wGCat\prn{\allcolim}$, establishing the following.
\begin{alphThm}[{\Cref{thm:prL-colims}}]\label{alphThm:symmon_PrL}
  The presentable $\baseCat$-categories $\udl{\presentable}^L \subset \wGCat(\udl{\allcolim})$ form an $
  \NN$-normed subcategory with distributive norm functors. 
\end{alphThm}

This distributivity is fundamental to the work of the first author, cf. \cite{kaifNoncommMotives,hilmanRamzi}, but we shall also already see an application in the next point. 

\introheader{Idempotents and modes.} 
With the distributively multiplicative theory of presentably parametrised categories set up, we would like to follow the strategy of \cite{lur17} in constructing the $\NN$-normed structure on genuine equivariant spectra by exhibiting it as an appropriate form of idempotent algebra, or \textit{mode}, following the terminology of \cite{csy20}. 

To this end, we extend the theory of idempotence sketched by Nardin in \cite{nardinThesis} to our setting in \Cref{subsection:modes};
that is, we establish theories of \textit{$\NN$-normed smashing $\TT$-localisations} and \textit{idempotent $\NN$-normed algebras} and show them to be in correspondence, generalising the non-parametrised case.
Moreover, in \Cref{prop:idem-Nalg} we show that idempotence of $\NN$-normed algebras is determined by its underlying $\EEE{0}$-algebra;
that is, to construct an idempotent $\NN$-normed algebra, we may first construct an idempotent $\EEE{0}$-algebra, then exhibit it as the unit of an $\NN$-normed algebra.

We then present pointed parametrised anim\ae{} as a relative cocompletion of $\Delta^1$, and we use the combinatorics of the related cocones to prove that the $\TT$-category of pointed anim\ae{} $\GBSpaces$ is an $\NN$-mode, hence constructing a $\NN$-normed structure on that category.

Next, recall that Robalo \cite{robaloKTheoryBridge} has developed a theory of formal inversions where we can universally invert a collection of objects in a symmetric monoidal category, and under some assumptions on the object being inverted, this universal construction may be computed in terms of a filtered colimit in presentable categories. 
This theory has been extended to the parametrised (non-normed) context by Cnossen in \cite{Cnossen2023}. In particular, this allows us to consider the formal inversion $\MM[\invObj^{-1}]$ for an arbitrary small collection of objects $\invObj\subseteq \MM$. 
Under some hypotheses on $\invObj$, we prove the following:

\begin{alphThm}[{\Cref{thm:invert_object_idempotent}}]\label{alphThm:invert_object_idempotent}
    Suppose $\invObj \subset \MM$ is an $\NN$-good (cf. \Cref{defn:multiplicative_invertibles}) full $\baseCat$-subcategory of an $\NN$-mode $\MM\in\presentable^L_{\baseCat}$. 
  Then the $\invObj$-inversion map $ \MM \rightarrow \MM\brk{\invObj^{-1}}$ 
  exhibits the codomain as an $\NN$-mode
  such that an $\MM$-module $\CC$ is an $\MM\brk{\invObj^{-1}}$-module if and only if  $\invObj$ acts invertibly on $\CC$.
\end{alphThm}

Before explaining our application in the equivariant setting, we briefly comment on the proof of this: the hypothesis of $\NN$-goodness ensures that the formal inversion $\MM[\invObj^{-1}]$ may be written as a filtered colimit. Exploiting the distributivity from \Cref{alphThm:symmon_PrL}, we thus write the norm $w_{\otimes}(\MM[\invObj^{-1}])$ as a certain filtered colimit, and this allows us to show that it is an $\NN$-mode given that $\MM$ is. 

\introheader{Finite-index norms for compact Lie groups.} 
As application of all the abstract theory, we  turn to the main motivating case of equivariant homotopy theory for compact Lie groups, which necessitated atomic orbital subcategories in the first place. 
In this case, we consider $\baseCat=\Orb_G$, the orbit category for a compact Lie group $G$. We then have:

\begin{alphThm}[{\Cref{thm:compact_lie_norms}}]\label{alphThm:compact_lie_norms}
    Let $G$ be a compact Lie group and $\finind \subset \Orb_G$ the quotient maps $G/K \rightarrow G/H$ along finite-index subgroup inclusions. Then $\myuline{\spectra}_G$ is $\finind$-idempotent. Hence it admits a unique $\finind$-normed structure compatible with the map $\Sigma^{\infty}_+\colon \udl{\Spaces}_G\rightarrow \myuline{\spectra}_G$  whose norm functors are distributive.
\end{alphThm}

Unsurprisingly, this will be a consequence of \Cref{alphThm:invert_object_idempotent}, where we formally invert the so-called compact homotopy spheres. 

Lastly, having worked out the theory for presentable categories, we then treat the case of small categories. We write $\udl{\presentable}^L_{G,\aleph_0} \subset \udl{\presentable}^L_G$ for the nonfull $G$-subcategory of compactly generated presentable $G$-categories and left adjoint $G$-functors which preserve compact objects, and we write $\udl{\cat}\irex_G $ (resp. $\udl{\cat}\perfect_G$) for the $G$-categories of small, idempotent-complete, and $G$–finite-cocomplete (resp. and additionally $G$-stable) categories. In this direction, we offer the following result:

\begin{alphThm}[{\Cref{prop:catperf_G}}] \label{alphThm:GSymmon_small_cats}
    Let $G$ be a compact Lie group. Then $\udl{\presentable}^L_{G,\aleph_0}$ inherits the distributive $\finind$-normed structure from $\udl{\presentable}^L_G$.
    Moreover, idempotent completion $\prn{-}^{\natural}\colon \udl{\cat}^{\mathrm{rex}}_G \rightarrow \udl{\cat}\irex_G$ is an $\finind$-norm compatible localisation, and under the resulting normed structure on $\udl{\cat}\irex_G$, the equivalence $\udl{\ind}\colon \udl{\cat}\irex_G\rightleftharpoons \udl{\presentable}^L_{G,\aleph_0}\cocolon (-)^{\aleph_0}$ enhances to an $\NN$-normed one.

    Additionally, the Bousfield localisation $\myuline{\spectra}_G^{\aleph_0} \otimes-\colon \udl{\cat}\irex_G\rightleftharpoons \udl{\cat}\perfect_G \cocolon \mathrm{incl}$ refines to an $\NN$-smashing localisation. In particular, the equivalence $\udl{\ind}\colon \udl{\cat}\perfect_G\rightleftharpoons \udl{\presentable}^L_{G,\aleph_0,\mathrm{st}}\cocolon (-)^{\aleph_0}$ enhances to an $\NN$-normed one.
\end{alphThm}

The key non-formal input  is the equivariant finiteness of equivariant cubes from \cite[Prop. C.0.5]{hilmanRamzi}, and \Cref{alphThm:GSymmon_small_cats} forms the basic substrate for the first author's work on obtaining norms on algebraic K-theory for finite groups in \textit{loc. cit.} and in \cite{kaifNoncommMotives}.

\subsection*{Relation to other work}
While one goal of this paper is to extend the constructions of normed structures on cocomplete and presentable categories from the atomic orbital setting to the one allowing for atomic orbital subcategories, we also believe that the constructions in the atomic orbital category currently available had gaps which are filled in this paper.

Indeed, a construction of the normed structure on presentable categories can be found in the unpublished thesis of Nardin \cite{nardinThesis}, while a normed structure on cocomplete categories is discussed in Nardin and Shah's unpublished preprint \cite{NS}.
Let us briefly discuss why we believe that both writeups were not fully sufficient for the construction of normed structures: 
\begin{itemize}
    \item The article \cite{NS} follows the same Lurie-inspired strategy we follow in this article. However, we believe that there is a gap in \cite[Cor.~6.0.9]{NS} (proved in our seeting as \Cref{prop:distributivity_inverted_maps_presheaves}: in the proof they state that \enquote{by definition $g_\ast\udl{\mathbb{P}}_{f_\ast\mathcal{All}}(f_\ast\CC)$ is a localisation of $g_\ast \udl{\mathbb{P}}_{\mathcal{T}^{/V}}(f_\ast\CC)$ at the relevant class of morphisms}. We could not verify that this is true by definition: the relevant categories were defined as the class of local objects of a class of morphism, i.e. as a \textit{Bousfield localisation}. Any Bousfield localisation is the Dwyer-Kan localisation at the \textit{strong saturation} of the class of morphisms one started with. This issue of saturation is not mentioned at all in the paper. It needs to be explained why saturation behaves well with respect to indexed products. This is verified in our \Cref{cor:distributivity_bousfield_localisations}.
    While this might sound like a very minor problem, it is actually the technical heart of the paper: it is the one step in the construction where non-parametrised constructions and proofs do not  carry over to the parametrised setting, leading to a delicate treatment of the interaction of distributive functors and saturated classes of morphisms (\Cref{lem:saturated-class}).
    \item In his thesis \cite{nardinThesis}, Nardin presents a construction of the normed structure on presentable categories. He does pay attention to the issue mentioned above by presenting an analog of \Cref{lem:saturated-class} (\cite[Lem.~3.22]{nardinThesis}). Our proof closely follows his approach. However, there seem to be some gaps in his proof of \textit{loc. cit.}: it is wrongly claimed that a map is right marked anodyne. Moreover, the deduction of the analogue of \Cref{cor:distributivity_bousfield_localisations} (\cite[Prop.~3.23]{nardinThesis}), stating that Bousfield localisations behave well with respect to distributive functors, has several gaps in the proof, mostly not explaining why certain functors are distributive. We also filled the gaps in this paper. 

    Nardin also claims to provide a normed structure on a parametrised category of spectra, therein defined universally with respect to fibrewise-stability and parametrised semi-additivity. 
    We could not verify his proof. Concretely, we could not verify that \enquote{the explicit description of this Bousfield localisation tells us that this is
just the category of distributive functors $\prod_f \mathbf{Top}_\ast^{\mathrm{fin},U} \to \mathbf{Sp}^{U^\prime}$, that is equivalent to $\mathbf{Sp}^U$ by 3.19.} On one hand, \cite[Prop.~3.19]{nardinThesis} does not seem to apply here. It only makes sense if we consider a distributive universal property of \textit{finite} cocompletion, which requires such an extension of \cite[Prop.~3.19]{nardinThesis}. On the other, we were not even able to confirm the first part, that is the explicit description of the Bousfield localisation does not seem to enable us to make the conclusion they claim;
for instance, without compatibility of cartesian norms with finiteness, it is unclear that the local classes even stay within the realm of \textit{finite colimits}.
\end{itemize}
In the general context of an atomic orbital category $\baseCat$, our paper provides us with a construction of a stable parametrised category of spectra $\myuline{\spectra}^{-1}$ which comes equipped with norms: We can simply invert norms of $S^1$ in the category of pointed anim\ae{} $\GBSpaces$. 
Under some conditions on $\baseCat$, \cite[Thm. B]{kaifSil} implies that there is a unique colimit-preserving $\NN$-normed $\baseCat$-functor from Nardin's category of genuine $\baseCat$-spectra $\myuline{\spectra}\rightarrow \myuline{\spectra}^{-1}$. 
This is an equivalence in the finite-group equivariant situation, but we do not know if this is an equivalence in general or whether Nardin's category admits norms, although our expectation is that this is not so. 

When $G$ is  finite group, the $G$-normed structures on pointed $G$-\animae{} $\udl{\Spaces}_*$ and genuine $G$-spectra $\myuline{\spectra}$ have also previously been constructed by  \cite{cnossennorms} with different methods: they constructed these by formally extending the one on pointed finite $G$-sets, and then argued that these structures are unique in \cite[Thm. 5.4.10]{cnossennorms} by showing that these are idempotent (in the usual sense) in $G$-presentable categories. In contrast, our approach follows the general strategy sketched by Nardin \cite{nardinThesis}. Namely, we show that $\udl{\Spaces}_*$ and $\myuline{\spectra}$  have the stronger property of being $G$-idempotent, and then \textit{deduce} that there must exist a unique distributive $G$-normed structure on them.

Finally, in \cite{mwpresentable} Martini and Wolf construct a symmetric monoidal structure on parametrised presentable categories, which recovers the underlying symmetric monoidal structure of our normed structure by \Cref{rmk:mw-vs-us}. Their treatment works in the generality of internal higher algebra over a topos, but does not include nontrivial norms, whose construction is one of the main technical results of this paper.

Some of the results herein have also  appeared within the third author's thesis \cite{nat-thesis}, sometimes via different proofs.

\subsection*{Organisation of the paper}
We begin in \Cref{sec:preliminaries} with recollections and preliminary results which do not depend on distributivity.
Namely, after recollections and conventions in \Cref{subsection:recollections}, in \Cref{subsection:atomic_orbital} we establish the language of \textit{atomic orbital subcategories and admissible maps}, on which the rest of the paper is founded. 
Following this, in \Cref{subsec:parametrised_higher_algebra} we set up a modicum of parametrised higher algebra in our setting. In \Cref{subsection:bousfield-localisations}, we discuss parametrised Bousfield localisations.

In \Cref{section:distributivity}, we establish the theory of distributivity in the generality we require. We first provide the definition of the category of cocones and use it to define the notion of distributivity in \Cref{subsec:definition_distributivity}. In \Cref{subsec:distributivity_cartesian}, we provide two examples of distributive categories: we show that the normed structure on the category of cocones is distributive and that the cartesian $\NN$-normed structure on any distributively cartesian  category is distributive. Next, we record some technical lemmas in \Cref{subsection:cubicalDecompositions} which will serve as the combinatorial basis for our analysis of distributivity throughout the article. We then generalise Nardin's distributive universal property of presheaf categories in \Cref{subsec:distributivity_free_cocompletions}.  Following that, we come  to \Cref{subsec:distributivity_general_cocompletions} where we generalise the universal property from \Cref{subsec:distributivity_free_cocompletions} to construct \textit{free cocompletions for an arbitrary class, modulo relations}.  With this in hand, we prove \Cref{alphThm:CATK_Symmon} in \Cref{subsec:sym_mon_structure_cocomplete}, and  \Cref{alphThm:symmon_PrL} in \Cref{subsec:norms_presentable}.

With all these foundations in place, we turn to our main applications in  \Cref{section:presentable_categories}. To begin with, we record the theory of idempotent normed algebras in \Cref{subsection:modes}.  In \Cref{subsec:pointed_mode}, we prove that pointed anima is an idempotent normed algebra; and in \Cref{subsec:formal_inversions}, we prove \Cref{alphThm:invert_object_idempotent}. In the last \Cref{subsection:the_equivariant_case}, we combine all the foregoing elements  to prove \Cref{alphThm:compact_lie_norms,alphThm:GSymmon_small_cats}.

To close the article, we complement the body of the document with four appendices. We prove the monoidal parametrised straightening-unstraightening equivalence in \Cref{sec:straightening_unstraightening}; we then show that our notion of distributivity agrees with that of Elmanto--Haugseng in \Cref{sec:unicity}; in \Cref{sec:size}, we deal with some set-theoretic matters to support our arguments about presentable categories; and in \Cref{appendix:point-set_norms}, we show that norms preserve equivariant finiteness and finite equivariant \animae, which is needed for our conclusions in \Cref{subsection:the_equivariant_case}.

\subsection*{Acknowledgments} 
We are grateful to  Emma Brink and Tobias Lenz for helpful conversations about model categories and for sharing the perspectives of their work on some of the points treated here. We would like to thank Bastiaan Cnossen for helpful discussions regarding parametrised Bousfield localisations. KH and BJ are supported by the European Research Council (ERC) under Horizon Europe (BorSym, ID: 101163408), and thank the University of Bonn for its conducive working environment. KH is  supported additionally by the European Research Council (ERC) under Horizon Europe (GeoCats, ID: 101042990), and he would also like to thank the Isaac Newton Institute for Mathematical Sciences, Cambridge, under the programme ‘Equivariant homotopy theory in context’, during which parts of this work began. BJ received additional support from the Danish National Research Foundation (DNRF)
through the Copenhagen Center for Geometry and Topology (DNRF151).  KH and NS thank the Max Planck Institute for Mathematics in Bonn for support and hospitality.

\subsection*{AI disclosure}
No generative AI was used in the conception, generation, and writing of this paper.

\section{Preliminaries}\label{sec:preliminaries}

\subsection{Recollections and conventions}\label{subsection:recollections}
\begin{itemize}
    \item We work using the theory of $\infty$-categories, which we will simply refer to as \textit{categories} as developed by Lurie in \cite{lur09} and \cite{lur17}.
    Following Clausen and Scholze, we refer to groupoids in this setting as \textit{\animae{}} (formerly known as spaces).\footnote{Past and future work of the third author continues to refer to these as spaces.}
    \item We fix three nested Grothendieck universes. Categories defined in the smallest universe are called \textit{small}, in the middle  \textit{large} (sometimes omitted), and the ones in the third universe are called \textit{very large}.
    \item We denote the large category of small categories by $\Cat$ and the large category of small anima by $\Spaces$. Moreover, we write $\LargeCat$ for the very large category of large categories.  
    \item We will moreover use the theory of \textit{parametrised higher category theory} as developed by Barwick, Dotto, Glasman, Nardin and Shah \cite{bdgns16}, \cite{Sha23}, \cite{shahPaperII} and \cite{NS}.
    When relevant, we use the synonymous theory of \textit{internal higher category theory} of \cite{MartiniWolf2024,mwpresentable}, internal to presheaf toposes.
    \item We will refer to monomorphisms of $\baseCat$-categories as \textit{$\baseCat$-subcategory inclusions};
    equivalently, these are fibrewise faithful and core-fully faithful $\baseCat$-functors, i.e. fibrewise replete subcategory inclusions into the homotopy category.
    We refer to the $\baseCat$-subcategory classified by the fibrewise-replete homotopy image of a $\baseCat$-functor simply as its \textit{image} and its fibrewise-essential image as its \textit{essential image}. A subcategory containing all equivalences will be called a \textit{wide subcategory}.
\end{itemize}

\subsection{Atomic orbital subcategories and admissible maps}\label{subsection:atomic_orbital}
\begin{definition}
  Given a small category $\TT$,
  the large category of small \mdef{$\TT$-categories} is given by the diagram category
  \[ \Cat[\TT] \coloneqq \Fun(\TT\op,\Cat) \]
  and the category of \mdef{$\TT$-\animae{}} the presheaf category
  \[
    \Spaces[\TT] \coloneqq \Fun(\TT\op,\Spaces) \,.
  \]
  \end{definition}
Given a $\TT$-category $\CC \in  \Cat[\TT]$ and an element $V \in \TT$, we write $\CC^V$ for the value of $\CC$ at $V$.
Given a map $f\colon U \rightarrow V$ in $\TT$, we will write $f^*$ for the corresponding functor $\CC^V \rightarrow \CC^U$.
Objects of $\CC^V$ will be called \mdef{$V$-objects of $\CC$}.
Moreover, objects of
\[
    \CC^{\baseCat} \coloneqq \Fun^{\mathrm{cocart}}_{/\baseCat{\op}} \prn{\baseCat{\op}, \int \CC} \simeq \lim_{V \in \baseCat{\op}} \CC^V
\]
will be referred to as \mdef{$\baseCat$-objects of $\CC$}.
\begin{remark}
    We also consider the \mdef{large $\TT$-categories}, i.e. $\LargeCat[\baseCat]\coloneqq\func(\TT\op, \LargeCat)$.
\end{remark}

\begin{example}\label{ex:coeff}
  Given a category $\C$, we define the \mdef{$\TT$-category of coefficient systems} via diagram categories indexed by slices:
  \[
    \cofree(\C) \colon \baseCat\op \xrightarrow{\;\;\;\;\; (\TT_{/\bullet})\op \;\;\;\;\;} \Cat[\mathrm{op}] \xrightarrow{\;\;\;\;\; \Fun\prn{-,\C} \;\;\;\;\;} \Cat.
  \]
  Informally, the $V$-objects of $\cofree(\C)$ are given by the ``$V$-equivariant coefficient systems'' $(\TT_{/V})\op \rightarrow \C$, with restrictions functors given by precomposition along the ``induction functor'' $f_!\colon (\TT_{/U})\op \rightarrow (\TT_{/V})\op$ induced by a map $f\colon U \rightarrow V$ in $\baseCat$.

  Correspondingly, we write $\GCat \coloneqq \cofree(\Cat)$ and $\GSpaces \coloneqq \cofree(\Spaces)$ for the $\TT$-categories of (small) $\TT$-categories and $\TT$-\animae{}.
\end{example}

\Cref{ex:coeff} has generally been referred to as the cofree parametrisation $\underline{\mathrm{Cofree}}(\C)$ in the first author's work 
\cite{kaifPresentable,kaifNoncommMotives,hilmanRamzi} and as the category of $\TT$-objects $\CC_{\baseCat}$ in Nardin and Shah's work \cite{nardinThesis,NS};
we follow the third author's convention in order to highlight its place in equivariant algebraic topology.

\begin{definition}
    We refer to the objects in the image of the Yoneda embedding $\TT \to \Spaces[\TT]$ as \mdef{orbits}.
  The category of \mdef{orbit-finite $\TT$-\animae{}} is defined via finite-coproduct-cocompletion
  \[
    \FinTN[\TT]{} \coloneqq \cbr{\coprod_{i = 0}^n V_i \;\; \middle| \;\; V_i \in \TT, n \in \mathbb{N}} \subset \Spaces[\baseCat].
  \]
\end{definition}
    By the universal property of the finite-coproduct-cocompletion, a $\TT$-category uniquely extends to a finite-product-preserving functor 
    $\left( \FinTN[\TT]{} \right)\op \longrightarrow \Cat$.
    Given $W \in \FinTN[\TT]{}$, we will thus write $\CC^W$ for its value on $W$.
    
    Given an orbit-finite $\TT$-anima $V$, we will call a $ \overslice{\TT}{V}$-category (sliced in $\TT$-\animae{}) simply a $V$-category and write $\Cat[V]$ for the category of $V$-categories. 
   Given a map $f \colon U \to V$ of orbit-finite $\TT$-anima, restriction along $\left(\overslice{\TT}{U} \right)\op \longrightarrow \left(\overslice{\TT}{V} \right)\op$ gives rise to a functor
    \[ f^\ast \colon \Cat[V] \longrightarrow \Cat[U] \]
    which admits a right adjoint 
    \[ f_\ast \colon \Cat[U] \longrightarrow \Cat[V] \,. \]

    Given a $\TT$-category $\CC$, we write $V^\ast \CC$ for the $V$-category obtained by restricting along the forgetful map $\overslice{\TT}{V} \to \TT$.

    Barwick, Dotto, Glasman, Nardin and Shah \cite{bdgns16}, \cite{Sha23}, \cite{shahPaperII} and \cite{NS}, as well as Martini and Wolf \cite{mwpresentable,MartiniWolf2024} have developed the theory of parametrised functor categories, (co)limits, Kan extensions which we will use freely in this paper. 
    
  Given a wide subcategory $\NN \subset \TT$, we define the wide subcategory $\FinTN[\TT]{\NN} \subset \FinTN[\TT]{}$ spanned by the morphisms of the form 
  \[
      \cbr{\coprod_{i \in I} \prn{\coprod_{j \in J_i} U_j \rightarrow V_i} \mid \forall i, j \in J_i, \;\;\;\; U_j \rightarrow V_i \in \NN};
  \]
  that is, all maps of orbit-finite $\TT$-\animae{} are induced under disjunctiveness and the universal property of coproducts by maps between orbits, and $\FinTN[\TT]{\NN}$ collects those whose corresponding orbit maps lie in $\NN$.
\begin{definition}
  We say that a morphism of $\TT$-\animae{} $f\colon X \rightarrow Y$ is \mdef{$\NN$-admissible} if, for each fibre square
  \[
      \begin{tikzcd}[ampersand replacement=\&]
        f^{-1}(V) \& Y \\
        V \& X
        \arrow[from=1-1, to=1-2]
        \arrow["{f_V}"', from=1-1, to=2-1]
        \arrow["\lrcorner"{anchor=center, pos=0.125}, draw=none, from=1-1, to=2-2]
        \arrow["f"', from=1-2, to=2-2]
        \arrow[from=2-1, to=2-2, "\iota"]
      \end{tikzcd}
  \]
  such that $V$ is an orbit-finite $\TT$-anima, $f_V$ lies in $\FinTN[\TT]{\NN}$.

  A wide subcategory $\NN \subset \TT$ is said to be an \mdef{orbital subcategory} if $\FinTN[\TT]{\NN} \subset \FinTN[\TT]{}$ consists precisely of the $\NN$-admissible maps;
  $\NN$ is additionally \mdef{atomic} if all maps in $\NN$ which admit a section in $\TT$ are equivalences.
\end{definition}
\begin{remark}\label{rmk:N-adm-F}
  Choosing $\iota$ to be the identity, it follows that all $\NN$-admissible maps between orbit-finite $\TT$-\animae{} lie in $\FinTN[\TT]{\NN}$ regardless of orbitality;
  in particular, our definition recovers \cllstab{Def.}{4.2.2} by \cllstab{Lem.}{4.3.2}.
\end{remark}
\begin{example}\label{example:orb_G}
    For $G$ a compact Lie group, we can choose $\TT$ to be the orbit category $\Orb_{G}$ of $G$ -- the category of transitive $G$-\animae{} -- and $\NN\coloneqq \finind$ to be the class of maps with finite fibres. 
    It follows by the double coset formula that this is an orbital subcategory;
    all endomorphisms are isomorphisms in $\Orb_{G}$, so it is moreover atomic.
\end{example}
\begin{example}\label{ex:transf}
    In the case that $\baseCat$ is already atomic orbital, \cite[Rmk.~I.1.99]{nat-thesis} shows that atomic orbital subcategories $\NN \subset \baseCat$ specialise to Blumbereg--Hill-Rubin's \textit{transfer systems} \cite{Blumberg-op,Rubin};
    in particular, under \cite[Thm.~I.4.15,Prop.~I.1.97]{nat-thesis}, they enumerate a poset of \textit{sub-terminal  $\baseCat$-operads admitting a map from $\EEE{\infty}$}.
\end{example}
\begin{example}
    Similarly to \Cref{ex:transf}, atomic orbital subcategories of the \textit{Global indexing category} for compact Lie groups have been studied under the name \textit{global transfer systems} \cite[Def.~4.3]{Barrero_2023,Barrero_ab}.  
\end{example}

We specify orbitality precisely for the following reason.
\begin{remark}\label{lem:fibrewise_product}
    Let $w\colon W \rightarrow V$ be an $\NN$-admissible map between orbit-finite $\TT$-\animae{}, $u\colon U \rightarrow V$ an arbitrary map between orbit-finite $\TT$-\animae{}, and $\udl{\sC}$ a $W$-category. Then $(w_*\udl{\sC})^U$ is a finite product of various evaluations of $\udl{\sC}$. More precisely, we have a pullback 
\begin{equation}\label{eqn:standard_basecat_pullback}
\begin{tikzcd}
    Z=\coprod_{i=1}^nZ_i \ar[rr,"\overline{w}=\sqcup_i\overline{w}_i"]\dar["\overline{u}=\sqcup_i\overline{u}_i"']\ar[drr,phantom,"\lrcorner"]&&U\dar["u"]\\
    W \ar[rr,"w"] && V    
\end{tikzcd}
\end{equation} in $\FinTN[\TT]{}$, where $\overline{w}$ and each $\overline{w}_i$ is $\NN$-admissible and each $Z_i$ is an orbit.
    
    Observe that the top left corner in the pullback is a finite coproduct, and so we get $(u^*w_*\udl{\sC})^U\simeq (\overline{w}_*\overline{u}^*\udl{\sC})^U\simeq \prod_{i=1}^n\udl{\sC}^{Z_i}\in\cat$, as asserted.
\end{remark}

\subsection{Some parametrised higher algebra}\label{subsec:parametrised_higher_algebra}
Given $\NN \subset \TT$ an atomic orbital subcategory, the wide subcategory inclusion $\FinTN[\TT]{\NN} \subset \FinTN[\TT]{}$ is easily seen to be a span pair in the sense of \cite[Def.~2.2.1]{elmantoHaugseng}. We write
\[ 
\SFTN{\TT}{\NN} \coloneqq \Span_{\FinTN[\TT]{\NN}}\prn{\FinTN[\TT]{}}
\] 
for the associated span category. 
Informally speaking, $\SFTN{\TT}{\NN}$ has the same objects as $\FinTN[\TT]{}$, a map $U \rightarrow V$ is given by a $\FinTN[\TT]{}$-span
\[ \begin{tikzcd}[row sep=tiny]
    & W \ar[dr,"w"] \ar[dl,swap,"f"] \\
    U & & V
\end{tikzcd} \]
such that $w$ is $\NN$-admissible, and 
composition uses pullbacks to create a new span from two spans. 
We call morphisms in which $f$ is an equivalence \mdef{forwards morphisms} and morphisms in which $w$ is an equivalence \mdef{backwards morphisms}.

This admits finite products:
the product of two objects $U,V$ is given by the disjoint union $U \sqcup V$ interpreted in the backwards wide subcategory $(\FinTN[\TT]{})\op \subset \SFTN{\TT}{\NN}$.
\begin{definition}
    A (small) \mdef{$\NN$-normed $\TT$-category} is a finite product-preserving functor $\SFTN{\TT}{\NN} \rightarrow \Cat.$
    We write
    \[
        \Cat[\TT]_{\NN-\otimes} \subseteq \Fun\prn{\SFTN{\TT}{\NN}, \Cat}
    \]
    for the full subcategory of $\NN$-normed $\TT$-categories.
\end{definition}
Restriction along the inclusion $\TT\op \hookrightarrow (\FinTN[\TT]{})\op \hookrightarrow \SFTN{\TT}{\NN}$ yields a functor
    \begin{equation}\label{eq:underlying-cat}
         \Cat[\TT]_{\NN-\otimes} \longrightarrow \Cat[\TT] 
    \end{equation}       
    which sends an $\NN$-normed $\TT$-category to its underlying $\TT$-category. 
     
    For $\CC$ an $\NN$-normed $\TT$-category and $w \colon W \rightarrow V$ an $\NN$-admissible map of orbit-finite $\TT$-\animae{}, we denote the covariant functoriality of $\CC$ in $w$ by 
    \[
        w_{\otimes} \colon \CC^W \longrightarrow \CC^V \,. 
      \]
    we call this functor the \mdef{$w$-norm functor} of $\CC$.
    Choosing $W$ to be the initial $\TT$-anima yields a functor $\ast \to \CC^V$, i.e. an object in $\CC^V$;
    we call this the \mdef{unit} of $\CC^V$ and usually denote by $\unit_V \in \CC$. 
    We will sometimes also omit the $V$;
    statements made with implicit $V$ will be true for any $V$.
    
    \begin{remark}\label{rmk:monoidality-of-fp}
      For all $V \in \TT$, there exists an essentially unique product-preserving functor $\iota_V\colon \Span(\mathbb{F}) \hookrightarrow \SFTN{\TT}{\NN}$ sending $* \mapsto V$;
      we acquire an evaluation functor 
      \[
      (-)^V = \iota_V^*\colon \Cat[\TT]_{\NN-\otimes} \rightarrow \Cat_\otimes \, ,
      \] endowing $\CC^V$ with a natural symmetric monoidal structure with unit $\unit_V \in \CC$ and binary multiplication $\nabla_{\otimes}$, for $\nabla\colon 2 \cdot V \rightarrow V$ the fold map.

      This structure is canonical;
      indeed, $\SFTN{\TT}{\NN}$ is semiadditive by \Cite[Lem.~C.3]{bh21}, so the universal property of commutative monoids yields an equivalence
      \begin{equation}\label{eq:cmon-lift}
        \Cat[\TT]_{\NN-\otimes} \simeq \Fun^\times\prn{\SFTN{\TT}{\NN}, \mathrm{CMon}(\Cat)},
      \end{equation}
      so that $w_{\otimes}$ and $w^*$ are naturally symmetric monoidal.
      The Eckmann-Hilton argument $\CMon(-) \simeq \CMon\CMon(-)$ (which is itself \eqref{eq:cmon-lift} for $\TT \simeq *$) identifies the value of \eqref{eq:cmon-lift} on $V$ with the above symmetric monoidal structure on $\CC^V$.
    \end{remark}

    Next, we would like to view this within a $\baseCat$-category, for which we need a few definitions.

\begin{definition}
   A $\TT$-category \mdef{admits $\NN$-indexed products} if its associated functor  $\prn{\FinTN[\TT]{}}\op \rightarrow \Cat$
    is \mdef{right $\FinTN[\TT]{\NN}$-adjointable};
    that is, for all cospans $u\colon U \rightarrow V \leftarrow W\cocolon w$ with $w$ $\NN$-admissible, $w^*$ admits a right adjoint $w_*$ whose Beck-Chevalley transformation
    \[
        u^* w_* \implies \overline{w}_* \overline{u}^*
    \]
    is an equivalence.
    We say that a functor $F \colon \CC \to \DD$ \mdef{preserves $\NN$-indexed products} if the Beck-Chevalley transformation
    \[ F w_\ast^{\CC} \Longrightarrow w_\ast^{\DD} F \colon \CC^V \longrightarrow \DD^W \] is an equivalence.
    We write $\Fun^{\NN-\times}(\CC,\DD) \subset \Fun(\CC,\DD)$ for the subcategory of $\NN$-indexed product preserving functors and we write $\Cat[\TT]_{\NN-\times} \subset \Cat[\TT]$ for the non-full subcategory of $\baseCat$-categories with $\NN$-indexed products and $\baseCat$-functors preserving them.  
\end{definition}
\begin{construction}
    Stefanich \cite[Thm.~1.1.1]{stefanich} provides the universal property for the $(\infty,2)$-category of spans which, after restricting to core categories, yields a map
    \[
        \Map_{\FinTN[\TT]{\NN}-\mathrm{Radj}}\prn{\FinTN[\TT]{}, \Cat} \longrightarrow \Map\prn{\SFTN{\TT}{\NN}, \Cat};
    \]
    since $\prn{\FinTN[\TT]{}}\op \hookrightarrow \SFTN{\TT}{\NN}$ creates products, this restricts to a map of \animae{}
    \[
        \prn{-}^\times\colon \prn{\Cat[\TT]_{\NN-\times}}^\simeq \longrightarrow \prn{\Cat[\TT]_{\NN-\otimes}}^\simeq.
    \]
\end{construction}
\begin{remark}
    \cite[Thm.~II.2.1]{nat-thesis} confirms that the above cartesian structure generalises \cite[Ex.~2.4.1]{NS}.
\end{remark}
\begin{construction}
    We now explain how to assemble $\Cat[V]_{\NN-\otimes}$ into a $\baseCat$-category. Note that there exists an equivalence $\FinTN[\TT_{/V}]{} \simeq \FinTN[\TT]{/V}$;
    thus postcomposition yields a functor
    \[
        \widetilde{\mathbb{F}}\colon  \baseCat^{\triangleright} 
        \longrightarrow \Cat,
    \]
    with cocone value $\FinTN[\TT]{}$ and $\widetilde{\mathbb{F}}(V) \simeq \mathbb{F}^V$.
    Writing $\FinTN[V]{\NN} \subset \FinTN[V]{}$ for the preimages of $\NN$-admissible maps, these are the $\NN_V$-admissible maps for the atomic orbital subcategory $\NN_V \rightarrow \baseCat_{/V}$ of preimages of $\NN$-maps (see \cite[Prop.~I.1.97,Prop.~I.1.113]{nat-thesis}). 
    In particular, applying this functoriality yields a $\baseCat$-category
    \[
        \GCat_{\NN-\otimes}\colon \baseCat\op \xrightarrow{\;\;\;\; \widetilde{\mathbb{F}}_{\NN} \subset \widetilde{\mathbb{F}} \;\;\;\;} \mathrm{SpanPair}\op \xrightarrow{\;\;\;\; \mathrm{Span}_{(-)}(-) \;\;\;\;}  \Cat[\mathrm{op}] \xrightarrow{\;\;\;\; \Fun^\times \prn{-,\Cat}  \;\;\;\;} \wCat,
    \]
    such that $\GCat[\TT]_{\NN-\otimes} \simeq \Cat[\TT]_{\NN-\otimes}$ and $\GCat[V]_{\NN-\otimes} \simeq \Cat[\TT_{/V}]_{\NN_V-\otimes}$.
\end{construction}

We now turn to the definition of normed algebras:
\begin{definition}
    Let $\CC , \DD \colon \SFTN{\TT}{\NN} \longrightarrow \Cat$ be $\NN$-normed $\TT$-categories. A \mdef{lax $\NN$-normed $\baseCat$-functor} is a functor of their cocartesian unstraightenings
    \[ \begin{tikzcd}[row sep=small]
        \TotFP \CC \ar[rr] \ar[dr] & & \TotFP \DD \ar[dl] \\
        & \SFTN{\TT}{\NN}
    \end{tikzcd} \] 
    which  preserves cocartesian lifts of backwards maps. 

    Note that $\NN$-normed $\baseCat$-categories have a \mdef{fibrewise opposite} involution 
    \[
        \prn{-}\op \colon \GCat_{\NN-\otimes} \longrightarrow \GCat_{\NN-\otimes}
    \]
    gotten by applying $(-)\op$ fibrewise.
    We define an \mdef{oplax $\NN$-normed $\baseCat$-functor} $\CC \rightarrow  \DD$ to simply be a lax $\NN$-normed functor $\baseCat$-functor $\CC{\op} \to \DD{\op}$.
    
    In the case where $\CC$ is the terminal $\NN$-normed $\TT$-category, we call a lax $\NN$-normed $\baseCat$-functor $\CC \rightarrow \DD$ a \mdef{normed algebra in $\DD$}; 
    that is, normed algebras are defined to be sections of the cocartesian unstraightening which are cocartesian on backwards morphisms.
\end{definition}
These themselves naturally fit within a $\baseCat$-category:
\begin{construction}
    Our next goal will be to construct a $\TT$-category of normed algebras.

    Let $\mathrm{WCat} \subset \Fun\prn{\Delta^1, \Cat}$ be the full subcategory of wide subcategory inclusions and let $\wCat_2$ be the very large category of large $(\infty,2)$-categories.
    Consider the functor sending a category $\CC$ together with a wide subcategory $\DD$ to the locally full subcategory of the $2$-category $\overslice{\wGCat}{\CC}$ spanned by the functors admitting cocartesian lifts over $\DD$ and functors preserving such lifts 
    \[
        \tCat[\mathrm{(-)-cocart}]_{/(-)}\colon \mathrm{WCat}\op \longrightarrow \wCat_{2}  \quad :: \quad (\DD\subset \CC) \mapsto \tCat[\mathrm{\DD-cocart}]_{/\CC}.
    \]
    Since the backwards wide subcategory is functorial in span pairs, we may thus define a parametrised $(\infty,2)$-category
    \[
        \tGCat[\mathrm{back-cocart}]_{/\widetilde \Span_{\NN}\prn{\FinTN[\TT]{}}}\colon \baseCat\op \xrightarrow{\;\; \widetilde{\mathbb{F}}_{\NN} \subset \widetilde{\mathbb{F}} \;\;} \mathrm{SpanPair}\op \xrightarrow{\;\; \mathrm{Span}_{(-)}(-) \;\;}  \mathrm{WCat}\op \xrightarrow{\;\; \tCat[\mathrm{(-)-cocart}]_{/(-)}  \;\;} \wCat_2.
    \]
    Given a pair of $\baseCat$-objects 
    \[
    \CC,\DD \in \prn{\tGCat[\mathrm{back-cocart}]_{/\widetilde \Span_{\NN}\prn{\FinTN[\TT]{}}}}^{\baseCat},
    \]
    a standard procedure using naturality of the Yoneda embedding then yields a $\baseCat$-category $\udl{\Fun}^{\mathrm{Back-cocart}}_{/\widetilde \Span_{\NN}\prn{\FinTN[\TT]{}}}\prn{\CC,\DD}$ such that
    \[
        \udl{\Fun}^{\mathrm{Back-cocart}}_{/\widetilde \Span_{\NN}\prn{\FinTN[\TT]{}}}\prn{\CC,\DD}^{V} \simeq \Fun^{\mathrm{Back-cocart}}_{/\SFTN{V}{\NN_V}}\prn{V^*\CC,V^*\DD}
    \]
    for all $V$, naturally in $\CC$, $\DD$, and $V$ (see \cite[\S~I.1.2]{nat-thesis}).
    Given $\DD$ an $\NN$-normed $\baseCat$-category, we use this to define the $\baseCat$-category of normed algebras
    \[
        \NAlg{\DD} \coloneqq \udl{\Fun}^{\mathrm{Back-cocart}}_{/\widetilde \Span_{\NN}\prn{\FinTN[\TT]{}}}\prn{\SFTN{\TT}{\NN},\DD},
    \]
   natural in lax $\NN$-normed $\TT$-functors, with $V$-objects the $\NN_V$-normed algebras in $V^* \DD$.
\end{construction}

\cite[Thm.~II.2.1]{nat-thesis} gives a \textit{parametrised Segal object} description for normed algebras in cartesian structures;
notably, this simplifies greatly in the case of coefficient systems.
\begin{proposition}[{\cite[Thm.~II.2.1, Thm.~I.4.1,Prop.~I.3.4]{nat-thesis}}]\label{prop:cart}
    If a category $\D$ has finite products, then $\cofree(\D)$ has $\NN$-indexed products;
    moreover, there is a functor
    \[
        F\colon \int \cofree(\D)^\times \rightarrow \D,
    \]
    natural in product-preserving functors in $\D$
    whose restriction 
    \[
        \Fun\prn{\baseCat_{/V}, \D} \simeq \cofree(\D)^V \hookrightarrow \int \cofree(\D)^\times \rightarrow \D
    \]
    is evaluation at $V$, and such that
    the composition
    \[
        \NAlg{\cofree\prn{\D}^\times}^{\baseCat} \hookrightarrow \Fun\prn{\SFTN{\TT}{\NN}, \int \cofree(\D)^\times} \xrightarrow{F_*} \Fun\prn{\SFTN{\TT}{\NN}, \D}
    \]
    is fully faithful with image the product-preserving functors. This yields an equivalence
    \[
        \NAlg{\GCat[\times]} \xrightarrow{\;\;\;\; \sim \;\;\;\;} \GCat_{\NN-\otimes}.
    \]
\end{proposition}
\begin{remark}
    The proof of \Cref{prop:cart} uses atomicity in an integral way, hence our ongoing assumption that we are normed along an atomic orbital subcategory.
\end{remark}

\begin{example}[{\cite[Cor.~II.1.9,Cor.~II.1.39]{nat-thesis}}]
    We may form the \mdef{pointwise $\NN$-normed structure}, as a $\baseCat$-functor
    \[
        \udl{\Fun}\prn{-,-}\colon \GCat[\mathrm{op}] \times \GCat_{\NN-\otimes} \longrightarrow \GCat_{\NN-\otimes}.
    \]
    Lifting the formation of functor $\baseCat$-categories, and satisfying the formula $\udl{\Fun}\prn{\udl{*}, \CC} \simeq \CC$.
    In particular, evaluation acquires an $\NN$-normed lift $\ev_D\colon \udl{\Fun}\prn{\DD,\CC} \rightarrow \CC$ for all $D \in \DD^{\baseCat}$.
\end{example}

\subsection{Bousfield localisations and presentability} \label{subsection:bousfield-localisations}
In this section, we will collect material about parametrised Bousfield localisation which are used to define the notion of parametrised presentability.
\begin{definition}
    A \mdef{Bousfield localisation} is a $\baseCat$-functor $L \colon \CC \to \BB$ with fully faithful right adjoint;
    a map in $\CC^V$ is an \mdef{$L$-equivalence} if it is sent to an equivalence by $L^V$.
\end{definition}
\begin{lemma}[{\cite[Prop.~6.3.2]{kaifPresentable}}] \label{lem:Bousfield-loc-is-loc}
    Let $L \colon \CC \to \BB$ be a Bousfield localisation. Let $\DD$ be another $\TT$-category. Then precomposition with $L$ induces a fully faithful functor 
    $L^\ast \colon \udl{\Fun}(\BB,\DD) \longrightarrow \udl{\Fun}(\CC,\DD)$
    with essential image the functors sending $L$-equivalences in $\CC$ to equivalences in $\DD$.
\end{lemma}
We will need often need normed lifts of Bousfield localisations:
\begin{proposition}[{\cite[Cor.~II.2.56]{nat-thesis}}] \label{prop:monoidal-structure-Bousfield-loc}
    Let $\CC$ be an $\NN$-normed $\baseCat$-category, and $L\colon \CC \rightarrow \BB$ a Bousfield localisation such that the $\baseCat$-class of $L$-equivalences is closed under $\NN$-norms.
    Then, $\BB$ and $L$ jointly uniquely lift to an $\NN$-normed $\baseCat$-category and $\NN$-normed functor.
\end{proposition}
We call these \mdef{$\NN$-normed Bousfield localisations}. 
Now, an important class of Bousfield localisation arises from \textit{local objects}:
\begin{definition}
    Let $\udl{S} \subset \udl{\Fun}\prn{\udl{\Delta}^1, \CC}$ be a small set of morphisms in a $\TT$-category $\CC$.
    Given $V$ an orbit,
    we call an $V$-object $C \in \CC^V$ \mdef{$\udl{S}$-local} if, for all maps of orbits $u \colon U \to V$ and all $U$-morphisms $s \colon A \to B \in \udl{S}^U$,  the induced map
    \[
        s^\ast \colon \Map_{\CC^U}\prn{B
    ,u^\ast C} \longrightarrow \Map_{\CC^U}\prn{A,u^\ast C}
    \] 
    is an equivalence.
    We write $L_{\udl{S}} \CC \subset \CC$ for the full $\baseCat$-subcategory of local objects.
\end{definition}
We may now recall a definition of presentability:
\begin{definition}
  A $\baseCat$-category $\CC$ is called \mdef{presentable} if there exists a small $\baseCat$-category $\AA$, small set of morphisms $\udl{S}$ in $\udl{\presheaf}(\AA)$, and an equivalence
  \[
    L_{\udl{S}} \udl{\presheaf}(\AA) \xrightarrow{\;\;\; \sim \;\;\;} \CC.
  \]
  We write $\udl{\Pr}^L \subset \wGCat$ for the non-full $\baseCat$-subcategory whose $V$-objects are presentable $V$-categories and $V$-morphisms are left adjoint $V$-functors.
\end{definition}
We direct the reader to \cite{kaifPresentable} or \cite{mwpresentable} for a thorough treatment of these.

Next, we recall the notion of strong saturations from \cite[Def. 6.3.5]{kaifPresentable}. 
The relevance of this for us is that it gives a manipulable characterisation of $L$-equivalences.

\begin{definition}\label{defn:GSaturationAxioms}
Let $\udl{{S}}$ be a $\baseCat$--collection of morphisms in a ${\baseCat}$--category $\underline{\sC}$. We say that it is \mdef{strongly saturated} if the following conditions are satisfied:
\begin{enumerate}
    \item (Pushout closure) For all orbits $V$ and spans $C \xleftarrow{c} A\xrightarrow{a} B$  in $\CC^V$ such that $a$ lies in $\udl{S}^V$, the map $a' \colon C \rightarrow C\cup_AB$ lies in $S^V$;
    \item (${\baseCat}$-colimit closure) The full ${\baseCat}$-subcategory $\udl{{S}} \subseteq \underline{\func}(\udl{\Delta}^1, \underline{\sC})$ is closed under ${\baseCat}$-colimits; and
    \item (2-out-of-3) for all orbits $V$ and composable sequences $A \xrightarrow{a} B \xrightarrow{b} C$ in $\CC^V$, whenever any two of the three morphisms $\{a, b, b\circ a\}$ are in $\udl{{S}}^V$, the third is as well.
\end{enumerate}
For any $\baseCat$--collection of morphisms $\udl{{S}}$, we define its \mdef{strong saturation} $\udl{\overline{S}}$ to be the smallest ${\baseCat}$--collection containing $\udl{S}$ which is strongly saturated.
\end{definition}
\begin{proposition} \label{lem:local-object-is-loc}
    Let $\CC$ be a presentable $\TT$-category and $\udl{S}$ a small $\TT$-set of morphisms. Then the inclusion $L_{\udl{S}}\CC \subset \CC$ of the $\udl{S}$-local objects admits a left adjoint $L \colon \CC \to L_{\udl{S}}\CC$, and the $L$-equivalences are precisely the strong-saturation $\overline{\udl{S}}$.
\end{proposition}
\begin{proof}
    The existence of a Bousfield localisation is \cite[Prop.~2.4.1.7.]{mwpresentable}.
    In the proof it is also shown that the unit morphisms $X \to LX$ of the adjunction are morphisms which are contained in the fibrewise saturation of $\GSpaces$-colimits of morphisms in $\udl{S}$, and hence in the parametrised saturation of $\udl{S}$. Using a $2$-out-of-$3$ argument this shows that any $L$-equivalences are contained in parametrised saturation of $\udl{S}$  
\end{proof}

\section{Distributivity and norms of cocomplete categories}\label{section:distributivity}
The notion of distributivity, originally due to Roos \cite{roos}, generalises colimit-preserving functors to the setting of multivariable functors. 
This was extended by Nardin \cite{nardinThesis} to parametrisation over atomic orbital categories, which we now further extend.

For convenience, we will henceforth freely use the notation of \Cref{lem:fibrewise_product}.
This section will crucially construct a diagram
\[\begin{tikzcd}[ampersand replacement=\&,row sep=tiny]
	{\GCatcoco{\allcolim}} \& {\Catcoconeproj\prn{\allcolim}} \& \Catcocone \& \Catcoconebig \& {\wGCat \times \wGCat} \\
	{\udl{\Cat}^{\mathrm{irex}(\kappa)}} \& {\udl{\Cat}^{\mathrm{rex}(\kappa)}} \& {\udl{\Cat}^{\mathrm{irex}(\kappa)}} \& {\udl{\Pr}^L} \& {\GCatcoco{\allcolim}}
	\arrow["{\udl{\wpresheaf}}"', two heads, from=1-2, to=1-1]
	\arrow[hook, from=1-2, to=1-3]
	\arrow[two heads, from=1-4, to=1-3]
	\arrow["{\prn{\ev_0,\ev_1}}", from=1-4, to=1-5]
	\arrow["\natural"', two heads, from=2-2, to=2-1]
	\arrow["{\udl{\ind}^{\kappa}}", hook, from=2-3, to=2-4]
	\arrow[hook, from=2-4, to=2-5]
\end{tikzcd}\]
where $\wGCat \times \wGCat$ has the cartesian structure, every two-headed arrow is an \textit{$\NN$-normed Bousfield localisation}, and every hooked arrow is an \textit{$\NN$-normed cocontinuous $\baseCat$-subcategory inclusion}.
In particular, we will lift this diagram uniquely from the underlying category of $\baseCat$-categories together with an $\NN$-normed lift of its subdiagram involving $\Catcocone$, $\Catcoconebig$, and $\wGCat \times \wGCat$ and the object $\udl{\Cat}^{\mathrm{rex}(\kappa)}$.
Notably, we show in \Cref{prop:transf_examples} that distributivity of norm multifunctors is transported leftward along each arrow. 

Indeed, in \Cref{subsec:definition_distributivity} we will construct $\Catcocone$, and use it to lay out the fundamentals of distributivity of parametrised multi-functors.
Then, in \Cref{subsec:distributivity_cartesian} we study distributivity of \textit{norm} multifunctors, constructing the two upper right arrows and proving that $\wGCat \times \wGCat$ (along with many other cartesian structures) is distributive.
Following this, in \Cref{subsection:cubicalDecompositions,subsec:distributivity_free_cocompletions,subsec:distributivity_general_cocompletions}, we explicitly construct \textit{parametrised relative cocompletions}, constructing left adjoint-objects to the inclusion $\GCatcoco{\allcolim} \hookrightarrow \Catcoconeproj\prn{\allcolim}$.
In \Cref{subsec:sym_mon_structure_cocomplete}, we finally provide the desired lifts for the two upper left arrows, as well as a distributive normed structure on $\udl{\Cat}^{\mathrm{rex}(\kappa)}$.
Lastly, in \Cref{subsec:norms_presentable}, we lift the bottom diagram.

\subsection{Distributivity and the category of cocones} \label{subsec:definition_distributivity}

The goal of this subsection is to define an $\NN$-normed $\baseCat$-category of classes of cocones $\Catcocone$: 
its $V$-objects are $V$-categories together with a $V$-class of cocones $\udl{K}^{\triangleright} \to U^*\CC$ and its $V$-morphisms are functors preserving the $V$-class of cocones.
Crucially, the $\baseCat$-category of cocomplete $\baseCat$-categories will be a \textit{full} $\baseCat$-subcategory of $\Catcocone$. 

We begin by constructing a variant with exactly one cocone.
\begin{construction} \label{const:lax-structure-cocones}
Given a large $\baseCat$-category $\CC$, we define the \mdef{$\baseCat$-category of cocones in $\CC$} as the pullback
\[ \begin{tikzcd} \Cocone{\CC} \ar[r] \ar[d] \arrow[rd, "\lrcorner" very near start,phantom] & \GCat
\ar[d,"{(-)^{\triangleright}}"] \\
\overslice{\GCat}{\CC} \ar[r] & \GCat \end{tikzcd} \]
where $\GCat_{/\CC} \coloneqq \wGCat_{/\CC} \times_{\wGCat} \GCat$.
Note that $\GCat_{/\CC}$ itself is only large, not very large, as the fibres of the forgetful functor $\GCat_{/\CC} \rightarrow \GCat$  to a large $\baseCat$-category are themselves only large;
hence $\Cocone{\CC}$ is also only large.
Given an oplax $\NN$-normed structure on $(-)^{\triangleright}$, 
\Cref{con:free-normed-cocart} will provide a lax $\NN$-normed structure on the $\baseCat$-functor 
\begin{align*}
    \wGCat &\longrightarrow \udl{\LargeCat} \\
    \CC & \mapsto \Cocone{\CC}
\end{align*}

Indeed, letting $\GCatterminal \subset \GCat$ denote the (non-full) $\baseCat$-subcategory spanned by $V$-categories with a terminal object and $V$-functors preserving them, 
this $\GCatterminal$ is closed under indexed products, so that the above inclusion enhances to a cartesian $\NN$-normed functor. 
It admits $(-)^{\triangleright} \colon \GCat \longrightarrow \GCatterminal$ as a left adjoint, which is given an oplax $\NN$-normed structure by the doctrinal adjunction \cite[Prop.~I.3.6]{nat-thesis}\footnote{This proposition statement has a typo; in its notation, and in the proof, $F$ is assumed to be \textit{lax}, not oplax.}. The oplax $\NN$-normed $\baseCat$-functor in question now arises as the composite of this left adjoint and the inclusion.
\end{construction}
Using the description of the norm functor from \Cref{con:free-normed-cocart}, we learn that the norm of a $W$-cocone $m \colon \udl{K}^{\triangleright} \to \CC$ along an $\NN$-admissible map $w \colon W \to V$ is the cocone 
\[ w_\ast m \colon (w_\ast \udl{K})^{\triangleright} \xrightarrow{\;\;\;\;\; \canonical \;\;\;\;\;} w_\ast (\udl{K}^{\triangleright}) \xrightarrow{\;\;\;\;\; w_* m \;\;\;\;} w_\ast \CC \]
where $\canonical$ is mate to the $\prn{w^* \dashv w_*}$-counit $w^* (w_\ast \udl{K})^{\triangleright} \simeq (w^\ast w_\ast \udl{K})^{\triangleright} \to \udl{K}^{\triangleright}$. 
\begin{construction} \label{const:lax-structure-subsets}
    Endow $\Arrow(\wGCat)$ with the pointwise $\NN$-normed structure.
    The (very large) \mdef{$\NN$-normed $\TT$-category of full subcategories} is the norm-closed full $\TT$-subcategory $\wWCat \subset \Arrow\prn{\wGCat}$ spanned by full subcategory inclusions;
    indeed norm-closure may be seen as $\Arrow\prn{\wGCat}$-norms are fibrewise and pointwise computed as products of various evaluations by \Cref{lem:fibrewise_product}. 
\end{construction}

Note that the $w$-norm of a subcategory $\DD \subset \CC$ simply applies $w_\ast$ to the inclusion.
We will show that the composite $\NN$-normed $\baseCat$-functor $t \colon \wWCat \rightarrow \wGCat$  is a cocartesian fibration, compatible with the norms in the following sense:
\begin{definition}\label{def:norm-cocart-fib}
    An $\NN$-normed $\baseCat$-functor $\varphi\colon \DD \to \CC$ is called an \mdef{$\NN$-normed cocartesian fibration} if $\DD^V \to \CC^V$ is a cocartesian fibration for all orbits $V$ and $\varphi$-cocartesian arrows are closed under restrictions and $\NN$-norms.
\end{definition}
\begin{lemma} \label{lem:SubCat-Cat-cocartesian}
    The target map $t \colon \wWCat \to \wGCat$ is an $\NN$-normed cocartesian fibration.
\end{lemma}
\begin{proof}
    For any $V$, the subcategory inclusions in $\wCat[V]$ are the right part of a factorisation system (together with the essentially surjective functors) and these factorisation systems are compatible with restrictions and norms. It follows from \cite[Prop.~6.7]{lnp25} that the functor $t \colon \wWCat^V \to \Cat[V]$ is a cocartesian fibration. It also follows from the description of the cocartesian arrows in \textit{loc. cit.} that the cocartesian arrows are closed under restriction and under norms.
\end{proof}

We may use this to construct the category of cocones:
\begin{construction}\label{cons:catcocone}
    \Cref{lem:SubCat-Cat-cocartesian} and \Cref{const:lax-structure-cocones} give us the necessary requirements to apply
    \Cref{lem:pullback-partially-lax} to describe the following pullback square in the category of $\NN$-normed $\baseCat$-categories and lax $\NN$-normed functors:
    \begin{equation}\label{eq:catcocone}
     \begin{tikzcd}[column sep=huge]
        \Catcocone \ar[r] \ar[d, "t"'] \arrow[rd, "\lrcorner" very near start, phantom] & \wWCat \ar[d,"t"] \\
        \wGCat \ar[r,"{\Cocone{-}}"] & \wGCat
    \end{tikzcd}
    \end{equation}
    We call a $V$-object in $\Catcocone$ a \mdef{class of cocones} and a $V$-morphism a \mdef{functor preserving the class of cocones}.

    Moreover, \Cref{lem:pullback-partially-lax} also describes the $w$-norm functor on $\Catcocone$:
    $w_{\otimes}\prn{\CC,\udl{\RR}}$ is the class of cocones $w_\ast \udl{\R}$ on $w_\ast \CC$ of the form $w_\ast m \colon (w_\ast \udl{K})^{\triangleright} \rightarrow w_*(\udl{K}^{\triangleright}) \to w_\ast \CC$ for some $W$-cocone $m \colon \udl{K}^{\triangleright} \to \CC$ in $\udl{\RR}$. In particular, it is easy to see that the lax $\NN$-normed functor $t\colon \Catcocone\rightarrow \wGCat$ is in fact an $\NN$-normed functor.

\end{construction}
By \Cref{lem:pullback-partially-lax}, taking underlying $\baseCat$-categories takes \eqref{eq:catcocone} to a pullback square;
that is, a $W$-class of cocones is simply a $W$-category $\CC$ and a full $W$-subcategory $\udl{\RR} \subset \Cocone{\CC}$,  i.e. a restriction-stable choice of distinguished $U$-cocones $\udl{K}^{\triangleright} \to U^*\CC$ for all $U \rightarrow W$.
The mapping \animae{} are similar:
\begin{proposition} \label{Catcocone-subcat}
    Let $(\CC,\udl{\R})$ and $(\DD,\udl{\mathcal{S}})$ be two classes of cocones. The map 
    \[ t\colon \Map_{\Catcocone} \left((\CC,\udl{\R}),(\DD,\udl{\mathcal{S}}) \right) \longrightarrow \Map_{\GCat}(\CC,\DD)  \] induced by $t$ on mapping \animae{}
    is an inclusion onto those path components $\CC \to \DD$ which send cocones in $\udl{\R}$ to cocones in $\udl{\mathcal{S}}$.
\end{proposition}
\begin{proof}
    We want to show that \eqref{eq:inverting-cocones} has $0$-truncated fibres. 
    As $\prn{\Catcocone}^{\baseCat} \to \GCat[\baseCat]$ is a cocartesian fibration, the fibre over a $F \colon \CC \to \DD$ is given by
    \begin{equation} \label{eq:inverting-cocones} \Map_{(\Catcocone)}(F_\ast(\CC,\udl{\R}),(\DD,\udl{\mathcal{S}})) \, \simeq \Map_{\overslice{\wWCat}{\DD}}(F_\ast \udl{\R},\udl{\mathcal{S}}) \end{equation}
    The category $\overslice{\WCat}{\DD}$ is a poset as subcategory inclusions are monomorphisms by \cite[Prop.~I.2.65, Cor.~I.2.67]{nat-thesis}. So the mapping anima actually is either empty of contractible. The description of the essential image follows by unwinding the functoriality of $\Catcocone$ and $\wWCat \to \wGCat$.
\end{proof}

Next, we may use \Cref{Catcocone-subcat} to embed in the $\KK$-cocomplete $\baseCat$-categories:
\begin{construction} \label{defn:K-cocompl-cat}
    Given a $V$-diagram $\udl{K} \to V^*\CC$ which admits a colimit, its colimit cocone yields a $V$-object in $\Cocone{\CC}$.

    For $\KK \subset \GCat$ a full $\baseCat$-subcategory and $\CC$ a $\KK$-cocomplete $\baseCat$-category, we write $\udl{\RR}_{\KK}(\CC) \subset \Cocone{\CC}$ for the full $\baseCat$-subcategory of colimit $V$-cocones over diagrams with domain lying in $\KK$.
    We write $\GCatcoco{\KK} \subset \Catcocone$ for the full $\baseCat$-subcategory spanned by $V$-classes of cocones $(\CC,\udl{\RR}_{\KK}(\CC))$.
\end{construction}

\begin{corollary}\label{cor:catK_subcat}
    The composite $\iota\colon \GCatcoco{\KK} \subset \Catcocone \to \wGCat$ is the inclusion of the $\baseCat$-subcategory of $\KK$-cocomplete $V$-categories and $\KK$-cocontinuous $V$-functors.
\end{corollary}
\begin{proof}
    By \Cref{Catcocone-subcat}, $\iota$ is faithful.
    Moreover, since every equivalence preserves $\KK$-colimits, the map of $\baseCat$-\animae{} $\iota^{\simeq}$ is an inclusion, i.e. $\iota$ is a $\baseCat$-subcategory inclusion.
    \Cref{Catcocone-subcat} and the definition of $\GCatcoco{\KK}$ identify it as the desired one.
\end{proof}

\begin{definition}
    Let $\CC$ be a $\baseCat$-category together with a $\baseCat$-class of cocones $\udl{\RR}$.
    We say that a $\baseCat$-functor $\CC \to \DD$ \mdef{sends $\udl{\RR}$-cocones to colimit cocones} if it lies in the image of 
    \[
    \Map_{\Catcocone} \left(\prn{\CC,\udl{\R}},\prn{\DD,\udl{\RR}_{\GCat}\prn{\DD}} \right) \hookrightarrow \Map_{\wGCat}(\CC,\DD),
    \]
    and we denote the full $\baseCat$-subcategory of these as 
    \[ 
        \udl{\Fun}^{\udl{\RR}}(\CC,\DD) \subset \udl{\Fun}(\CC,\DD).
    \]

    Let  $w\colon W\rightarrow V$ be an $\NN$-admissible map of orbit-finite $\TT$-\animae{}. 
    Given $\CC$  a $W$-category and $\DD$  a $V$-category,
    a \mdef{$w$-multifunctor} from $\CC$ to $\DD$ is a $V$-functor $F\colon w_* \CC \rightarrow \DD$.
    We say that such a multifunctor \mdef{distributes $\RR$-cocones} if it sends $w_\ast \RR$-cocones in $w_\ast \CC$ to colimit cocones.

    In the case where $\udl{\RR} = \udl{\RR}_{\KK}(\CC)$ is the $\baseCat$-class of $\KK$-colimit cocones in $\CC$, we simply say that $F$ \mdef{distributes $\KK$-colimits}; we write \[
\udl{\func}\distributiveLocalisation{w}{\KK}{}(w_*\CC,\DD)  \coloneqq
\udl{\func}^{w_*\udl{\RR}_{\KK}\prn{\CC}}(w_*\CC,\DD).
\]
    If $\KK=\underline{\Cat}$, we simply say that the multifunctor is distributive, and
    we write 
    \[
    \udl{\func}\distributeAll{w}(w_*\udl{\sC},\udl{\D})\coloneqq \udl{\func}\distributiveLocalisation{w}{\udl{\RR}_{{\GCat}}\prn{\CC}}{}(w_*\CC,\DD).
    \]
    We call this the \mdef{the $\baseCat$-category of distributive $w$-multifunctors.}
\end{definition}

\begin{remark}\label{rmk:assembly-map}
    Assume $\CC$ and $\DD$  admit colimits for $\KK$-indexed and $w_*{\KK}$-indexed diagrams, respectively.
    We can describe the special case of  $\KK$-distributivity as the stipulation that the colimit assembly map
    \[  
        \beta_{F,m}\colon \colim\prn{\overline{w}_* \udl{K}\xrightarrow{\;\;\; w_\ast m \;\;\;}\overline{w}_*\overline{u}^\ast \CC \simeq u^\ast w_\ast \CC \xlongrightarrow{u^*F} u^\ast \DD}         \longrightarrow F \left( \colim(\udl{K} \xrightarrow{\;\;\; m \;\;\;} \overline{u}^*\CC) \right)
    \]
   is an equivalence for any map $u\colon U \rightarrow V$ between orbit-finite $\baseCat$-\animae{} and any $Z$-functor $m \colon \udl{K} \to \overline{u}^*\CC$ with $\udl{K} \in \KK^Z$, where $Z$ is as in \eqref{eqn:standard_basecat_pullback}.
    Concretely, $\beta$ is the value at $m$ of the vertically left-adjointed Beck-Chevalley  map of the following square in $\GCat[U]$
    \begin{equation}\label{eq:BC_pulled}
    \begin{tikzcd}
        \overline{w}_*\overline{u}^*\udl{\sC} \ar[rr, "u^*F"] \dar["\overline{w}_*\prn{\udl{K}^*}"] && u^*\udl{\D}\dar["(\overline{w}_*\udl{K})^*"]\\
        \overline{w}_*\udl{\func}({\udl{K}},\overline{u}^*\udl{\sC}) \rar["\canonical"]& \udl{\func}(\overline{w}_*{\udl{K}},\overline{w}_*\overline{u}^*\udl{\sC}) \rar["u^*F"]& \udl{\func}(\overline{w}_*{\udl{K}},u^*\udl{\D}).
    \end{tikzcd}
  \end{equation}
    This may be reframed as the simultaneous condition that
    \begin{enumerate}[label={(\alph*)}]
      \item \label[Condition]{cond:a} \eqref{eq:BC_pulled} is vertically left-adjointable when $u$ is an equivalence (hence $\overline{w} = w$ and $\overline{u}$ is an equivalence), and
      \item \Cref{cond:a} holds for the $\overline{w}$-multifunctor $u^* F$ for all $u\colon U \rightarrow V$.
    \end{enumerate}
\end{remark}

The property of being distributive enjoys a closure property under compositions.

\begin{proposition}\label{prop:dist-compose} \label{lem:postcomposing_distributive_with_colimit_pres}
  Given a subcategory $\KK \subset U^\ast \GCat$ and a composable sequence of $\NN$-admissible maps $U \xrightarrow{u} W \xrightarrow{w} V$ and a $\KK$-cocomplete $U$-category $\CC$, a $u_\ast \KK$-cocomplete $W$-category $\DD$ and $(wu)_\ast \KK$-cocomplete $V$-category $\EE$. 
  Let $F \colon w_\ast \DD \to \EE$ and $G \colon u_\ast \CC \to \DD$ be $u_*\KK$- and $\KK$-distributive multifunctors. Then $F \circ w_\ast G \colon (wu)_\ast \CC \to \EE$ is $\KK$-distributive.
\end{proposition}
\begin{proof}
     By definition, there is a morphism in $\Catcocone$ lifting $G$ from $u_{\otimes}(\CC,\udl{\RR}_{\KK}(\CC))$ into $(\DD,\udl{\RR}_{u_\ast \KK}(\DD))$. As the forgetful functor to $\Catcocone \rightarrow \wGCat$ is $\NN$-normed, we can apply $w_{\otimes}$ to obtain a morphism from $w_{\otimes} u_{\otimes} (\CC,\udl{\RR}_{\KK}(\CC))$ to $w_{\otimes}(\DD,\udl{\RR}_{u_\ast \KK}(\DD))$. Moreover, there is a morphism in $\Catcocone$ lifting $F$ from $w_{\otimes}(\DD,\udl{\RR}_{u_\ast \KK}(\DD))$ to $(\EE,\udl{\RR}_{(wu)_\ast \KK}(\EE))$. Composing these two morphism witnesses that the composite $F \circ w_\ast G$ is $\KK$-distributive.
\end{proof}
\subsection{Distributive normed categories}\label{subsec:distributivity_cartesian}
We  now define distributive normed categories, generalising the notion of categories for which the tensor product commutes with colimit in each variable seperately.
\begin{definition} \label{def:distributive-category}
    For $\KK \subset \GCat$, an $\NN$-normed $\baseCat$-category $\CC$ is called \mdef{$\KK$-distributive} if for all $\NN$-admissible maps of orbit-finite $\TT$-\animae{} $w \colon W \to V$, the $w$-norm $w$-multifunctor 
    \[ w_{\otimes} \colon w_{\ast} W^\ast \CC \longrightarrow V^\ast \CC \] is $\KK$-distributive. 
    In the case where $\KK=\GCat$, we say $\CC$ is \mdef{distributive}.
\end{definition}
Before we move on, we point out a few techniques for transferring this property.
\begin{proposition}\label{prop:transf_examples}
    Let $\varphi\colon \CC \rightarrow \DD$ be a $\KK$-cocontinuous $\NN$-normed $\baseCat$-functor between $\KK$-cocomplete $\NN$-normed $\baseCat$-categories.
    \begin{enumerate}[label = (\arabic*)]
        \item If $\varphi$ is conservative and $\DD$ is $\KK$-distributive, then $\CC$ is $\KK$-distributive.
        \item If $\varphi$ is a Bousfield localisation and $\CC$ is $\KK$-distributive, then $\DD$ is $\KK$-distributive.
    \end{enumerate}
\end{proposition}
\begin{proof}
    For Claim 1,
    let $m\colon \udl{K} \rightarrow \overline{u}^* \CC$ be a $\KK$-diagram.
    Cocontinuity and $\NN$-normedness produce an equivalence of  \Cref{rmk:assembly-map}'s assembly maps $\varphi\prn{\beta_{\overline{w}^{\CC}_{\otimes}, m}} \simeq \beta_{\overline{w}^{\DD}_{\otimes}, \varphi \circ m}$.
    Distributivity of $\DD$ implies $\beta_{\overline{w}^{\DD}_{\otimes}, \varphi \circ m}$ is an equivalence, so conservativity implies that $\beta_{\overline{w}^{\CC}_{\otimes}, m}$ is an equivalence, yielding distributivity of $\CC$ under \Cref{rmk:assembly-map}.

    Claim 2 is similar: 
    if instead $m\colon \udl{K} \rightarrow \overline{u}^*\DD$ and $\varphi \dashv \iota$, then we have equivalences $\varphi\prn{\beta_{\overline{w}^{\CC}, \iota \circ m}} \simeq \beta_{\overline{w}^{\DD}, \varphi \circ \iota \circ m} \simeq  \beta_{\overline{w}^{\DD}, m}$; since $\beta_{\overline{w}^{\CC}, \iota \circ m}$ is an equivalence, $\beta_{\overline{w}^{\DD}, m}$ is, too. 
\end{proof}

A fundamental example of distributive $\NN$-normed structures out of which we will build others is the cartesian structure on coefficient systems.
The following generalises \cite[Prop. 3.2.5]{NS}.

\begin{proposition}\label{prop:distributivity_of_cartesian_symmon}
    Let $\D$ be a category which admits finite products. 
    Then $\cofree(\D)\in \LargeCat[\TT]$  admits finite $\NN$-indexed products. Moreover, if  $\D$ is additionally a cocomplete category and the product bifunctor commutes with colimits separately in each variable, then $\cofree(\D)\in \LargeCat[\TT]$ is also cocomplete and its cartesian $\NN$-normed structure is distributive. 
\end{proposition}
For instance, these hypotheses hold when $\D$ is cartesian-closed.
Now, to prove this proposition, we will use properties of the coefficient system functor.
Indeed, $\cofree$ has the universal property as a right adjoint to the cocartesian unstraightening functor $\int\colon \wCat[\baseCat] \rightarrow \wCat$, under which we have a fibrewise expressions for (co)limits.
\begin{lemma}[{\cite[Prop.~5.5-5.6]{Sha23}}]\label{lem:cofree-colims}
    Fix $\D$ a category and $\udl{K}$ a $\baseCat$-category. A $\baseCat$-cocone $\overline{p}\colon \udl{K}^{\triangleright} \rightarrow \cofree(\D)$ is a colimit if and only if its restricted mate
        \[
            (\udl{K}^{V})^{\triangleright} \hookrightarrow \int\udl{K}^{\triangleright}
            \xrightarrow{\;\;\;\;\; p' \;\;\;\;\;} \D
        \]
        is a colimit diagram for all orbits $V \in \baseCat$.
\end{lemma}

\begin{proof}[Proof of \cref{prop:distributivity_of_cartesian_symmon}] 
    We give a proof similar to \cite[Lem.~3.18]{nardinThesis}.

    By restriction-stability of our assumptions, we need only show for all $\NN$-admissible maps $w\colon W \rightarrow V$ with $V \in \baseCat$ a terminal object that
    \[
        w_*\colon w_*w^*\cofree(\D) \longrightarrow \cofree(\D)
    \]
    is distributive.

    Let $u\colon U \rightarrow V$ be the unique map from some orbit $U$. 
    Recall from \eqref{eqn:standard_basecat_pullback} that we have an equivalence $u^*w_*\simeq \overline{w}_*\overline{u}^*\simeq \prod_{i=1}^n\overline{w}_{i*}\overline{u}_i^*$.
    This time, we have a mate correspondence
   \[
    \adjustbox{max width=\textwidth}{
     \begin{tikzcd}[ampersand replacement=\&, column sep=small]
	{ (u^*w_*\udl{K})^{\triangleright}} \& {u^*w_* \prn{\udl{K}^{\triangleright}}} \& {u^*w_*w^*\cofree(\D)} \& { u^*\cofree(\D)} \\
	{ \int\prn{u^*w_*\udl{K}}^{\triangleright}} \& {\int\prn{u^*w_*\prn{\udl{K}^{\triangleright}}}} \& {\int\prn{u^*w_* w^*\cofree(\D)}} \rar["u^*w_*"]\& \int\prn{u^*\cofree(\D)} \rar["\epsilon"]\& \D \dar[equal]\\
	\left(\prod_{i=1}^n{\udl{K}^{Z_i}}\right)^{\triangleright} \& \& \prod_{i=1}^n\func((\baseCat_{/Z_i})\op,\D)\rar["\prod_i\overline{w}_{i*}"]\uar[hook]\&\func((\baseCat_{/U})\op,\D)\rar["\eval_{U}"]\uar[hook]\& \D 
	\arrow["\can", from=1-1, to=1-2]
	\arrow["{u^*w_* \overline{p}}", from=1-2, to=1-3]
	\arrow["{u^*w_*}", from=1-3, to=1-4]
	\arrow[between={0.3}{0.7}, squiggly, tail reversed, from=1-3, to=2-3]
	\arrow["\can", from=2-1, to=2-2]
	\arrow[from=2-2, to=2-3]
	\arrow[from=3-1, to=2-1, hook]
	\arrow[from=3-1, to=3-3]
\end{tikzcd}
}
\]
    Here the bottom right square involving $\epsilon$ and $\eval_U$ is obtained by \cite[Obs. A.0.2]{hilmanRamzi}. On the other hand, for each $i$, the composite $\func((\baseCat_{/Z_i})\op,\D)\xrightarrow{\overline{w}_{i*}} \func((\baseCat_{/U})\op,\D)\xrightarrow{\eval_U}\D$  is given by the evaluation functor $\func((\baseCat_{/Z_i})\op,\D)\xrightarrow{\eval_{Z_i}}\D$. Therefore, the bottom horizontal composite is given by the product across $i$ of the colimit diagrams $(\udl{K}^{Z_i})^{\triangleright}\rightarrow \D$, and hence is a colimit diagram by distributivity of $\D$.
    Quantifying this over all $U$ demonstrates the condition of \Cref{lem:cofree-colims} for the top horizontal composite, so \Cref{lem:cofree-colims} shows that it is a colimit diagram,
    completing the proof.
\end{proof}

We will use this to show that $\Catcocone$ itself is distributive.
Consider the pullback diagram of lax $\NN$-normed $\baseCat$-functors
\[ \begin{tikzcd} 
    \Catcoconebig \ar[r] \ar[d]  \arrow[rd, "\lrcorner" very near start, phantom] & \Arrow(\wGCat) \ar[d, "t"] \\ 
    \wGCat \ar[r,"{\Cocone{-}}"] & \wGCat 
\end{tikzcd} 
\]
which exists under \Cref{lem:pullback-partially-lax}, since $\Arrow\prn{\wGCat}  \rightarrow \wGCat$ is an $\NN$-normed cocartesian fibration.

As $\wWCat \subset \Arrow(\wGCat)$ was defined to be the full subcategory on the fully faithful functors, which is the right class of a factorisation system, the inclusion admits a left adjoint $L\colon \Arrow(\wGCat) \to \wWCat$ sending an arbitrary functor to the inclusion of the essential image. Moreover, $L$ is an $\NN$-normed Bousfield localisation, as the classes of fully faithful functors and essentially surjective functors are closed under indexed products, so that the full image inclusion of an indexed product is the indexed product of the full image inclusion.
Pulling this back along $\Cocone{-}$ yields an $\NN$-normed Bousfield localisation $\Catcoconebig \to \Catcocone$ over $\wGCat$.
We use this to prove the following.
\begin{corollary} \label{prop:Catcocone-distributive}
    The $\NN$-normed $\baseCat$-categories $\Catcoconebig$ and $\Catcocone$ are distributive.
\end{corollary}
We will need the following lemma:
\begin{lemma} \label{lem:colimits-in-slice}
    Let $\CC$ and $\DD$ be cocomplete $\TT$-categories and $F \colon \CC \to \DD$ be a functor. 
    The pullback $\baseCat$-category
    \[ \begin{tikzcd}
        \PP \ar[r, "F'"] \ar[d,"t'"] \arrow[rd, "\lrcorner" very near start, phantom] & \Arrow(\DD) \ar[d,"t"] \\
        \CC \ar[r,"F"] & \DD
    \end{tikzcd} \]
    is cocomplete and the $\baseCat$-functor  $t' \times s \circ F'\colon \PP \rightarrow \CC \times \DD$ is cocontinuous.
\end{lemma}
\begin{proof}
    The non-parametrised proof in \cite[Prop.~II.1.5(iv)]{nikolausScholze} works here similarly.
\end{proof}

\begin{proof}[Proof of \Cref{prop:Catcocone-distributive}]
    By \Cref{prop:transf_examples}, it is enough to verify just that $\Catcoconebig$ is distributive.
    Indeed, \Cref{lem:colimits-in-slice} yield a conservative cocontinuous $\baseCat$-functor
    \[
        \begin{tikzcd}[ampersand replacement=\&,column sep=large,row sep=tiny]
        	\Catcoconebig \& {\wGCat \times  \wGCat} \\
        	{\prn{\CC, \DD \rightarrow \Cocone{\CC}}} \& {\prn{\CC,\DD}}.
        	\arrow["t^\prime \times s", from=1-1, to=1-2]
        	\arrow["\in"{marking, allow upside down}, draw=none, from=2-1, to=1-1]
        	\arrow[maps to, from=2-1, to=2-2]
        	\arrow["\in"{marking, allow upside down}, draw=none, from=2-2, to=1-2]
        \end{tikzcd}
    \]
    Here, $t^\prime$ is given an $\NN$-normed structure by \Cref{lem:pullback-partially-lax} and $s$ as in the proof of \Cref{con:free-normed-cocart}.
    
    The product of the maps $t'$ and $s$ in normed algebras in very large $\baseCat$-categories yields a normed structure on $t' \times s$.
    Moreover, it is easy to see that distributivity of $\NN$-normed $\baseCat$-categories is closed under finite products, so \Cref{prop:distributivity_of_cartesian_symmon} implies that $\wGCat \times \wGCat$ is distributive;
    hence \Cref{prop:transf_examples}
    shows that $\Catcoconebig$ is distributive.
\end{proof}
\begin{remark} \label{rmk:colimits-Catcocone}
    A colimit in $\Catcocone$ is computed as follows: given a diagram $\udl{I}  \rightarrow \Catcocone $ sending $i$ to $ (\CC_i,\udl{\RR}_i)$,
    $\colim_{\udl{I}} \prn{\CC_i,\udl{\RR}_i}$ has $\baseCat$-category $\colim_{i \in I} \CC_i$ and $\baseCat$-class of cocones the essential image of
    \[ \colim_{i \in I} \udl{\RR}_{i} \longrightarrow \colim_{i \in I} \Cocone{\CC_i} \longrightarrow \Cocone{\colim_{i \in I} \CC_i} \,, \]
    i.e. it consists of cocones in $\colim_{i \in I} \CC_i$ coming from a cocone in $\CC_i$ via the inclusion.
\end{remark}

\subsection{Cubical analysis of distributivity}\label{subsection:cubicalDecompositions}
In general, an arbitrary class of cocones can lead to ill-behaved notions of distributivity. We thus single out the class of colimit cocones associated to a $
\baseCat$-subcategory $\KK\subset \udl{\cat}$ which contains the terminal category $\udl{\ast}$, which will play a basic role throughout the rest  of the article. Morally, it allows us to hold some variables constant while allowing others to vary as we are wont to in the non-parametrised situation.

\begin{definition}\label{defn:projectable_classes} 
    Let $\udl{\RR}$ be a $\baseCat$-class of cocones in $\udl{\sC}$. We say that $\udl{\RR}$ is \mdef{projectable} if it contains the constant cocones $\constant(X)\colon \udl{\ast}\tcocone\rightarrow U^*\udl{\sC}$ for all orbits $U$ and $X\in U^*\udl{\sC}$.
\end{definition}
\begin{example}
    For $\KK \subset \GCat$, the class $\udl{\RR}_{\KK}\prn{\CC}$ of $\KK$-colimit cocones in $\CC$ is projectable as soon as $\udl{\ast} \in \KK^U$ for all orbits $U$. 
\end{example}
\begin{definition}
    A cocone $\udl{K}\tcocone\rightarrow\udl{\sC}$ is said to be an \mdef{empty cocone} if $\udl{K}\simeq \udl{\emptyset}$.
\end{definition}

The next lemma generalises the following familiar situation: if $\varphi\colon \sC^{\times n}\rightarrow \D$ is a multivariable functor valued in a pointed category which preserves colimits in each variable, then $\varphi(c_1,\ldots,c_n)\simeq 0$ if at least one of the $c_i$'s is the initial object of $\sC$.

\begin{lemma}\label{lem:distributing_over_initial_object}
    Let 
    $w\colon W \rightarrow V$ be an $\NN$-admissible map of orbit-finite $\TT$-\animae{}, $\udl{\sC}$ a $W$-category,  $\udl{\RR}$ a projectable $W$-class of cocones in $\CC$, and $\varphi \colon w_*\udl{\sC}\rightarrow \udl{\D}$ a $(w,\udl{\RR})$-distributive functor.
    Suppose moreover that $u\colon U \rightarrow V$  is a map of orbit-finite $\TT$-\animae{} and  
    \[
    x\colon \udl{\ast}\longrightarrow u^*w_*\CC\simeq \overline{w}_*\overline{u}^*\CC \simeq \prod_{i=1}^n\overline{w}_{i*}\overline{u}_i^*\udl{\sC}
    \]
    a $U$-object whose $i$-th coordinate $\udl{\emptyset}^{\triangleright} \simeq \udl{\ast}\xrightarrow{x_i} \overline{w}_{i*} \overline{u}_i^* \udl{\sC}$ is an empty $\udl{\RR}$-cocone for some $i \in \cbr{1,\dots,n}$.
    Then $\varphi(x) \in \DD^U$ is an initial $U$-object.
\end{lemma}
\begin{proof}
    Assume without loss of generality that $i = 1$.
   Note that $x$ is the $\overline{w}$-image of a $Z$-cocone $p$, written in coordinates as
     \[
     p=\prn{x_1,\constant\prn{x_2},\ldots,\constant\prn{x_n}}
     \colon (\udl{\emptyset}, \udl{\ast},\ldots,\udl{\ast})\tcocone \xrightarrow{\;\;\;\;\;\;\;\;\;\;} (\overline{u}_1^*\udl{\sC}, \overline{u}_2^*\udl{\sC}\ldots, \overline{u}_n^*\udl{\sC}) = \overline{u}^*\udl{\sC},\]
    i.e. as constant in the non-empty coordinates. Projectability stipulates that this is in $\udl{\RR}$.
    
    Now, $(w,\udl{\RR})$-distributivity of $\varphi$  means that $\varphi(x)$ is classified by the colimit diagram
    \[\udl{\ast} \simeq \udl{\emptyset}\tcocone \simeq \big(\overline{w}_*(\udl{\emptyset}, \udl{\ast},\ldots,\udl{\ast})\big)\tcocone \xlongrightarrow{\;\;\;\; \overline{w}_* p \;\;\;\;} \overline{w}_*\overline{u}^*\udl{\sC} \simeq u^* w_* \CC \xlongrightarrow{\;\;\;\; u^* \varphi \;\;\;\;}u^*\udl{\D}\]
    which is an intial $U$-object of $\udl{\D}$, as desired.
\end{proof}

We now work towards supplying the fundamental engine to deal with the combinatorics of distributivity and cubes in the form of \Cref{prop:norming_pushouts}. It is a straightforward extension of the key points in \cite[\textsection 3.2]{kaifNoncommMotives} from the setting where $\baseCat$ is an atomic orbital category to our present setting where $\baseCat$ is equipped with an atomic orbital subcategory $\NN$ consisting of admissible maps. Fortunately, the contents of \textit{loc. cit.} goes through here essentially word-for-word;
indeed, writing $\udl{I}$ for the constant parametrised category associated to a non-parametrised category $I$, the parametrised category $w_*\udl{I}$ is fibrewise various finite products of $I$ by \Cref{lem:fibrewise_product}, which puts us in the setting of the analyses in \cite[\textsection 3.2]{kaifNoncommMotives}.

These considerations will rest on the following  decomposition from \cite{kaifNoncommMotives} for the diagram $w_*\udl{\Lambda}^2_0$ which controls the interaction between $w$-distributive functors and pushouts.

\begin{construction}[``{\cite[Cons. 3.2.3]{kaifNoncommMotives}}'']\label{cons:subcategoriesOfHornPowers}
    We write the inclusion $\Lambda^2_0\subseteq (\Delta^1)^{\times2}$ as 
    \begin{center}
      \begin{tikzcd}[row sep=scriptsize]
            00 \rar\dar & 10 \ar[drr, phantom, "\hookrightarrow"]&& 00 \rar \dar& 10\dar \\
            01 & && 01 \rar & 11
        \end{tikzcd}
    \end{center}
    In this way, for any $n\geq 1$, we may denote the objects in the full subcategory $\prod^n_{i=1}\Lambda^2_0\subseteq \prod^n_{i=1}(\Delta^1)^{\times2}$ as those tuples $(a_i)_{1\leq i\leq n}$ where $a_i\in\{00,10,01\}$. We want to decompose $\prod^n_{i=1}\Lambda^2_0$ in terms of subcategories that we will now specify:
    \begin{itemize}
        \item Let $M_n\subseteq \prod^n_{i=1}\Lambda^2_0$ be the full subcategory $\prod^n_{i=1}\left(\Delta^1\times\{0\}\right)$. Concretely in terms of the tuples description above, this is the full subcategory consisting of those tuples such that $a_i\in\{00,10\}$.
        \item Let $B_n\subseteq \prod^n_{i=1}\Lambda^2_0$ be the full subcategory $\left(\prod^n_{i=1}\Lambda^2_0\right)\backslash \{(10,\ldots,10)\}$, i.e. of those tuples where some of the $a_i$'s must either be $00$ or $01$.
        \item Let $F_n\subseteq B_n$ be the full subcategory of those tuples where $a_i\in\{01,10\}$, that is, where none of the $a_i$'s are $00$ and at least one of them is $01$. 
        \item Let $J_n\coloneqq M_n\cap B_n$, that is, where $a_i \in \{00,10\}$ but not all are $10$.
    \end{itemize}
    Similarly, since for any map $w\colon W\rightarrow V$ in $\NN$,  $w_*\underline{\Lambda}^2_0$ is fibrewise just products of multiple copies of $\Lambda^2_0$, we may similarly define in a fibrewise fashion the $V$--full subcategories $\underline{J}_w, \underline{F}_w, \underline{B}_w,\underline{M}_w\subseteq w_*\underline{\Lambda}^2_0$. Here, $\underline{J}_w,\underline{B}_w, \underline{M}_w$ are clearly $V$--subcategories, and $\underline{F}_w$ is too  by the exact same argument as in \cite[Cons. 3.2.3]{kaifNoncommMotives}. 
    Moreover, the  $\Cat[\baseCat]$-square \begin{equation}\label{obs:pushoutDecompositionOfHornPowers}
      \begin{tikzcd}[row sep=scriptsize]
            \underline{J}_w \rar[hook]\dar[hook]\ar[dr,phantom,very near end, "\ulcorner"] & \underline{M}_w\dar[hook]\\
            \underline{B}_w\rar[hook] & w_*\underline{\Lambda}^2_0
        \end{tikzcd}
    \end{equation}
   is a pushout; 
   indeed pushouts may be checked fibrewise, which is \cite[Obs. 3.2.4]{kaifNoncommMotives}.
\end{construction}

It turns out that colimits over $\udl{B}_w$ may be further simplified as follows:

\begin{lemma}\label{lem:F_fB_fAdjunction}
    The inclusion of $\baseCat$--categories $\underline{F}_w\subseteq \underline{B}_w$ is cofinal, i.e. restriction along it induces a natural equivalence $\colim_{\underline{B}_w}\prn{-}\xrightarrow{\;\;\; \sim \;\;\;} \colim_{\underline{F}_w}\prn{-}$.
\end{lemma}
\begin{proof}
    Parametrised cofinality is checked fibrewise \cite[Thm. 6.7]{Sha23}, and the fibrewise statement is given by \cite[Lem. 3.2.5]{kaifNoncommMotives}.
\end{proof}

We now come to a decomposition of parametrised cubical colimits.
The argument works verbatim as in \cite[Prop. 3.2.8]{kaifNoncommMotives}, but since we will need the result in the extra generality of the present setting, we provide the proof here for the reader's convenience.
\begin{proposition}\label{prop:norming_pushouts}
    Let $w\colon W\rightarrow V$ be an $\NN$-admissible map of orbit-finite $\TT$-\animae{}, $\CC$ a $W$-category, $\udl{\RR}$ a $W$-class of cocones in $\CC$, and $\varphi\colon w_*\udl{\sC}\rightarrow\udl{\D}$ a $w$-multifunctor distributing $\udl{\R}$-cocones. 
    Given an $\udl{\RR}$-diagram in $\partial \colon (\udl{\Lambda}^2_0)\tcocone\rightarrow \udl{\sC}$ as in the left square of 
    \[
      \begin{tikzcd}[row sep=scriptsize]
        X \rar \dar& Y \dar\\
        Z\rar & D
    \end{tikzcd}
    \hspace{15mm}
    \begin{tikzcd}
        \colim_{\underline{J}_w}\varphi\partial\rar\ar[dr,phantom, "\ulcorner", very near end] \dar&  \varphi Y\dar\\
        \colim_{\underline{F}_w}\varphi\partial\simeq\colim_{\underline{B}_w}\varphi\partial\rar & \varphi D,
    \end{tikzcd}
    \]
    the righthand square is a pushout in $\DD^V$.
    
    If moreover $\udl{\D}$ is pointed, $\udl{\RR}$ is projectable, and $Z\in\udl{\sC}^V$ is the cocone point of an empty $\RR$-cocone, then $\colim_{\udl{B}_w}F\partial \simeq  0_{\DD} \in \DD^V$, so we have a $\DD^V$-cofibre sequence 
    \[
        \colim_{\underline{J}_w}\varphi\partial\rightarrow \varphi Y\rightarrow \varphi D.
    \]
\end{proposition}
\begin{proof}
    By applying the colimit decomposition from \cite[Cor. 2.2.10]{kaifNoncommMotives}  on the pushout decomposition of $w_*\udl{\Lambda}^2_0$ from \eqref{obs:pushoutDecompositionOfHornPowers}, we obtain the pushout diagram in $\underline{\D}^V$
    \begin{center}
      \begin{tikzcd}[row sep=scriptsize]
            \colim_{\underline{J}_w}\varphi\partial \rar\dar\ar[dr, phantom, very near end, "\ulcorner"] & \colim_{\underline{M}_w}\varphi \partial\simeq \varphi Y\dar\\
            \colim_{\underline{B}_w}\varphi\partial \simeq \colim_{\underline{F}_w}\varphi\partial \rar & \colim_{w_*\udl{\Lambda}^2_0}\varphi \partial\simeq \varphi D.
        \end{tikzcd}
    \end{center}
    The bottom left equivalence is  \Cref{lem:F_fB_fAdjunction} and the bottom right equivalence is the statement that $\varphi$ distributes $\partial$.
    The top right equivalence follows since $\underline{M}_w$ has a final $V$-object whose $\varphi \delta$-image is $\varphi Y$.

    Assume the later assumptions.
    The restricted $V$-functor $\underline{F}_w\subseteq w_*\underline{\Lambda}^2_0\xrightarrow{\varphi \partial}\underline{\D}$ has constant value the zero object by \Cref{lem:distributing_over_initial_object} since, for every $\NN$-admissible $u\colon U \rightarrow V$, every point $\udl{\ast}\rightarrow u^*\udl{F}_w\subseteq u^*w_*\udl{\Lambda}^2_0$ picks out $Z$ in some coordinate of $u^*w_*\udl{\sC}$ by definition of $\udl{F}_w$;
    the result follows by noting that colimits of initial diagrams are initial.
\end{proof}



We will  use this to study saturations and a version of $F$-equivalences for \textit{multifunctors} (c.f. \cite[Def.~3.20]{nardinThesis}, of which this is a modification): 

\begin{definition}\label{defn:F-equivalences}
    Let $w\colon W \rightarrow V$ be an $\NN$-admissible map of orbit-finite $\baseCat$-\animae{} and  $F\colon w_*\udl{\sC}\rightarrow\DD$ a $w$-multifunctor. 
    Given a map of orbit-finite $\baseCat$-\animae{} $t\colon T \rightarrow W$, we say that a $\CC^T$-morphism $f$ lies in 
    the \mdef{$W$-collection $\udl{S}$ of $F$-equivalences} if, for every diagram of orbit-finite $\baseCat$-\animae
    \begin{equation}\label{eq:F-equiv-eqn}
    \begin{tikzcd}[ampersand replacement=\&, column sep=large, row sep=scriptsize]
	\emptyset \& A \& T \\
	C \& Z \& W \\
	\& U \& V
	\arrow["{!}", from=1-1, to=1-2]
	\arrow["{!}", from=1-1, to=2-1]
	\arrow["a", from=1-2, to=1-3]
	\arrow["s", from=1-2, to=2-2]
	\arrow["t", from=1-3, to=2-3]
	\arrow["c", from=2-1, to=2-2]
	\arrow["\lrcorner"{anchor=center, pos=0.125, rotate=180}, draw=none, from=2-2, to=1-1]
	\arrow["{\overline{u}}", from=2-2, to=2-3]
	\arrow["{\overline{w}}", from=2-2, to=3-2]
	\arrow["\lrcorner"{anchor=center, pos=0.125}, draw=none, from=2-2, to=3-3]
	\arrow["w", from=2-3, to=3-3]
	\arrow["u", from=3-2, to=3-3]
    \end{tikzcd}
    \end{equation}
    with upper left square a pushout and lower right square a pullback and every $C$-object $x \in \CC^C$, the map $(a^*f,\id_x)$ in $\udl{\sC}^A\times \udl{\sC}^C \simeq (w_* \CC)^U$ is sent to an equivalence under $F^U$.
\end{definition}
We leave it to the reader to verify that this is indeed restriction-stable.
Informally, a $T$-morphism $f$ is an $F$-equivalence precisely if, whenever a restriction of $f$ can be padded with identities to a morphism $f'$ in the domain of a restriction of $F$, $F(f')$ is an equivalence.
The following is crucial.

\begin{proposition} \label{lem:saturated-class}
    Let $w\colon W \rightarrow V$ be an $\NN$-admissible map of orbit-finite $\TT$-\animae{} and $F \colon w_\ast \CC \to \DD$ a distributive $w$-multifunctor. The $W$-collection $\udl{S}$ of $F$-equivalences is strongly saturated.
\end{proposition}
\begin{proof}
    The 2-out-of-3 axiom (3) from \Cref{defn:GSaturationAxioms} is easy to see.
    To check the colimit-closure axiom (2), fix a diagram as in \eqref{eq:F-equiv-eqn} and let $\partial\colon \udl{K}\rightarrow T^*\udl{{S}} \subset \udl{\func}(\udl{\Delta}^1,T^*\udl{\sC})$. We need to show that the colimit lies again in $T^*\udl{S}$, i.e. the morphism $(a^*\colim_{a^*\udl{K}}\partial,\id_x)$ in $\CC^A \times \CC^C\simeq u^*w_*\udl{\sC}$ is sent to an equivalence by $u^*F$.

    To this end, consider the $Z$-category $(a^*\udl{K},\udl{\ast})$ and the $Z$-diagram 
    \[
    (a^*\partial,\id_x)\colon (a^*\udl{K},\udl{\ast})\longrightarrow \udl{\func}(\udl{\Delta}^1,Z^*\udl{\sC})=\big(\udl{\func}(\udl{\Delta}^1,A^*\udl{\sC}),\udl{\func}(\udl{\Delta}^1,C^*\udl{\sC})\big).
    \]
    We acquire a composite diagram 
    \[\widetilde{F}\colon \overline{w}_*(a^*\udl{K},\udl{\ast})\xlongrightarrow{w_*(\partial,\id_x)} \overline{w}_*\udl{\func}(\udl{\Delta}^1,Z^*\udl{\sC})\simeq \udl{\func}(\udl{\Delta}^1,\overline{w}_*Z^*\udl{\sC})\xlongrightarrow{u^*F}\udl{\func}(\udl{\Delta}^1,U^*\udl{\D}).\] Since $u^*F\colon \udl{\func}(\udl{\Delta}^1,\overline{w}_*Z^*\udl{\sC})\rightarrow\udl{\func}(\udl{\Delta}^1,U^*\udl{\D})$ is $w$-distributive, we get an equivalence
    \[F\left(\colim_{(a^*\udl{K},\udl{\ast})}(a^*\partial,\id_x)\right)\simeq \colim_{\overline{w}_*(a^*\udl{K},\udl{\ast})}\widetilde{F}\in\udl{\func}(\udl{\Delta}^1,Z^*\udl{\D}).\] But then by assumption on $\partial$, $\widetilde{F}$ takes values in arrows which are equivalences in $Z^*\udl{\D}$. Since colimits of equivalences are equivalences, we thus see that the morphism $F\left(\colim_{(a^*\udl{K},\udl{\ast})}(a^*\partial,\id_x)\right)\colon \udl{\Delta}^1\rightarrow Z^*\udl{\D}$ is also an equivalence.
    To finish, note that cocompleteness of $\GCat$ provides a Beck-Chevalley natural equivalence interchanging colimits with $a^*$;
    in particular, we acquire an equivalence of arrows
    \[
        \colim_{(a^*\udl{K},\udl{\ast})}(a^*\partial,\id_x)\simeq \prn{ a^* \colim_{\udl{K}} \partial, \colim_{\udl{*}} \id_X} \simeq \prn{ a^* \colim_{\udl{K}} \partial, \id_X},
    \]
    and the colimit-closure axiom (2) follows.
    
    Finally, to verify the pushout axiom (1) of \Cref{defn:GSaturationAxioms}, again suppose we are given a morphism $t\colon T \rightarrow W$ and  suppose we have a pushout diagram $\partial \colon \udl{\Delta}^1\times\udl{\Delta}^1 = (\udl{\Lambda}^2_0)\tcocone\rightarrow T^*\udl{\sC}$
    \[
    \begin{tikzcd}[row sep=scriptsize]
        00 \rar\dar & 10\dar \ar[drr, phantom, "\mapsto"]&& y \rar\dar \ar[dr, phantom, "\ulcorner", very near end]& c\dar\\
        01 \rar& 11 && z\rar & d
    \end{tikzcd}
    \]
    where $y\rightarrow z$ is in $\udl{{S}}^T$. We need to argue that $c\rightarrow d$ is in $\udl{{S}}^T$ too.  For this again fix the notation of \eqref{eq:F-equiv-eqn}  and $x\in \udl{\sC}^C$. We need to show that $(a^*c,x)\rightarrow (a^*d,x)$ in $\CC^A \times \CC^C\simeq (u^*w_*\udl{\sC})^U$ is sent to an equivalence by $u^*F$.
    
    Consider the pushout diagram $(\partial,x)\colon (\udl{\Lambda}^2_0,\udl{\ast})\tcocone\rightarrow (A^*\udl{\sC},C^*\udl{\sC})=Z^*\udl{\sC}$. We write
    \[
        \delta\colon 
        \prn{\overline{w}_*\prn{\udl{\Lambda}^2_0,\udl{\ast}}}\tcocone
            \xrightarrow{\;\;\; \canonical \;\;\;} 
        \overline{w}_*\prn{\prn{\udl{\Lambda}^2_0,\udl{\ast}}\tcocone}
            \xrightarrow{\;\;\;\overline{w}_*(\partial,x)\;\;\;} 
        \overline{w}_*Z^*\udl{\sC}
            \xlongrightarrow{\;\;\;u^*F\;\;\;}
        U^*\udl{\D}.
    \] 
    Note that $\delta$ is a colimit diagram by $w$-distributivity of $F$.
    Also note that,  $\overline{w}_*(\udl{\Lambda}^2_0,\udl{\ast})\simeq s_*\udl{\Lambda}^2_0$. where $s\colon A \rightarrow U$ is defined in \eqref{eq:F-equiv-eqn}. 
    \Cref{prop:norming_pushouts} yields a pushout
    \[
      \begin{tikzcd}[row sep=scriptsize]
        \colim_{\udl{J}}\delta \rar\dar\ar[dr, phantom, "\ulcorner", very near end] & \delta_{(10,10,\ldots,10)}\simeq F(a^*c,x)\dar\\
        \colim_{\udl{B}}\delta \rar & \delta_{(11,11,\ldots,11)}\simeq F(a^*d,x).
    \end{tikzcd}
    \]
    where $\udl{J} \subset  \udl{B} \subset \overline{w}_*(\udl{\Lambda}^2_0,\udl{\ast})\simeq s_*\udl{\Lambda}^2_0$ are nested $U$-subposets.
    We must show that the right vertical map is an equivalence, for which it suffices to show that the left vertical map is an equivalence.

    Writing $i\colon \udl{J}\hookrightarrow \udl{B}$ for the inclusion, it sufffices to verify that the counit transformation \begin{equation}\label{eqn:left_kan_extension_subcube}
        i_!i^*\delta\rightarrow\delta
    \end{equation}  
    of functors $\udl{B}\rightarrow U^*\udl{\D}$ is an equivalence, as then we will acquire the desired equivalence
    \[
        \colim_{\udl{J}}\delta\simeq J_!i^*\delta\simeq B_!i_!i^*\delta\xrightarrow{\;\;\; \sim \;\;\;} B_!\delta\simeq \colim_{\udl{B}}\delta.
    \] 

    Since $i\colon\udl{J}\hookrightarrow\udl{B}$ is fully faithful, it suffices to prove that \eqref{eqn:left_kan_extension_subcube} is an equivalence on points in $\udl{B}\backslash\udl{J}$. 
    Let $d\colon D \rightarrow U$ be an orbit map and $b\in\udl{B}^D\backslash\udl{J}^D$. 
    Give a distinguished enumeration of the orbits of $Z'\coloneqq D \times_U A  \simeq \coprod_{i =1}^m Z'_i$.
    We acquire from this a double coset formula 
    \[
      \udl{B}^D =\left(\prod^m_{i=1}\Lambda^2_0\right)\backslash\{(10,\ldots,10)\} \subset \prn{d^*s_* \udl{\Lambda}_0^2}^D
    \]
    so we may without loss of generality write 
    \[
      b=(\{00\}_K,\{01\}_L, \{10\}_M), \hspace{30pt} \cbr{1,\dots,m}\simeq K \sqcup L \sqcup M,
    \]
    where $L$ is nonempty (lest $b$ lie in $\udl{J}^D$).   
    Because $\udl{B}$ is a $\baseCat$-poset, the slice $D$-category $\udl{J}\downarrow b$ to compute the left Kan extension $i_!$ at $b$ is just the full $D$-subcategory $D^*\udl{J}_{/b}\subseteq D^*\udl{J}$. But then $D^*\udl{J}_{/b}$ clearly has a final $D$-object, namely $(\{00\}_K,\{00\}_L,\{10\}_M)$. 
    Hence, we get by the pointwise formula for Kan extensions that $(i_!i^*\delta)(b)\simeq \delta(\{00\}_K,\{00\}_L,\{10\}_M)$. 
    
    Now, let $s'\colon A' \rightarrow Z'$ be the summand consisting of orbits corresponding with indices $i$ such that $i \in L$ and $c'\colon C' \rightarrow Z'$ is its complement. 
    By definition, $\overline{d} \circ s'$ factors through $s$, essentially uniquely since $s$ is monic.
    Thus we acquire a commutative diagram
    \[\begin{tikzcd}[ampersand replacement=\&,column sep=large, row sep=scriptsize]
	\emptyset \& {A'} \& A \& T \\
	{C'} \& {Z'} \& Z \& W \\
	\& D \& U \& V
	\arrow[from=1-1, to=1-2]
	\arrow[from=1-1, to=2-1]
	\arrow["{a'}", from=1-2, to=1-3]
	\arrow["{s'}", from=1-2, to=2-2]
	\arrow["a", from=1-3, to=1-4]
	\arrow["s", from=1-3, to=2-3]
	\arrow["t", from=1-4, to=2-4]
	\arrow["{c'}"', from=2-1, to=2-2]
	\arrow["\lrcorner"{anchor=center, pos=0.125, rotate=180}, draw=none, from=2-2, to=1-1]
	\arrow["{\overline{d}}", from=2-2, to=2-3]
	\arrow["{\widetilde{w}}", from=2-2, to=3-2]
	\arrow["\lrcorner"{anchor=center, pos=0.125}, draw=none, from=2-2, to=3-3]
	\arrow["{\overline{u}}", from=2-3, to=2-4]
	\arrow["{\overline{w}}", from=2-3, to=3-3]
	\arrow["\lrcorner"{anchor=center, pos=0.125}, draw=none, from=2-3, to=3-4]
	\arrow["w", from=2-4, to=3-4]
	\arrow["d", from=3-2, to=3-3]
	\arrow["u", from=3-3, to=3-4]
\end{tikzcd}\]
    whose outer diagram is of the form \eqref{eq:F-equiv-eqn}.
    Moreover, evaluating \eqref{eqn:left_kan_extension_subcube} at $b$ gives the map
    \[
        \delta(\{00\}_K,\{00\}_L,\{10\}_M)\longrightarrow \delta(\{00\}_K,\{01\}_L,\{10\}_M)
    \] 
    which is the $\delta$-image of $L$-copies of the map $00\rightarrow 01$. 
    This is a morphism of the form  $F\prn{a^{\prime *}y,x} \rightarrow F\prn{a^{\prime *}z,x}$, for $x \in \CC^{C'}$, so it is an equivalence by our assumption that $y\rightarrow z$ lies in $\udl{{S}}^T$, completing the proof.
\end{proof}

\subsection{Distributivity and free cocompletions}\label{subsec:distributivity_free_cocompletions}
For $\AA$ a small $\TT$-category, we denote its Yoneda embedding as 
\[ \yoneda[\AA] \colon \AA \longrightarrow \udl{\presheaf}(\AA). \]

In \cite[Prop.~3.19]{nardinThesis}, Nardin provided $\yoneda[\AA]$ with a universal property for \textit{distributive mulifunctors.}
Our goal in this subsection is elaborate upon his proof of this result, since we need it in greater generality.
This will be given as \Cref{prop:presheaf-dist}. We begin by recording Nardin's key insight for the situation at hand.

\begin{lemma}\label{lem:distributive_is_LKE}
  Let $w\colon W \rightarrow V$ be an $\NN$-admissible map of orbit-finite $\TT$-\animae{}, $\AA$ a small $W$-category,  $\DD$ a cocomplete $V$-category, and  $F \colon w_*\udl{\presheaf}(\AA)\rightarrow \DD$ a $w$-distributive functor. Then $F$ is left Kan extended along $w_*\yoneda[\AA] \colon w_\ast \AA \to w_\ast \presheaf(\AA)$, i.e. the counit map $\varepsilon\colon (w_*\yoneda[\AA])_!(w_*\yoneda[\AA])^*F\rightarrow F$ is an equivalence. 
\end{lemma}
\begin{proof}
    That $\varepsilon$ is an equivalence can be checked pointwise, so fix $X\in \left(w_*\udl{\presheaf}(\AA)\right)^V$.
    Note that the colimit of $\overslice{\AA}{X} \to \AA \xhookrightarrow{\yoneda[\AA]} \udl{\presheaf}(\AA)$ is $X$ by the density of the Yoneda embedding. Hence, by distributivity of $F$, we can compute $F(X)$ as the colimit of
    \[\overslice{\left(w_\ast \AA\right)}{X} \simeq w_*(\overslice{\AA}{X}) \longrightarrow w_\ast \AA \xrightarrow{\, w_\ast \yoneda[\AA] \,} w_\ast \udl{\presheaf}(\AA) \xrightarrow{\, F \,} \DD\] where the first equivalence holds because $w_*$ commutes with pullbacks and taking arrow categories.
    This in turn is precisely  the pointwise formula to compute the value of the left Kan extension of $(w_\ast \yoneda[\AA])^{\ast} F$ along $w_\ast \yoneda[\AA]$ at $X$, and so $F$ was left Kan extended along $w_*\yoneda[\AA]$ as required. 
\end{proof}

\begin{construction} \label{con:index_prod_multifunctor_on_presheaves}
    Fix $w \colon W \to V$ be an $\NN$-admissible map of orbit-finite $\TT$-\animae{} and $\AA$ be a small $W$-category. We define the \mdef{indexed product multifunctor} \[
    \mu\colon w_*\udl{\presheaf}(\AA)\longrightarrow \udl{\presheaf}(w_*\AA) \,
\]
    directly as the mate of the composite functor 
    \[ w_*\AA\op\times w_*\udl{\presheaf}(\AA)\simeq w_*(\AA\op\times\udl{\presheaf}(\AA))\xrightarrow{w_*\ev} w_*W^*\udl{\Spaces}\xrightarrow{w_*}V^*\udl{\Spaces} \,. \] 
\end{construction}
 
\begin{lemma} \label{lem:fast-left-fibration}
    Let $f \colon T \to S$ be a map of orbit-finite $\TT$-\animae{}. and $\pi\colon \DD \to \CC$ a left fibration of $T$-categories with straightening $T$-functor $F\colon \CC \to T^*\GSpaces$.  
    Then $f_* \pi\colon f_\ast \DD \to f_\ast \CC$ straightens to the composite $F\colon f_\ast \CC \to f_\ast T^*\GSpaces \simeq f_* f^* S^* \xrightarrow{\text{limit}} S^*\GSpaces$.
\end{lemma}
\begin{proof}
    We first verify the claim for the universal left fibration $T^*\GBSpaces \simeq \underslice{T^* \GSpaces}{\udl{\ast}} \to T^*\GSpaces$. Applying $f_\ast$ yields the left fibration $f_\ast \underslice{\prn{T^*\GSpaces}}{\udl{\ast}} \simeq  \underslice{\prn{f_\ast T^*\GSpaces}}{f_\ast \udl{\ast}}\simeq \underslice{\prn{f_\ast T^*\GSpaces}}{\udl{\ast}} \to f_\ast T^*\GSpaces$. This straightens to the functor taking maps out of the terminal object in $f_\ast T^* \GSpaces$ which indeed represents taking the limit.

    Now an arbitrary left fibration is pulled back from the universal left fibration along its straightening by \cite[Cor.~3.3.2.7]{lur09}, and the unstraightening of left $\baseCat$-fibrations are left fibrations, so the mapping-in property of coefficient systems yields the lefthand pullback square of $\baseCat$-categories
\[\begin{tikzcd}[ampersand replacement=\&, row sep=small]
	\DD \& \GBSpaces \&\& {\int \DD} \& \BSpaces \\
	\CC \& \GSpaces \&\& {\int \CC} \& \Spaces
	\arrow[from=1-1, to=1-2]
	\arrow[from=1-1, to=2-1]
	\arrow[""{name=0, anchor=center, inner sep=0}, from=1-2, to=2-2]
	\arrow[from=1-4, to=1-5]
	\arrow[""{name=1, anchor=center, inner sep=0}, from=1-4, to=2-4]
	\arrow["\lrcorner"{anchor=center, pos=0.125}, draw=none, from=1-4, to=2-5]
	\arrow[from=1-5, to=2-5]
	\arrow[from=2-1, to=2-2]
	\arrow[from=2-4, to=2-5]
	\arrow[between={0.2}{0.8}, squiggly, tail reversed, no head, from=0, to=1]
\end{tikzcd}\]
  This pullback is preserved by $f_\ast$, proving the claim.
\end{proof}

This yields the following basic relationship between two natural choices for a Yoneda embedding of indexed products, which is a coherent formulation of the fact that mapping \animae{} of indexed products are computed as the indexed products of mapping \animae{}.

\begin{corollary}\label{cor:triangle_of_indexed_yoneda} 
    Let $w\colon W \rightarrow V$ be an $\NN$-admissible map of orbit-finite $\TT$-\animae{} and $\AA$ a small $W$-category. Then the following triangle commutes:
    \[
    \begin{tikzcd}
      w_*\AA \ar[dr, "{\yoneda[\AA]}",hook]\dar["{w_*\yoneda[\AA]}"',hook]\\
        w_*\udl{\presheaf}(\AA) \rar["\mu"]& \udl{\presheaf}(w_*\AA).
    \end{tikzcd}
    \]
\end{corollary}
\begin{proof}
    Consider the following commuting diagram
    \[
    \begin{tikzcd}
      w_*\AA\op\times w_*\AA \ar[rr,"{\myuline{\Map}_{w_*\AA}}"]\ar[d, hook, "{\id\times w_*\yoneda[\AA]}"']\ar[drr, "{w_*\myuline{\Map}_{\AA}}"description] && V^*\udl{\Spaces}\\
        w_*\AA\op\times w_*\udl{\presheaf}(\AA) \ar[rr,"w_*\ev"'] && w_*W^*\udl{\Spaces}\ar[u, "w_*"']
    \end{tikzcd}
    \]
    where the bottom left triangle commutes by definition of the Yoneda embedding.
    Adjointing the $w_*\AA\op$ over to the upper right corner will  yield the desired triangle once we know that the upper right triangle commutes.
    By \Cref{lem:fast-left-fibration}, it suffices to provide an equivalence between $w_*$ applied to the $W$-categorical unstraightening of $\myuline{\Map}_{\AA}$ and the $V$-categorical unstraightening of $\myuline{\Map}_{w_* \AA}$.
    
    To see this, by the discussion after \cite[Thm.~11.5]{Sha23}, we know that $\myuline{\Map}_{\AA}$ unstraightens to the (fibrewise) twisted arrow category, i.e. the functor
    \[
        \mathrm{Tw}\colon\baseCat\op \longrightarrow \Fun(\Delta^1, \Cat)
    \]
    valued on $V \in \baseCat$ by the twisted arrow category left-fibration $\mathrm{Tw}\prn{\CC^V} \rightarrow \CC^{V, \mathrm{op}} \times \CC^V$; the desired equivalence is then just the base change formula of \eqref{eqn:standard_basecat_pullback}.
\end{proof}

\begin{lemma}
  \label{lem:special_case_yoneda_embedding_distributive}
    The $w$-multifunctor $\mu \colon w_*\udl{\presheaf}(\AA) \rightarrow \udl{\presheaf}(w_*\AA)$ from \Cref{con:index_prod_multifunctor_on_presheaves} is $w$-distributive. Furthermore, it is  the left Kan extension of $\yoneda[w_*\AA]\colon w_*\AA\hookrightarrow \udl{\presheaf}(w_*\AA)$ along the functor $w_*\yoneda[\AA]\colon w_*\AA\hookrightarrow w_*\udl{\presheaf}(\AA)$.
\end{lemma}
\begin{proof}
    Since distributivity is a colimit condition on the target category $\udl{\presheaf}(w_*\AA)$, and a diagram here may be checked to be a colimit diagram by evaluating at all points of $w_*\AA$, we may as usual without loss of generality work at level $V$ and verify that
    \[ w_\ast\udl{\presheaf}(\AA) \xlongrightarrow{\mu} \udl{\presheaf}(w_*\AA) \xrightarrow{\, \ev_A \,} V^*\udl{\Spaces} \] 
    is distributive for all $A \in (w_*\AA)^V$. We claim that this is equivalent to 
    \[ w_\ast \udl{\presheaf}(\AA) \xrightarrow{\, w_\ast ( \ev_A) \,} w_\ast W^*\udl{\Spaces} \xlongrightarrow{w_*} V^*\udl{\Spaces} \,. \]
    This latter $w$-multifunctor is $w$-distributive by \Cref{lem:postcomposing_distributive_with_colimit_pres}, as $\ev_A$ preserves colimits and $w_*\colon  w_\ast W^\ast \udl{\Spaces} \to V^*\udl{\Spaces}$ is $w$-distributive by \Cref{prop:distributivity_of_cartesian_symmon};
    hence we are now reduced to proving the claim.

    To prove the claim, note that the datum of a point $A\colon \udl{\ast}\rightarrow w_*\AA$ is equivalent to a datum of a point $\overline{A}\colon \udl{\ast}\rightarrow \AA$ by the $w^*\dashv w_*$ adjunction. By unwinding the identifications, this means that $A$ is equivalent to the datum $w_*\overline{A}\colon w_*\udl{\ast}\simeq \udl{\ast}\rightarrow w_*\AA$. Then, by construction of $\mu$, we have the following commuting diagram
    \[
    \begin{tikzcd}
        w_*\udl{\ast}\times w_*\udl{\presheaf}(\AA) \ar[rr,"w_*\overline{A}\times \id"]\dar["\id\times\mu"] && w_*\AA\op\times w_*\udl{\presheaf}(\AA)\rar["w_*\ev"] \dar["\id\times\mu"] & w_*W^*\udl{\Spaces} \dar["w_*"]\\
        w_*\udl{\ast}\times\udl{\presheaf}(w_*\AA)\ar[rr,"A\times \id"] && w_*\AA\op\times \udl{\presheaf}(w_*\AA) \rar["\ev"] & V^*\udl{\Spaces}.
    \end{tikzcd}
    \] The identification of the top-right  and the bottom-left composite yields the claim. The last statement is now an immediate combination of \Cref{cor:triangle_of_indexed_yoneda} and \Cref{lem:distributive_is_LKE}.
\end{proof}

We arrive at our generalisation of Nardin's distributive universal property for $\udl{\presheaf}$.
\begin{theorem}\label{prop:presheaf-dist} 
  Let $w\colon W \rightarrow V$ be an $\NN$-admissible map of orbit-finite $\TT$-\animae{}.
  Given $\AA$ a small $W$-category and $\CC$ a cocomplete $V$-category, precomposition with $w_*\yoneda[\AA]\colon w_*{\AA}\hookrightarrow w_*\udl{\presheaf}(\AA)$ yields an equivalence
      \[
        (w_*\yoneda[\AA])^*\colon \udl{\Fun}\distributeAll{w}\prn{w_*\udl{\presheaf}(\AA), \CC} \xlongrightarrow{\;\;\;\;\;\; \sim \;\;\;\;\;} \udl{\Fun}\prn{w_*\AA,\CC},
      \]
    whose inverse is given by left Kan extension.
\end{theorem}
\begin{proof}
  Since $w_*\yoneda[\AA]\colon w_*{\AA}\hookrightarrow w_*\udl{\presheaf}(\AA)$ is fully faithful and left Kan extensions along such are also fully faithful,  it is enough to identify the image of $\prn{w_* \yoneda[\AA]}_!$;
    that is, we need to show that any such left Kan extension is distributive and that any distributive $w$-multifunctor is left Kan extended along $w_*\yoneda[\AA]$. 
    The latter point is precisely supplied by \Cref{lem:distributive_is_LKE}, so we are left with addressing the former.

    For distributivity, fix $F\colon  w_\ast \AA \to \CC$ a $w$-multifunctor. 
    By the ordinary universal property of $\baseCat$-presheaves, $F$ factors through the Yoneda embedding $w_\ast \AA \hookrightarrow \udl{\presheaf}(w_*\AA)$ via a colimit-preserving $\baseCat$-functor $\overline{F}\colon \udl{\presheaf}(w_*\AA) \to \CC$. 
    Post-composition with this  preserves left Kan extension, so we have a commutative diagram 
    \[ \begin{tikzcd}
      w_\ast \AA \ar[r, "{\yoneda[\AA]}"] \ar[d, "{w_*\yoneda[\AA]}"'] & \udl{\presheaf}(w_*\AA) \ar[r, "\overline{F}"] & \CC \\
        w_*\udl{\presheaf}(\AA) \ar[ur, "\mu"] \ar[urr, bend right=10]
    \end{tikzcd} \]
    where the two diagonal functors are the left Kan extensions (the case of $\mu$ by \Cref{lem:special_case_yoneda_embedding_distributive}).     The lower  triangle expresses  $ w_*\udl{\presheaf}(\AA)  \to \CC$ as a composition of a $w$-distributive functor and a colimit-preserving functor, so \Cref{lem:postcomposing_distributive_with_colimit_pres} shows that it is $w$-distributive.
\end{proof}

From this, we may extract the following consequence.
\begin{corollary}\label{cor:distributive_yoneda_universal_property}
  Let $w \colon W \to V$ be an $\NN$-admissible map of orbit-finite $\TT$-\animae{},  $\AA$ be a small $W$-category, and $\CC$ a cocomplete $V$-category. 
    We have a commutative diagram of equivalences
    \[
    \begin{tikzcd}
      \udl{\func}(w_*\AA,\udl{\sC}) \ar[d, "{(w_*\yoneda[\AA])_!}"', "\simeq"]\ar[dr, "\simeq"', "{\yoneda[\AA!]}", shift left = 2 ]\\
        \udl{\func}\distributeAll{w}(w_*\udl{\presheaf}(\AA),\udl{\sC})& \udl{\func}^{\mathrm{colim}}(\udl{\presheaf}(w_*\AA),\udl{\sC}).\lar["\mu^*", "\simeq"']
    \end{tikzcd}
    \]
\end{corollary}
\begin{proof}
    By \Cref{cor:triangle_of_indexed_yoneda}, we clearly have the diagram when we replace $(w_*\yoneda[\AA])_!$ and $\yoneda[\AA!]$ with their right adjoints $(w_*\yoneda[\AA])^*$ and $\yoneda[\AA]^*$ respectively. 
    The map on the left is an equivalence by \Cref{prop:presheaf-dist}, and so is the one on the right by the universal property of the Yoneda embedding by \cite[Thm. 7.1.1]{MartiniWolf2024}. 
    We conclude that all maps are equivalences, and passing to inverses yields the conclusion.
\end{proof}

\subsection{Distributivity and relative cocompletions}\label{subsec:distributivity_general_cocompletions}
Our goal in this subsection is to prove \Cref{thm:distributivity_localisation-cocompletion} which gives a distributive universal property for partial cocompletions with relations. Just as for Lurie's construction of his presentable tensor product, the manoeuvre of Bousfield localisation will play an indispensable role in our treatment of $\NN$-normed structure on presentables. As such, we begin with some basic facts  about localisations.
\begin{lemma}\label{obs:bousfield_equivalences_in_indexed_products}
    If $L\colon \CC \rightleftharpoons \udl{\B} \cocolon R$ is a Bousfield localisation, $\udl{{S}}$ the class of $L$-equivalences, and $w \colon W \to V$ an $\NN$-admissible map between orbit-finite $\TT$-\animae{}, then in the Bousfield localisation $w_*L\colon w_*\CC\rightleftharpoons w_*\udl{\B} \cocolon w_*R$, the $w_*L$-equivalences are precisely $w_*\udl{{S}}$.
\end{lemma}
\begin{proof} 
      This is clear since fibrewise, the map $w_*L$ is just given by products of the maps $L$ at various levels by \Cref{lem:fibrewise_product}. 
\end{proof}

\begin{lemma}\label{lem:bousfield_localisations_and_distributivity}
    Let $L\colon \CC\rightleftharpoons \udl{\B} \cocolon R$ be a Bousfield localisation of cocomplete $W$-categories and $w \colon W \to V$ an $\NN$-admissible map between orbit-finite $\TT$-\animae{}. Suppose we have a $w$-multifunctor $\overline{F}\colon w_*\udl{\B}\rightarrow \udl{\D}$ such that $F\coloneqq \overline{F}\circ w_*L\colon w_*\CC\rightarrow\udl{\D}$ is $w$-distributive. Then $\overline{F}$ is also $w$-distributive.
\end{lemma}
\begin{proof}
    To see this, writing $j\colon \udl{I}\hookrightarrow \udl{I}\tcocone$ for the inclusion and letting $\overline{\partial}\colon \udl{I}\tcocone\rightarrow \udl{\B}$ be a colimit diagram, we know by \cite[Prop. 3.1.12]{kaifPresentable}  that $\overline{\partial}$ is  the image of $\partial $ under the composite
    \[ \udl{\func}(\udl{I},\udl{\B})\xhookrightarrow{R} \udl{\func}(\udl{I},\udl{\A}) \xrightarrow{j_!} \udl{\func}(\udl{I}\tcocone,\udl{\A}) \xrightarrow{L}\udl{\func}(\udl{I}\tcocone,\udl{\B})\] i.e. colimits in $\udl{\B}$ are localised from their $\udl{\A}$-colimits. This yields a commutative diagram
    \[
    \begin{tikzcd}
        (w_*\udl{I})\tcocone \rar & w_*(\udl{I}\tcocone) \rar["w_*\overline{\partial}"] \ar[dr, "w_*(j_!R\partial)"']& w_*\udl{\B} \rar["\overline{F}"] & \udl{\D}\\
        & & w_*\udl{\A}.\ar[ur, "F"']\uar["w_*L"description]
    \end{tikzcd}
    \] 
    By assumption, the bottom composite is a colimit, so the top is as well, as required.
\end{proof}
\begin{nota}
    Let $w\colon W \rightarrow V$ be an $\NN$-admissible map of orbit-finite $\TT$-\animae{} and $\CC$ a $W$-category with a $\baseCat$-collection of morphisms $\udl{S}$. 
    Denote its $w_*$-image as follows:
    \[
        w_*\udl{S}\subseteq w_*\udl{\func}(\udl{\Delta}^1,\CC)\simeq \udl{\func}(\udl{\Delta}^1,w_*\CC);
    \]
    Concretely, for a map $u\colon U \rightarrow V$, under the basechange equivalence $u^*w_*\simeq \prod_{i=1}^n\overline{w}_{i*}\overline{u}_i^*$ from \eqref{eqn:standard_basecat_pullback}, $(u^*w_*\udl{S})^U\simeq \prod_{i=1}^n(\overline{u}_i^*\udl{S})^{Z_i}$ is the collection of morphisms in $(u^*w_*\CC)^U\simeq \prod_{i=1}^n(\overline{u}_i^*\CC)^{Z_i}$ which  coordinate-wise lie in $\udl{S}$.

    For any $V$-category $\udl{\D}$, we will write 
    \[\udl{\func}_{w_*\udl{S}^{-1}}(w_*\CC,\DD)\subseteq \udl{\func}(w_*\CC,\udl{\D})\] for the full $\baseCat
    $-subcategory of \mdef{$w$-multifunctors which invert morphisms in $w_*\udl{S}$}, i.e. those which send morphisms in $w_*\udl{S}$ to equivalences in $\udl{\D}$. 
\end{nota}

From the saturation property of $F$-equivalences proved in the previous subsection, we obtain the following key result:
\begin{corollary}\label{lem:saturation_distribution}
    Let $w\colon W \rightarrow V$ be an $\NN$-admissible map between orbit-finite $\TT$-\animae{}. Let $\udl{{S}}$ be a $W$-class of morphisms in a $W$-category $\udl{\sC}$ containing all equivalences and $\udl{\overline{S}}$ its strong saturation. 
    Then for a cocomplete $V$-category $\DD$, a $w$-distributive functor 
    $w_\ast \CC \to \DD$ inverts  $w_\ast \udl{\overline{S}}$ if it inverts $w_\ast \udl{S}$;
    that is, 
    \[
    \udl{\func}\distributeAll{w}_{w_*\udl{\overline{S}}^{-1}}(w_*\udl{\sC},\udl{\D})=\udl{\func}\distributeAll{w}_{w_*\udl{S}^{-1}}(w_*\udl{\sC},\udl{\D}).
    \] 
\end{corollary}
\begin{proof}
    Suppose $F\colon w_*\CC\rightarrow\DD$ is a $w$-distributive multifunctor which inverts $w_*\udl{S}$. 
    It suffices to prove that $\udl{S}$ consists of $F$-equivalences, as \Cref{lem:saturated-class}, we see that $\overline{\udl{S}}$ is also contained within the $F$-equivalences.
    
    Fix a diagram as in \eqref{eq:F-equiv-eqn},  $f$ a map in $\udl{S}^T$, and $x \in \CC^C$.
    Because $\udl{S}$ contains all equivalences the resulting map $\prn{a^*f, \id_x}$ in $\CC^A \times \CC^C \simeq \CC^Z$ lies in $\prn{w_* S}^Z$, so by assumption $F\prn{a^* f, \id_x}$ is an equivalence; 
    hence $f$ is an $F$-equivalence, completing the proof.
\end{proof}

\begin{lemma}\label{lem:yoneda_image_generation_distributivity}
    Let  $w \colon W \to V$ be an $\NN$-admissible map of orbit-finite $\TT$-\animae, $\CC$ a cocomplete $W$-category, $\DD$ a cocomplete $V$-category, $\udl{\A}\subseteq \udl{\sC}$ a small $\baseCat$-subcategory which generates $\udl{\sC}$ under $V$-colimits, and $\udl{S}$ a $\baseCat$-set of morphisms in $\udl{\sC}$. Then the inclusion of $\udl{\func}\distributeAll{w}_{w_*(S\cup {\sC}^{\simeq})^{-1}}(w_*\udl{\sC},\udl{\D})$ into $ \udl{\func}\distributeAll{w}_{w_*(S\cup {\A}^{\simeq})^{-1}}(w_*\udl{\sC},\udl{\D})$ is an equivalence.
\end{lemma}
\begin{proof}
    Let $F\in \udl{\func}\distributeAll{w}_{w_*(S\cup {\A}^{\simeq})^{-1}}(w_*\udl{\sC},\udl{\D})$. We need to show that it also inverts morphisms in $w_*(\udl{S}\cup \udl{\sC}^{\simeq})$. Consider a map $u\colon U \rightarrow V$ so that we have a morphism in $\wGCat$
    \[u^*F \colon u^*w_*\udl{\sC}\simeq \overline{w}_*\overline{u}^*\udl{\sC}\longrightarrow u^*\udl{\D};\] 
    we need to show that $u^*F$ inverts  $\overline{w}_*\overline{u}^*(\udl{S}\cup \udl{\sC}^{\simeq})$. Let $(\udl{s}_i, \udl{\id}_j)\colon (\udl{A}_i, \udl{X}_j)\rightarrow (\udl{B}_i, \udl{X}_j)$ be such a tuple of morphisms in $\overline{w}_\ast \overline{u}^\ast \CC$. By hypothesis, for each $j$, we may write $\udl{\id}_j\colon \udl{X}_j\rightarrow \udl{X}_j$ as a parametrised colimit of identities in $\udl{\A}$. 
    By distributivity of $u^*F$ and the assumption that $F$ inverts $w_*(\udl{S}\cup \udl{\A}^{\simeq})$, we see that $u^*F$ also inverts $(\udl{s}_i,\udl{\id}_j)$.
\end{proof}

We may now obtain a distributive mapping out property for Bousfield localisations.
\begin{corollary}\label{cor:distributivity_bousfield_localisations}
    Let $L\colon \udl{\sC}\rightleftharpoons \udl{\B} \cocolon R$ be a Bousfield localisation of cocomplete $W$-categories, $\udl{\D}$ a cocomplete $V$-category, and $\udl{S}$ a class of morphisms in $\udl{\sC}$ containing all equivalences whose strong saturation $\overline{\udl{S}}$ coincides with the class of  $L$-equivalences. 
    
    The map $w_*L\colon w_*\udl{\sC}\rightarrow w_*\udl{\B}$ then induces an equivalence
    \[(w_*L)^*\colon \udl{\func}\distributeAll{w}(w_*\udl{\B},\udl{\D}) \xlongrightarrow{\;\;\;\; \sim \;\;\;\;} \udl{\func}\distributeAll{w}_{w_*\overline{S}^{-1}}(w_*\udl{\sC},\udl{\D})\simeq \udl{\func}\distributeAll{w}_{w_*S^{-1}}(w_*\udl{\sC},\udl{\D})\] with inverse given by left Kan extension along $w_*L$. 
\end{corollary}
\begin{proof}
    Since the adjunction $w_*L\colon w_*\udl{\sC}\rightleftharpoons w_*\udl{\B} \cocolon w_*R$ is a Bousfield localisation,  the restriction functor $(w_*L)^*\colon \udl{\func}(w_*\udl{\B},\udl{\D}) \hookrightarrow \udl{\func}(w_*\udl{\sC},\udl{\D})$ is fully faithful by \Cref{lem:Bousfield-loc-is-loc}.
    Moreover, since $w_*L$ preserves colimits, the restriction also restricts to a fully faithful functor $(w_*L)^*\colon \udl{\func}\distributeAll{w}(w_*\udl{\B},\udl{\D}) \hookrightarrow\udl{\func}\distributeAll{w}(w_*\udl{\sC},\udl{\D})$ by \Cref{lem:postcomposing_distributive_with_colimit_pres}.

    Finally, to see that $(w_*L)^*$ is essentially surjective onto $\udl{\func}\distributeAll{w}_{w_*\overline{S}^{-1}}(w_*\udl{\sC},\udl{\D})$, consider a $w$-distributive functor $F\colon w_*\udl{\sC}\rightarrow\udl{\D}$ which inverts morphisms in $w_*\udl{\overline{S}}$. 
    By \Cref{obs:bousfield_equivalences_in_indexed_products} and \Cref{lem:Bousfield-loc-is-loc}, we get a factorisation $F\colon w_*\udl{\sC}\xrightarrow{w_*L}w_*\udl{\B}\xrightarrow{\widetilde{F}}\udl{\D}$ where $\widetilde{F}\simeq (w_*L)_!F$. But since $F$ was $w$-distributive, we thus see that $\widetilde{F}$ is also $w$-distributive by \Cref{lem:bousfield_localisations_and_distributivity}, whence follows the desired essential surjectivity. 
    The second equivalence is immediate from \Cref{lem:saturation_distribution}.
\end{proof}

We now apply the preceding abstract results to connect Dwyer-Kan localisations with distributivity against a projectable class of cocones as introduced in \Cref{defn:projectable_classes}. As we shall see, the property of being projectable is precisely to arrange that the hypotheses of \Cref{lem:saturation_distribution} are satisfied.
\begin{proposition}\label{prop:distributivity_inverted_maps_presheaves}
    Let $(\CC,\udl{\R})$ be a small projectable class of cocones on a small category $\CC$, $\udl{S}$ be the collection of the colimit cocone comparison morphisms in $\udl{\presheaf}(\udl{\sC})$ of the form $\alpha\colon \colim_{\udl{K}}\yonedaMap \partial \rightarrow \yonedaMap(\partial_{\infty})$ for $\partial\colon \udl{K}\tcocone\rightarrow \udl{\sC}$  an $\RR$-cocone.
    Then $w_*\yonedaMap$ induces an equivalence
    \[(w_*\yonedaMap)^*\colon \udl{\func}\distributeAll{w}_{w_*\overline{\udl{S}}^{-1}}(w_*\udl{\presheaf}(\udl{\sC}),\udl{\D}) \xlongrightarrow{\simeq}\udl{\func}\distributiveLocalisation{w}{\RR}{}(w_*\udl{\sC},\udl{\D}),\] with inverse given by left Kan extension.
\end{proposition}
\begin{proof}
    By \Cref{cor:distributive_yoneda_universal_property}, restriction and left Kan extension are inverse equivalences
    \[(w_*\yonedaMap)_!\colon \udl{\func}(w_*\udl{\sC},\udl{\D})\rightleftharpoons \udl{\func}\distributeAll{w}(w_*\udl{\presheaf}(\udl{\sC}),\udl{\D}) \cocolon (w_*\yonedaMap)^*.\] 
    By \Cref{lem:saturation_distribution}, a $w$-distributive functor $w_\ast \udl{\presheaf}(\CC) \to \DD$ inverts $\overline{\udl{S}}$ if and only if it inverts $w_\ast (S \sqcup \udl{\presheaf}(\CC)^{\simeq})$; by \Cref{cor:distributive_yoneda_universal_property}, this follows if it inverts $w_\ast \udl{S}$, since $\udl{\R}$ contains all constant diagrams on $\udl{\ast}\tcocone$ by projectability, so $\udl{S}$ contains all equivalences on the generators $\udl{\sC}\subseteq \udl{\presheaf}(\udl{\sC})$. 
    The proposition is thus reduced to the following two claims:
    \begin{enumerate}[label=(\alph*)]
         \item $(w_*\yonedaMap)^*$ restricts to a functor $\udl{\func}\distributeAll{w}_{w_*\udl{S}^{-1}}(w_*\udl{\presheaf}(\udl{\sC}),\udl{\D}) \longrightarrow \udl{\func}\distributiveLocalisation{w}{\RR}{}(w_*\udl{\sC},\udl{\D})$;
         and
         \item $(w_*\yonedaMap)_!$ restricts to a functor $\udl{\func}\distributiveLocalisation{w}{\RR}{}(w_*\udl{\sC},\udl{\D}) \longrightarrow \udl{\func}\distributeAll{w}_{w_*\udl{S}^{-1}}(w_*\udl{\presheaf}(\udl{\sC}),\udl{\D})$.
    \end{enumerate}  

    For Claim (a),  let $F\colon w_* \udl{\presheaf}(\udl{\sC})\rightarrow \udl{\D}$ be a $w$-distributive functor which inverts morphisms in $w_\ast\udl{S}$.
    Fix a colimit $W$-diagram $\partial\colon \udl{K}\tcocone \rightarrow \udl{\sC}$ in $\udl{\R}$ with cocone point $\partial_\infty\colon \udl{*} \rightarrow \CC$.
    We must show that $F \circ w_* \yo$ distributes $\partial$.
    
    By definition, the assembly map $\alpha\colon \colim_{\udl{K}}\yoneda\circ \partial \rightarrow \yoneda\circ \partial_{\infty}$ is a morphism in $\udl{S}$. 
    Moreover, under the $w^*\dashv w_*$ adjunction, we may view $\alpha$ as the datum of a morphism in the collection $w_*\udl{\overline{S}}$ of morphisms in $w_*\udl{\presheaf}(\udl{\sC})$, so $F$ sends $\alpha$ to an equivalence in $\udl{\D}$. But by $w$-distributivity of $F$, we see that $F(\colim_{\udl{K}}\yoneda\circ \partial)\simeq \colim_{w_*\udl{K}}F\circ (w_*\yoneda)\circ (w_*\partial)$. This is precisely saying that $(w_*\yoneda)^*F\in \udl{\func}\distributiveLocalisation{w}{\RR}{}(w_*\udl{\sC},\udl{\D})$.
    
    For Claim (b),  suppose $F \colon w_*\underline{\sC} \rightarrow \underline{\D}$ is a functor that distributes $\udl{\R}$-diagrams. We want to show that $(w_*\yonedaMap)_!F \colon w_*\udl{\presheaf}(\udl{\sC}) \rightarrow \underline{\D}$ inverts maps in $w_*\udl{\overline{S}}$, i.e. those maps that get inverted by $w_*L$. Now, $(w_*\yonedaMap)_!F \colon w_*\underline{\presheaf}(\underline{\sC}) \rightarrow \underline{\D}$ is $w$-distributive by \Cref{cor:distributive_yoneda_universal_property}. Hence, by \Cref{cor:triangle_of_indexed_yoneda} and \Cref{cor:distributive_yoneda_universal_property}, we obtain a diagram 
    \[
      \begin{tikzcd}[row sep=scriptsize, column sep=large]
        w_*\udl{\sC}\ar[dr, "\yonedaMap" description] \dar["w_*\yonedaMap"']\ar[drr, "F" description, bend left = 20]\\
        w_*\udl{\presheaf}(\udl{\sC})\ar[rr, bend right = 15, "(w_*\yonedaMap)_!F"description]\rar["\mu" description] & \udl{\presheaf}(w_*\udl{\sC}) \rar["\exists!\widetilde{F}"]& \udl{\D}
    \end{tikzcd}
    \]
    where $\widetilde{F}$ preserves colimits and $\mu$ is $w$-distributive. Using this factorisation of $(w_*\yonedaMap)_!F$, we argue that $(w_*\yonedaMap)_!F$ sends $w_*\udl{S}$ to equivalences.

    Let $\partial\colon \udl{K}\tcocone\rightarrow\udl{\sC}$ be a diagram in $\udl{\R}$. Writing $\infty\in \udl{K}\tcocone$ for its cocone point, the associated morphism in $\udl{S}\subseteq \udl{\presheaf}(\udl{\sC})$ is then given by $\colim_{\udl{K}}\yonedaMap\partial\rightarrow\yonedaMap\partial_{\infty}$. We may view this as a morphism in $w_*\udl{\presheaf}(\udl{\sC})$ and we would like to show that $(w_*\yonedaMap)_!F$ sends this to an equivalence. By $w$-distributivity of $\mu$, we have that \[\mu(\colim_{\udl{K}}\yonedaMap\partial)\simeq \colim_{w_*\udl{K}}\mu w_*(\yonedaMap\partial)\simeq \colim_{w_*\udl{K}}\yonedaMap (w_*\partial)\] where the second equivalence is by the left triangle in the diagram above. Thus, we get
    \[(w_*\yonedaMap)_!F(\colim_{\udl{K}}\yonedaMap\partial)\simeq \widetilde{F}(\colim_{w_*\udl{K}}\yonedaMap (w_*\partial))\simeq \colim_{w_*\udl{K}}\widetilde{F}\yonedaMap (w_*\partial)\simeq \colim_{w_*\udl{K}}F(w_*\partial)\]
    where the second equivalence is since $\widetilde{F}$ preserves colimits and the third by the top triangle in the diagram above. By assumption on $F$, the last term is equivalent $F\partial_{\infty}\simeq ((w_*\yonedaMap)_!F)\circ(w_*\yonedaMap)(\partial_{\infty})$, as was to be shown.
\end{proof}
Next, we recall partial cocompletions with relations from \cite[Prop.~5.3.6.2]{lur09}, \cite[Thm.~6.4.2]{kaifPresentable} and \cite[App.~A.4]{mwpresentable}.

\begin{construction}\label{recollect:localisation_cocompletion}
    Let $\udl{\K}$ be a $\baseCat$-subcategory of $\udl{\cat}$. Let $\underline{\sC}$ be a large $\baseCat$-category together with a $\baseCat$-class of cocones $\udl{\RR}$ indexed by categories in $\KK$.
    We then write $\udl{\wpresheaf}(\udl{\sC})$ for the very large category of $\wGSpaces$-valued presheaves, i.e. the large-cocompletion of $\CC$, and $\udl{\wpresheaf}(\udl{\sC},\udl{\RR}) \subset \udl{\wpresheaf}(\udl{\sC})$ for the very large subcategory of functors $\CC{\op} \to \wGSpaces$ whose opposite functor sends $\RR$-cocones to colimit cocones; in other words, the local objects for the large class of morphisms of the form $\alpha\colon \colim_{\udl{K}}\yonedaMap \partial \rightarrow \yonedaMap(\partial_{\infty})$ for $\partial\colon \udl{K}\tcocone\rightarrow \udl{\sC}$  an $\RR$-cocone. By \Cref{lem:local-object-is-loc}, the inclusion admits a left adjoint. 
    
    We write $\udl{\wpresheaf}^{\K}(\udl{\sC},\udl{\RR}) \subset \udl{\wpresheaf}(\udl{\sC},\udl{\RR})$ for the full subcategory generated under $\udl{\K}$-indexed colimits by the full image of localised Yoneda $i\colon \udl{\sC}\xhookrightarrow{\yonedaMap}\udl{\wpresheaf}(\udl{\sC})\twoheadrightarrow \udl{\wpresheaf}(\udl{\sC},\udl{\RR})$.
\end{construction}
\begin{proposition}
    $\udl{\wpresheaf}^{\K}(\udl{\sC},\udl{\RR})$ is a large $\baseCat$-category.
\end{proposition}
\begin{proof}
    As the category in question is defined as a small colimit closure, it can be obtained from the large category $\CC$ by a transfinite induction through the large set of small cardinals, see e.g. \cite[Sec.~3.5]{kelly1982basic} for an argument. 
\end{proof}
We will also need the following easy observation.

\begin{lemma}\label{lem:localisation_extension_is_kan_extension}
    Let $L\colon \CC\rightleftharpoons \udl{\B} \cocolon R$ be a Bousfield localisation, $y\colon \udl{\A}\rightarrow\CC$ a $\baseCat$-functor, and $f\colon \udl{\A}\rightarrow \udl{\D}$ another $\baseCat$-functor such that $\udl{\D}$ is cocomplete. Write $i\coloneqq L\circ y\colon \udl{\A}\rightarrow\udl{\B}$. If $y_!f \colon \CC\rightarrow\udl{\D}$ has the property of factoring through the Bousfield local subcategory $\udl{\B}$ so that $y_!f\simeq \widetilde{y_!f}\circ L$ for some $\widetilde{y_!f}\colon \udl{\B}\rightarrow \udl{\D}$, then in fact $\widetilde{y_!f}\simeq i_!f$. 
\end{lemma}
\begin{proof}
    Let $G\in\udl{\func}(\udl{\B},\udl{\D})$. Then we have
    \[\udl{\nattrans}(\widetilde{y_!f},G) \simeq \udl{\nattrans}(y_!f,L^*G)\simeq \udl{\nattrans}(f, y^*L^*G)\simeq \udl{\nattrans}(f, i^*G) \simeq \udl{\nattrans}(i_!f,G)\]
    where the first equivalence uses that $L^*\colon \udl{\func}(\udl{\B},\udl{\D})\rightarrow \udl{\func}(\udl{\A},\udl{\D})$ is fully faithful.
\end{proof}
With these, we acquire a distributive universal property for cocompletion with relations.
\begin{corollary}\label{cor:distributivity_localised_presheaves}
    Let $w\colon W\rightarrow V$ be a map of orbit-finite $\TT$-\animae{}, $\CC$ a large $W$-category, $\udl{\RR}$ a projectable class of cocones in $\CC$, and $\DD$ a very large $\wGCat$-cocomplete category. The composite $w_* i\colon w_*\udl{\sC}\xrightarrow{w_*\yonedaMap}w_*\udl{\wpresheaf}(\udl{\sC})\xrightarrow{w_*L}w_*\udl{\wpresheaf}(\udl{\sC},\udl{\RR})$  induces an equivalence
    \[i^*\colon \udl{\func}^{\wGCat-\delta_w}(w_*\udl{\wpresheaf}(\udl{\sC},\udl{\RR}),\udl{\D}) \xlongrightarrow{\;\;\;\; \sim \;\;\;\;}\udl{\func}\distributiveLocalisation{w}{\RR}{}(w_*\udl{\sC},\udl{\D})\] with inverse given by left Kan extension along $w_*\udl{\sC}\rightarrow w_*\udl{\wpresheaf}(\udl{\sC},\udl{\R})$.
\end{corollary}
\begin{proof}
    Let $\udl{S}$ denote the class of morphisms inverted under the Bousfield localisation $\udl{\wpresheaf}(\udl{\sC}) \to \udl{\wpresheaf}(\udl{\sC},\udl{\R})$.
    Consider the composite
    \begin{align*}
        \udl{\func}^{\wGCat-\delta_w}(w_*\udl{\wpresheaf}(\udl{\sC},\udl{\R}),\udl{\D}) 
        \xrightarrow[\sim]{\;\;\;\;(w_*L)^*\;\;\;\;}& \udl{\func}^{\wGCat-\delta_w}_{w_*{S}^{-1}}(w_*\udl{\wpresheaf}(\udl{\sC}),\udl{\D})\\ \xrightarrow[\sim]{\;\;\;\;(w_*\yonedaMap)^*\;\;\;\;}&\udl{\func}\distributiveLocalisation{w}{\udl{\R}}{}(w_*\udl{\sC},\udl{\D})
    \end{align*}
    where the first equivalence is by \Cref{cor:distributivity_bousfield_localisations} and the second is by \Cref{prop:distributivity_inverted_maps_presheaves}. To see that the inverse to this composite is given by left Kan extension along $w_*\udl{\sC}\rightarrow w_*\udl{\wpresheaf}(\udl{\sC},\udl{\R})$, first recall that the inverse of $(w_*\yonedaMap)^*$ is given by the left Kan extension $(w_*\yonedaMap)_!$ by virtue of \Cref{prop:distributivity_inverted_maps_presheaves} again. On the other hand, by \Cref{lem:localisation_extension_is_kan_extension}, the associated $w$-distributive functor $w_*\udl{\wpresheaf}(\udl{\sC},\udl{\R})\rightarrow\udl{\D}$ is indeed given by left Kan extending along $w_*\udl{\sC}\xrightarrow{w_*\yonedaMap}w_*\udl{\wpresheaf}(\udl{\sC})\xrightarrow{w_*L}w_*\udl{\wpresheaf}(\udl{\sC},\udl{\R})$, as required.
\end{proof}

We are now ready to generalise to \textit{partial} cocompletion with relations.

\begin{theorem}\label{thm:distributivity_localisation-cocompletion}
Let $\udl{\K}\subset \udl{\cat}$ be a $\baseCat$-subcategory containing $\udl{\ast}$, $(\CC,\udl{\RR})$ be a projectable class of cocones indexed by diagrams in $\KK$ where $\CC$ is a large $W$-category, and $\underline{\D}$ be a large $w_*{\udl{\K}}$-cocomplete category. 
Then  $i \colon \underline{\sC} \rightarrow \underline{\wpresheaf}^{\K}(\underline{\sC},\udl{\R}) $ induces an equivalence 
\[
    (w_*i)^* \colon \underline{\func}\distributiveLocalisation{w}{\K}{}\prn{w_*\underline{\wpresheaf}^{\K}(\underline{\sC},\udl{\RR}), \underline{\D}} \xrightarrow{\;\;\;\; \sim \;\;\;\;} \udl{\func}\distributiveLocalisation{w}{\RR}{}(w_*\underline{\sC},\underline{\D})
    \]
    with inverse given by left Kan extending along $w_*i\colon w_*\udl{\sC}\rightarrow w_*\underline{\wpresheaf}^{\KK}(\underline{\sC},\udl{\RR})$. Furthermore, if $\udl{\R}$ consists of colimit diagrams in $\udl{\sC}$, then $i$ is fully faithful.
\end{theorem}
\begin{proof}
The second statement that $i$ is fully faithful when $\udl{\R}$ consists of colimit diagrams follows because representable functors send colimits to limits, so that they are contained in $\udl{\wpresheaf}(\CC,\udl{\R})$. 

Let us now turn towards the first statement. By the exact same argument as in the proof of \cite[Thm. 6.4.2 (2)]{kaifPresentable}, we may reduce to considering the case when $\udl{\D}$ is $\wGCat$-cocomplete by considering the opposite large Yoneda embedding.

We write $j\colon \udl{\mathcal E} \subseteq \udl{\wpresheaf}(\udl{\sC})$ for the preimage $L^{-1}\udl{\wpresheaf}^{\K}(\udl{\sC},\udl{\R})$ of $\udl{\wpresheaf}^{\K}(\udl{\sC},\udl{\R})$ under the Bousfield localisation $L\colon \udl{\wpresheaf}(\udl{\sC})\twoheadrightarrow \udl{\wpresheaf}(\udl{\sC},\udl{\R})$. 
This ambient Bousfield localisation restricts to a Bousfield localisation $L\colon \udl{\mathcal E}\twoheadrightarrow \udl{\wpresheaf}^{\K}(\udl{\sC},\udl{\R})$;
let $\udl{S}$ be the collection of morphisms inverted by this functor. 
We also write $y\colon \udl{\sC}\hookrightarrow \udl{\mathcal E}$ for the inclusion. We acquire a composite
\begin{align*}
    (w_*i)^*\colon \underline{\func}\distributiveLocalisation{w}{\K}{}\prn{w_*\udl{\wpresheaf}^{\KK}(\CC,\udl{\R}),\underline{\D}} \xrightarrow[\sim]{\;\;\; L^* \;\;\;}
    &\underline{\func}\distributiveLocalisation{w}{\K}{w_*\udl{S}^{-1}}\prn{w_*\underline{\mathcal E}, \underline{\D}}\\
    \xlongrightarrow{\;\;(w_*y)^*\;\;}& \udl{\func}\distributiveLocalisation{w}{\RR}{}(w_*\underline{\sC},\underline{\D})
\end{align*}
where the first equivalence is by \Cref{cor:distributivity_bousfield_localisations}, with inverse given by left Kan extension along $L$. Our goal is to prove that the second map is an equivalence. Since $w_*y\colon w_*\udl{\sC}\hookrightarrow w_*\udl{\mathcal E}$ is fully faithful, it suffices to argue that $(w_*y)_!\colon \udl{\func}\distributiveLocalisation{w}{\RR}{}(w_*\underline{\sC},\underline{\D})\hookrightarrow \underline{\func}(w_*\udl{\mathcal E}, \underline{\D})$ lands in $\underline{\func}\distributiveLocalisation{w}{\K}{w_*\udl{S}^{-1}}(w_*\udl{\mathcal E}, \underline{\D})$ and that the adjunction counit $(w_*y)_!(w_*y)^*\rightarrow \id$ in $\underline{\func}\distributiveLocalisation{w}{\K}{w_* \udl{S}}(w_*\udl{\mathcal{E}}, \underline{\D}) $ is an equivalence.

To see that $(w_*y)_!$ lands appropriately,  let $F\in \udl{\func}\distributiveLocalisation{w}{\RR}{}(w_*\underline{\sC},\underline{\D})$ and consider 
    \begin{center}
      \begin{tikzcd}[row sep=scriptsize, column sep=large]
        &w_*\underline{\sC} \ar[dl, "w_*y"', bend right = 15, hook]\dar[hook, "w_*\yonedaMap"'] \ar[drr, "F", bend left = 15]\\
        w_*\udl{\mathcal E}\dar["L", two heads]\rar[hook, "w_*j"]&w_*\underline{\wpresheaf}(\underline{\sC}) \ar[rr, "(w_*\yonedaMap)_!F"'] \dar[two heads, "L"'] && \underline{\D} \\
        w_*\udl{\wpresheaf}^{\K}(\udl{\sC},\udl{\R})\ar[r, hook]&w_*\underline{\wpresheaf}(\underline{\sC},\udl{\R})  \ar[urr, dashed, bend right = 20]
        \end{tikzcd}
    \end{center}
    where the dashed arrow exists and is $\prn{\wGCat,w}$-distributive by \Cref{cor:distributivity_localised_presheaves}.  
    By \Cref{lem:postcomposing_distributive_with_colimit_pres}, precomposing along the inclusion $w_*\udl{\E} \hookrightarrow w_*\underline{\wpresheaf}(\underline{\sC})$ preserves $w$-distributed $\K$-colimit diagrams, and so the restriction 
    \[ 
        (w_*j)^*\colon \udl{\func}\prn{w_*\udl{\wpresheaf}(\udl{\sC}),\udl{\D}}\longrightarrow\udl{\func}\prn{w_*\udl{\mathcal E},\udl{\D}}
    \] 
    restricts to a functor 
    \[
        \udl{\func}^{\wGCat-\delta_w}\prn{w_*\udl{\wpresheaf}(\udl{\sC}),\udl{\D}}\longrightarrow\udl{\func}\distributiveLocalisation{w}{\K}{}\prn{w_*\udl{\mathcal{E}},\udl{\D}} \,. 
        \] 
        Furthermore, the dashed factorisation and bottom-left commuting square demonstrate
    \[
        (w_*j)^*(w_*\yonedaMap)_!F \in \udl{\func}\distributiveLocalisation{w}{\K}{w_*\udl{S}^{-1}}\prn{w_*\udl{\mathcal E},\udl{\D}}.
    \]
     But then $(w_*j)^*(w_*\yonedaMap)_!F\simeq (w_*y)_!F$ since $w_*j$ was fully faithful, so  $(w_*y)_!F$  lies indeed in $ \udl{\func}\distributiveLocalisation{w}{\K}{w_*\udl{S}^{-1}}\prn{w_*\udl{\mathcal E},\udl{\D}}$ as required.

    Next, let $F\colon w_*\udl{\mathcal E} \to \DD$ be $\KK$-distributive functor which inverts $w_\ast \udl{S}$. We  show that  
    \[
        \epsilon\colon(w_*y)_!(w_*y)^*F\longrightarrow F
    \]
    is an equivalence. The left-hand side also inverts $w_\ast \udl{S}$ by the previous point. 
    We use the equivalence $\underline{\func}\distributiveLocalisation{w}{\K}{}\prn{w_*\udl{\wpresheaf}^{\KK}(\CC,\RR),\udl{\R})}, \underline{\D}) \xrightarrow[\sim]{\;\;\; L^* \;\;\;}\underline{\func}\distributiveLocalisation{w}{\K}{w_*\udl{S}^{-1}}\prn{w_*\underline{\mathcal E}, \underline{\D}}$ given by left Kan extension along the Bousfield inclusion or, equivalently, restriction along the inclusion, again. This shows that the unit transformation naturally descends to a transformation of $\KK$-distributive functors $w_\ast \udl{\wpresheaf}(\CC,\udl{\R}) \to \DD$.
    
    Let $\udl{\A}\subseteq \udl{\wpresheaf}^{\mathcal K}(\CC,\udl{\R})$ be the largest full subcategory such that the counit is an equivalence on $w_*\udl{\A}\subseteq w_*\udl{\wpresheaf}^{\mathcal K}(\CC,\RR)$. Since $y\colon \udl{\sC}\rightarrow\mathcal \E$ was fully faithful, we have that $L(\udl{\sC})\subseteq \udl{\A}$. On the other hand, $\udl{\A}\subseteq \udl{\wpresheaf}^{\mathcal K}(\CC,\udl{\RR})$ is closed under $\K$-indexed colimits by the distributivity of both sides.
    Since $L(\CC) \subset \udl{\mathcal A}$ generates $\udl{\presheaf}^{\KK}(\CC,\udl{\RR})$ under $\KK$-indexed colimits by definition, we conclude that 
    $\udl{\mathcal A} = \udl{\presheaf}^{\KK}(\CC,\udl{\R})$ and the claim follows.
\end{proof}

\subsection{Norms of partially cocomplete categories} \label{subsec:sym_mon_structure_cocomplete}

We now employ \Cref{thm:distributivity_localisation-cocompletion} to prove \Cref{alphThm:CATK_Symmon}, constructing an $\NN$-normed structure on $\udl{\K}$-cocomplete categories when $\udl{\K}$ is compatible with $\NN$ as follows:
\begin{definition}
    We call a $\baseCat$-subcategory $\udl{\K}\subseteq \udl{\Cat}$ \mdef{$\NN$-compatible} if, for all $\NN$-admissible maps of orbit-finite $\baseCat$-\animae{} $w\colon W \rightarrow V$, $w_*$ restricts to a functor $w_*\colon \udl{\K}^W\rightarrow \udl{\K}^V$.
\end{definition}
\begin{remark}\label{rmk:compatible-point}
    The case $W = \emptyset$ stipulates that every $\NN$-compatible $\KK$ contains $\udl{*}$.
\end{remark}

Write $\Catcoconeproj(\KK) \subset \Catcocone$ for the full $\baseCat$-subcategory of projectable classes of cocones indexed by diagrams in $\KK$.
This has a natural normed structure:
\begin{lemma}
    If $\KK$ is $\NN$-compatible, then $\Catcoconeproj(\KK) \subset \Catcocone$ is closed under $\NN$-norms.
\end{lemma}
\begin{proof}
    Given a constant cocone $\udl{\ast}^{\triangleright} \to \CC$, applying $w_\ast$ yields a constant colimit $( w_\ast \ast )^\triangleright \simeq \ast^{\triangleright} \to w_\ast \CC$ yields all constant cocones in $w_\ast \CC$, so projectibility is norm-closed. 
    That norm-closure of indexing categories in $\KK$ is equivalent to $\NN$-compatibility of $\KK$ follows from the description of $\Catcocone$-norms in \Cref{cons:catcocone}. 
\end{proof}
Moreover, since $\udl{*} \in \KK^V$ by \Cref{rmk:compatible-point}, we have an inclusion 
\[
\GCatcoco{\KK} \subset \Catcoconeproj(\KK) \subset \Catcocone.
\]


We use this to construct norms of $\KK$-cocomplete categories.
\begin{theorem} \label{thm:catK_symMon} 
    Let $\KK \subset \udl{\Cat}$ be an $\NN$-compatible subcategory.
    The subcategory inclusion $\GCatcoco{\KK} \subset \Catcoconeproj(\KK)$ admits a left adjoint which is a $\NN$-normed Bousfield localisation.

    Consequently, $\GCatcoco{\KK}$ and the localisation $\Catcoconeproj(\KK) \rightarrow \GCatcoco{\KK}$ jointly admit a unique $\NN$-normed structure, and the inclusion a lax $\NN$-normed structure.
    The $\GCatcoco{\KK}$-norm satisfies the universal property that the lax $\NN$-normed structure map $\mu\colon w_* \CC \rightarrow w_{\otimes} \CC$ induces a natural equivalence
    \begin{equation}\label{eq:catk_uni_prop}
        \mu^* \colon 
        \udl{\Fun}^{\KK}\prn{w_{\otimes} \CC, \DD}
        \xrightarrow{\;\;\;\;\; \sim \;\;\;\;\;}
        \udl{\Fun}\distributiveLocalisation{w}{\KK}{}\prn{w_* \CC, \DD}.
    \end{equation}
\end{theorem}
\begin{proof}
    We start with the existence of the left adjoint. We claim that the functor $(\CC,\udl{\RR}) \to \prn{\udl{\wpresheaf}^{\KK}(\CC,\udl{\RR}),\udl{\RR}_{\KK}\prn{\udl{\wpresheaf}^{\KK}(\CC,\udl{\RR})}}$ exhibits the target as the left adjoint object of the source. Unwinding definition and using the description of mapping anim\ae{} from \Cref{Catcocone-subcat}, we must verify for any $\KK$-cocomplete category $\DD$ that restriction along Yoneda induces an equivalence 
    \[ i^\ast \colon \underline{\func}^{\KK}\left(\underline{\wpresheaf}^{\KK}(\underline{\sC},\udl{\RR}), \underline{\D}\right)
    \xrightarrow{\;\;\;\; \sim \;\;\;\;}
    \udl{\func}^{\udl{\RR}}(\underline{\sC},\underline{\D}) \,. \]
    This follows from \Cref{thm:distributivity_localisation-cocompletion} for $w$ being the identity.

    For the $\NN$-compatibility,
    let $w \colon W \to V$ be an $\NN$-admissible map of orbit-finite $\TT$-\animae.
    By $2$-out-of-$3$ it is enough to verify the following: for $(\CC,\udl{\R}) \in \Catcocone$, the $w$-norm in $\Catcoconeproj$ sends $i \colon (\CC,\udl{\RR}) \to (\udl{\wpresheaf}^{\KK}(\CC,\udl{\RR}),\udl{\RR}_{\KK})$ (abbreviating $\udl{\RR}_{\KK}(\udl{\wpresheaf}^{\KK}(\CC,\udl{\RR}))$ by $\udl{\RR}_{\KK}$) to a morphism inverted by $\udl{\wpresheaf}^{\KK}$. 
    This in turn means that for any $\KK$-cocomplete $\DD$, restriction along $w_\ast i$ induces an equivalence
    \[ \myuline{\Map}\hspace{-1pt}\left((w_\ast\udl{\wpresheaf}^{\KK}(\CC,\udl{\RR}),w_*\udl{\RR}_{\KK}),(\DD,\udl{\RR}_{\KK}(\DD))\right) \rightarrow \myuline{\Map}\hspace{-1pt}\left((w_\ast\CC,w_*\udl{\RR}),(\DD,\udl{\RR}_{\KK}(\DD))\right). \]
    This again follows from combining \Cref{thm:distributivity_localisation-cocompletion} and \Cref{Catcocone-subcat}.

    The description of the normed structure on $\GCatcoco{\KK}$ follows from \Cref{thm:distributivity_localisation-cocompletion}.
\end{proof} 
In addition to the norm functors, we canonically describe the coherences:
\begin{lemma} \label{lem:total-space-CatK-Cat}
    The functor of total categories $\int \GCatcoco{\KK} \longrightarrow \int \wGCat $ over $\SFTN{\TT}{\NN}$ induced by the lax $\NN$-normed inclusion $\GCatcoco{\KK} \subset \wGCat$    is an inclusion of a subcategory.
\end{lemma}
\begin{proof}
    We first verify that its an inclusion of path components on cores. The cores live in the fibres over $\SFTN{\TT}{\NN}$, so that, using the Segal condition, we can reduce to the fact that $\GCatcoco{\KK}$ is a subcategory of $\GCat$ which has been checked in \Cref{Catcocone-subcat}.

    For faithfulness, we use that the functor is the composite 
    $\int \GCatcoco{\KK} \subset \int \Catcocone \subset \int \wGCat$
    where the first functor is fully faithful by construction and the second is faithful as it is the map on total categories of a faithful $\NN$-normed $\baseCat$-functor by \Cref{Catcocone-subcat}.
\end{proof}
\begin{remark}\label{rmk:mw-vs-us}
    In the case $\NN = \baseCat^{\simeq}$, taking fibre over an orbit $V \in \baseCat$ of \Cref{lem:total-space-CatK-Cat} shows that the induced lax-symmetric monoidal functor $\wGCat\prn{\KK}^{V} \rightarrow \wGCat[V]$ induces a subcategory inclusion  of the unstraightenings over $\SFTN{}{}$, and hence a subcategory inclusion after pulling along the inclusion $\FinTN[]{*} \hookrightarrow \SFTN{}{}$.
    It follows by unwinding definitions that this agrees with the $V$-fibre of the construction in \cite[\S~2.5]{mwpresentable};
    said another way, we have provided norms and associated coherences for the  specialisation of their internal tensor product of $\KK$-cocomplete $\baseCat$-categories to presheaf toposes.
\end{remark}

Combining this with \cite[Prop.~II.2.54]{nat-thesis} yields the following.
\begin{corollary} \label{cor:NAlg-in-CatK}
    The $\baseCat$-functor on normed algebras
    \[ \NAlg{\GCatcoco{\KK}} \longrightarrow \NAlg{\wGCat} \simeq \wGCat_{\NN-\otimes} \]
    is an inclusion of the $\baseCat$-subcategory of $\KK$-distributive $\NN$-normed $\baseCat$-categories and $\KK$-cocontinuous $\NN$-normed $\baseCat$-functors.
\end{corollary}

We now briefly discuss the functoriality of the normed structure in $\KK$.
\begin{proposition}[Normed relative cocompletion] \label{prop:nested-indexing}
    Let $\LL \subset \KK \subset \GCat$ be two nested $\NN$-compatible classes of indexing categories. 
    The composite lax $\NN$-normed $\baseCat$-functor
    \[ \GCatcoco{\LL} \to \Catcoconeproj(\LL) \subset \Catcoconeproj(\KK) \longrightarrow \GCatcoco{\KK} \]
    is an $\NN$-normed $\baseCat$-functor,
    left adjoint to the non-full inclusion $\GCatcoco{\KK} \hookrightarrow \GCatcoco{\LL}$.
\end{proposition}
\begin{proof}
    All but the first functor $\GCatcoco{\LL} \to \Catcoconeproj(\LL)$ is already $\NN$-normed, the last by \Cref{thm:catK_symMon}.
    As this is the inclusion associated to the Bousfield localisation $\Catcoconeproj(\LL) \to \GCatcoco{\LL}$, the lax structure maps are inverted by the localisation (as the composite of inclusion and localisation is the $\NN$-normed structure on the identity). 
    It thus suffices to prove that a morphism which is inverted by $\Catcoconeproj(\LL) \to \GCatcoco{\LL}$ is also inverted by $\Catcoconeproj(\LL) \subset \Catcoconeproj(\KK) \to \GCatcoco{\KK}$. 

    Indeed \Cref{Catcocone-subcat}, we see that a morphism $(\CC,\udl{\R}) \to (\DD,\udl{\mathcal{S}})$ in $\Catcoconeproj(\LL)$ is inverted by the first functor if 
    $\Fun^{\udl{\mathcal S}}(\DD,\EE) \longrightarrow \Fun^{\udl{\R}}(\CC,\EE)$
    is an equivalence for all $\LL$-cocomplete $\EE$. Similarly, we see that it is inverted by the second composite if this is an equivalence for all $\KK$-cocomplete $\EE$, which is clearly a weaker requirement.

    That the underlying $\baseCat$-functor is left adjoint to the forgetful functor is easy to verify by the universal property of the left adjoint of the inclusion $\GCatcoco{\KK} \subset \Catcoconeproj(\KK)$. 
\end{proof}
Finally, we verify that $\GCatcoco{\KK}$ itself is distributive.
\begin{proposition} \label{prop:CatKK-distributive}
    The $\NN$-normed structure on $\GCatcoco{\KK}$ is distributive. 
\end{proposition}
\begin{proof}
    By \Cref{prop:transf_examples} and the $\NN$-normed Bousfield localisation $\Catcoconeproj(\KK)\rightarrow \wGCat(\KK)$, it is enough to verify that the $\NN$-normed structure on $\Catcoconeproj(\KK)$ is distributive. 
    Since the inclusion $\Catcoconeproj(\KK)\subset \Catcocone$ is an $\NN$-normed conservative $\baseCat$-functor with distributive codomain under \Cref{prop:Catcocone-distributive}, \Cref{prop:transf_examples} reduces us to verifying that  $\Catcoconeproj(\KK) \subset \Catcocone$ is closed under colimits.

    We claim that this inclusion actually admits a right adjoint: Given a class of cocones $\udl{\RR}$ on a $\baseCat$-category $\CC$, let $\CC^{\mathrm{proj}}$ denote the full $\baseCat$-subcategory of $\CC$ consisting of the objects for which $\udl{\R}$ contains constant colimits. 
    Moreover, let $\udl{\R}_{\KK}^{\mathrm{proj}}$ denote the $\baseCat$-class of cocones $\udl{K}^{\triangleright} \to \CC^{\mathrm{proj}}$ whose composite to $\CC$ is lies  in $\udl{\RR}$ and for which $\udl{K} \in \KK$. 
    We claim that $(\CC^{\mathrm{proj}},\udl{\R}_{\KK}^{\mathrm{proj}}) \to (\CC,\udl{\RR})$ exhibits the left hand side as the right adjoint object of $(\CC,\udl{\RR})$ to $\Catcoconeproj(\KK) \subset \Catcocone$.
    First note that the left hand side actually is contained in $\Catcoconeproj(\KK)$, by definition. Next, let $(\DD,\udl{\mathcal S}) \in \Catcoconeproj(\KK)$. Consider
    \[ \begin{tikzcd}[row sep=scriptsize]
        \myuline{\Map}_{\Catcocone}((\DD,\udl{\mathcal S}),(\CC^{\mathrm{proj}},\udl{\R}_{\KK}^{\mathrm{proj}})) \ar[r] \ar[d] & \myuline{\Map}_{\Catcocone}((\DD,\udl{\mathcal S}),(\CC,\udl{\R})) \ar[d] \\
        \myuline{\Map}_{\wGCat}(\DD,\CC^{\mathrm{proj}}) \ar[r] & \myuline{\Map}_{\wGCat}(\DD,\CC)
    \end{tikzcd} \]
    where the vertical arrows are induced by the forgetful functor and the horizontal by postcomposition. 
    The vertical arrows are monomorphisms by \Cref{Catcocone-subcat} and the lower arrow is a monomorphism as $\CC^{\mathrm{proj}} \subset \CC$ is a $\baseCat$-subcategory inclusion. 

    It follows that the upper arrow is a monomorphism, too. 
    We claim that it is $\pi_0$-surjective by \Cref{Catcocone-subcat}: 
    indeed, given a functor $\DD \to \CC$ sending cocones to cocones, it has to factor through $\CC^{\mathrm{proj}}$ as $\DD$ is projectable. Moreover, as $\DD$ only contains $\KK$-indexed cocones, it is enough to verify that it can only send $\udl{\mathcal S}$-cocones to $\udl{\R}_{\KK}^{\mathrm{proj}}$-cocones, but this is clear.
\end{proof}

We finish the subsection by addressing some small matters:
consider the full $\baseCat$-subcategory $\smallCatcoconeproj(\KK) \subset \Catcoconeproj(\KK)$ spanned by small $V$-categories equipped with a $V$-class of $\KK$-cocones. This is closed under small colimits and $\NN$-norms by the cocontinuity and $\NN$-normedness of the forgetful functor to $\wGCat$ together with norm closure of $\GCat \subset \wGCat$.
    The inclusion $\GCat(\KK) \subset\smallCatcoconeproj(\KK)$ admits a left adjoint: it is the $\KK$-colimit closure of the Bousfield localisation $\udl{\presheaf}(\CC,\RR)$ of $\udl{\presheaf}(\CC)$ at the small class of colimit-comparison maps, by the same proof of \Cref{thm:distributivity_localisation-cocompletion}. 
    For the natural equivalence, it suffices to check that this left adjoint is the restriction of the left adjoint $\Catcoconeproj(\KK) \to \wGCat(\KK)$;
    this follows since the universal property provided in the proof will also hold for mapping into large $\KK$-cocomplete categories. 

    This shows that $\Catcoconeproj(\KK) \to \GCatcoco{\KK}$ and $\smallCatcoconeproj(\KK) \to \GCat(\KK)$ is a pair of $\NN$-normed Bousfield localisations compatible with the respective inclusions. It follows that $\GCat(\KK) \to \GCatcoco{\KK}$ is $\NN$-normed and cocontinuous.
    
    Evaluating the left adjoint object of a small category $\CC$ with a small class of projectable cocones $\udl{\R}$ under the inclusion $\GCatcoco{\allcolim} \subset \Catcoconeproj$ is given by  $\udl{\presheaf}(\CC,\udl{\R})$ which is presentable. Applying this to a small $\KK$-cocomplete category together with its small class of $\KK$-cocones shows that the relative cocompletion $\udl{\presheaf}^{\allcolim}_{\KK}(\CC)$ is presentable.

    All in all, we have proved:
\begin{proposition}\label{prop:relative_cocompletion_through_presentable}
    Suppose $\KK$ is small.
    Then, $\GCat\prn{\KK} \subset \wGCat\prn{\KK}$ is closed under small $\baseCat$-colimits and $\NN$-norms;
    hence it is distributive.
    Moreover, there is an $\NN$-normed natural equivalence $\udl{\wpresheaf}^{\allcolim}_{\KK}
    \mid_{\GCat\prn{\KK}}  \simeq  \udl{\presheaf}^{\allcolim}_{\KK}$, and this cocompletion factors as follows:
    \[
        \udl{\presheaf}^{\allcolim}_{\KK}\colon \GCat\prn{\KK} \longrightarrow \udl{\Pr}^L \subset \wGCat\prn{\allcolim}.
    \]
\end{proposition}

\subsection{Norms of presentable categories}\label{subsec:norms_presentable}
Our goal in this subsection is to construct the $\NN$-normed structure on presentable categories by inheriting it from that on cocomplete large categories, and to show moreover that it is distributive. This will involve leveraging some of the set-theoretic matters established in \cite{MartiniWolf2024}, details of which we shall suppress here and defer to \Cref{sec:size}. 

To begin the discussion, we need the following set-theoretic language. 
\begin{notation}\label{def:filt}
    For $\kappa$ a regular cardinal, we write $\filt(\kappa) \subset \Cat$ for the full subcategory of $\kappa$-filtered categories;
    hence we write $\infl{}{\baseCat}\left(\filt(\kappa)\right) \subset \GCat$ for the full $\baseCat$-subcategory of constant $\filt(\kappa)$-valued $\baseCat$-categories, and similarly $\infl{}{\baseCat}\mathrm{idem} \subset \GCat$ for those constant on the category $\mathrm{idem}$ corepresenting idempotent endomorphisms.
    
    We will call a $V$-object $c$ in a $\baseCat$-category $\CC$ \mdef{$\kappa$-compact} if its co-Yoneda image $\myuline{\Map}(c,-) \colon V^*\CC \to V^*\GSpaces$ commutes with $\infl{}{\baseCat}\left(\filt(\kappa)\right)$-colimits, and we denote by $\CC^{\kappa}\subset \CC$ the full $\baseCat$-subcategory of $\kappa$-compact objects. 
\end{notation}
\begin{definition}\label{def:small}
    Let $\kappa$ be a regular cardinal.
    \begin{itemize}
    \item  The category of \mdef{$\kappa$-small} $\baseCat$-categories $\sm(\kappa)^{\baseCat} \subset \Cat[\baseCat]$ is the closure of $\baseCat\subset \Spaces[\baseCat]\subset \Cat[\baseCat]$ under non-parametrised $\kappa$-small colimits.
     \item We call $\infl{}{\baseCat}\mathrm{idem}$-cocomplete $\baseCat$-categories \mdef{idempotent-complete}, and we write the idempotent cocompletion as
        \[ 
            \AA^{\natural} \coloneqq \udl{\presheaf}^{\infl{}{\baseCat}\left(\mathrm{idem}\right)}\prn{\AA}. 
        \]  
    \item For $\AA$ a small $\baseCat$-category, we denote the $\infl{}{\baseCat}(\filt(\kappa))$-cocompletion as 
    \[ 
    \udl{\ind}^\kappa\prn{\AA} \coloneqq \udl{\presheaf}^{\infl{}{\baseCat}\left(\filt(\kappa)\right)}\prn{\AA}. 
    \]  
    \item The $\baseCat$-categories of the form  $\udl{\ind}^\kappa\prn{\AA}$ for $\AA$ small are called \mdef{$\kappa$-accessible};
    functors between $\kappa$-accessible $\baseCat$-categories are called \mdef{$\kappa$-accessible $\baseCat$-functors} if they are $\infl{}{\baseCat}\prn{\filt(\kappa)}$-cocontinuous and they preserve $\kappa$-compact $V$-objects for all orbits $V$.
    We write $\GAcc{\kappa} \subset \wGCat$ for the full $\baseCat$-subcategory of $\kappa$-accessible $V$-categories and $\kappa$-accessible $V$-functors.
    \item We define \mdef{$\kappa$-presentable} $\baseCat$-categories and $\baseCat$-functors as those which are separately presentable and $\kappa$-accessible;
    that is, they are the $\baseCat$-objects and $\baseCat$-morphisms of
    \[
        \udl{\Pr}_\kappa^L \coloneqq \GAcc{\kappa} \cap  \udl{\Pr}^L \subset \wGCat.
    \]
    \end{itemize}
\end{definition}
The following lemma specialises \cite[Prop.~5.4.5]{mwpresentable}.
\begin{lemma}\label{lem:ind-fibrewise}
    For any subcategory $\K\subset \Cat$,  let $\infl{}{\baseCat} \KK \subset \GCat$ be the associated $\baseCat$-subcategory of constant parametrised categories. There is an equality of $\baseCat$-subcategories
    \[
    \wGCat\prn{\infl{}{\baseCat}\K}\simeq \cofree\prn{\wCat\prn{\K}} \subset \wGCat.
    \]
    We acquire natural equivalences, $\udl{\ind}^{\kappa}\prn{\AA}^V \simeq \ind^{\kappa}\prn{\AA^V}$ and $
        \prn{\AA^{\natural}}^V \simeq \prn{\AA^V}^{\natural}$.
\end{lemma}

This allows us to identify $\kappa$-accessibility as a fibrewise condition.
\begin{proposition}\label{prop:acc-fibrewise}
    There is an equality of $\baseCat$-subcategories of $\wGCat$:
    \[
        \GAcc{\kappa} = \cofree\prn{\Acc{\kappa}} \subset \wGCat.
    \]
    In particular, if $\CC$ is $\kappa$-accessible, we acquire an equality of full subcategories 
    \[
        \prn{\CC^\kappa}^V = \prn{\CC^V}^{\kappa} \subset \CC^V.
    \]
\end{proposition}
To prove this, we first identify its fibrewise expression for $\kappa$-compact objects:
\begin{lemma}\label{lem:res-of-compact}
    A $\baseCat$-category $\CC$ admits equalities $\prn{\CC^\kappa}^V = \prn{\CC^V}^{\kappa} \subset \CC$ if and only if the inclusions $\prn{\CC^{V}}^{\kappa} \subset \CC^V$ form a $\baseCat$-subcategory.
\end{lemma}
\begin{proof}
    By examining the defining natural equivalence $\myuline{\Map}_{V^*\CC}\prn{c,-}^U \simeq {\Map}_{\CC^U}\prn{c,-}$ and applying the fibrewise expression for inflated-cocontinuous $\baseCat$-functors (as e.g. supplied by \Cref{lem:ind-fibrewise}), we have that $c \in V^* \CC$ is $\kappa$-compact if and only if $W^* c \in \CC^W$ is (non-parametrised) $\kappa$-compact for all $W \rightarrow V$.
    That this is equivalent to $c \in \CC^V$ being (non-parametrised) $\kappa$-compact for all pairs $(V,c)$ is precisely the latter condition.
\end{proof}

\begin{proof}[Proof of \Cref{prop:acc-fibrewise}]
    Applying \Cref{lem:res-of-compact}, it suffices to prove that fibrewise $\kappa$-compact objects are closed under restriction in $\udl{\ind}^{\kappa}\prn{\CC}$ when $\CC$ is small.
    But the non-parametrised functor $\udl{\ind}^{\kappa}\prn{-}$ sends functors to $\kappa$-compact preserving functors, so this follows by the natural equivalence  $\udl{\ind}^{\kappa}\prn{\AA}^V \simeq \ind^{\kappa}\prn{\AA^V}$ of \Cref{lem:ind-fibrewise}.
\end{proof}

In order to incorporate norms into this, $\kappa$ must additionally satisfy the following:
\begin{definition} \label{def:kappa-normable}
    A regular cardinal $\kappa$ is called \mdef{$\baseCat$-restrictable} (resp. \mdef{$\NN$-normable}) if $\sm(\kappa)^{V} \subset \GCat[V]$ form a $\baseCat$-subcategory (resp. $\NN$-normed $\baseCat$-subcategory).
    In this case, we denote the resulting $\baseCat$-subcategory by $\udl{\sm}(\kappa)$.
\end{definition}
The following related pieces of notation will be useful.
\begin{notation}\label{def:Card}
    Let $\Card$ be the homotopy poset of the wide category $\mathrm{Set}^{\mathrm{inj}} \subset \mathrm{Set}$ of injections.
    We let $\Card[\NN] \subset \Card$ be the full subposet of $\NN$-normable cardinals.
\end{notation}
\begin{notation}
    Let $\kappa$ be a $\baseCat$-restrictable regular cardinal.  We write 
    \[ \udl{\Cat}^{\mathrm{rex}(\kappa)} = \udl{\Cat}(\udl{\sm}(\kappa))
    \quad\text{ and }\quad\quad \udl{\Cat}^{\mathrm{irex}(\kappa)} = \udl{\Cat}(\udl{\sm}(\kappa) \sqcup \infl{}{\baseCat} \mathrm{idem}) \,.
\]
Note that $\udl{\Cat}^{\mathrm{rex}(\kappa)} = \udl{\Cat}^{\mathrm{irex}(\kappa)}$  when $\kappa$ is uncountable by \Cref{lem:small-vs-cpct}. 
\end{notation}

With these definitions, we may state the theorem to be proven.

\begin{theorem}\label{cor:irex-vs-prl}\label{thm:prL-colims}
    The subcategory $\udl{\Pr}^L \subset \wGCat\prn{\allcolim}$ is closed under small $\baseCat$-colimits and $\NN$-norms. Consequently, $\udl{\Pr}^L$ admits a unique distributive $\NN$-normed structure under which the inclusion is $\NN$-normed.
    
    In fact,  the inclusion $\udl{\Pr}^L_{\kappa}\subset \udl{\Pr}^L$ is cocontinuous and $\NN$-norm-closed for any $\NN$-normable $\kappa$, inducing an identification  $\udl{\Pr}^L = \bigcup_{\kappa \in \Card[\NN]} \udl{\Pr}^L_{\kappa}$ of $\NN$-normed $\baseCat$-subcategories.  
\end{theorem}
The idea for the proof is to express $\udl{\Pr}^L_{\kappa}$ as the image of the ($\NN$-normed) relative cocompletion from $\kappa$-small-cocomplete $\baseCat$-categories.
We start by noting that \Cref{lem:ind-fibrewise} and \cite[Lem.~5.4.2.4]{lur09} immediately yield the following.
\begin{lemma}\label{lem:idem-bousfield}
    Let $\kappa$ be a $\TT$-restrictible cardinal. There is a natural equivalence of functors 
$\prn{\udl{\ind}^{\kappa}(\AA)}^{\kappa} \simeq \AA^{\natural}$; 
hence fibrewise idempotent completion restricts to a Bousfield localisation \[
(-)^{\natural}\colon \udl{\Cat}^{\mathrm{rex}(\kappa)}\longrightarrow \udl{\Cat}^{\mathrm{irex}(\kappa)}.
\]
\end{lemma}
We can strengthen this with more existing $\udl{\ind}^{\kappa}$ work using the following.
\begin{remark}\label{rmk:reg-brink}
    The colimit preservation is a specialisation of \cite[Prop.~5.18]{brink2026globalinftycategoriesglobalthom} to the terminal cleft $\baseCat = \baseCat$, so we will not reprove it, though our argument for norm-closure is similarly shaped to theirs.
    Moreover, we verify in \Cref{cor:everything-is-everything} that $\baseCat$-restrictability is equivalent to the $\prn{\baseCat,\baseCat}$-regularity of \cite[Def.~5.20]{brink2026globalinftycategoriesglobalthom}, so we may and will use their results on $\udl{\ind}^{\kappa}$ freely.
\end{remark}
In particular, \Cref{rmk:reg-brink} exhibits \cite[Lem.~5.21-5.22]{brink2026globalinftycategoriesglobalthom} as a generalisation of the following:
\begin{lemma}\label{lem:filt-lim-commute}
    Let $\kappa$ be $\TT$-restrictible. Then $\infl{}{\baseCat}\left(\filt(\kappa)\right)$-colimits commute with $\kappa$-small limits in $\GSpaces$;
    hence, for any $\baseCat$-category $\CC$, $\CC^{\kappa} \subset \CC$ is closed under $\kappa$-small colimits.
\end{lemma}
\begin{lemma}\label{prop:ind_filt}\label{prop:univ_prop_ind_cocomplete}
    Let $\kappa$ be a $\baseCat$-restrictable cardinal. There is an equality of subcategories 
    \[
    \udl{\Fun}^{\udl{\sm}(\kappa)-\mathrm{lim}}\prn{\AA\op, \GSpaces} =\udl{\ind}^{\kappa}\prn{\AA}
    \]
    of $\udl{\presheaf}\prn{\AA}$, where the left-hand side denotes $\udl{\sm}(\kappa)$-limit-preserving functors. 
    If moreover $\AA$ is $\udl{\sm}(\kappa)$-cocomplete, then these are equal to $\udl{\presheaf}_{\udl{\sm}(\kappa)}^{\allcolim}\prn{\AA}$.
\end{lemma}

Combining these lemmata yields the following corollary.
\begin{corollary}\label{cor:kappa-ind-adjunction}\label{cor:ind-of-compact}
    Let $\kappa$ be a $\baseCat$-restrictable cardinal. Then $\udl{\presheaf}_{\udl{\sm}(\kappa)}^{\allcolim}$ and $\prn{-}^{\kappa}$ restrict to an adjunction \[\udl{\presheaf}_{\udl{\sm}(\kappa)}^{\allcolim}\simeq \udl{\ind}^{\kappa}\colon \udl{\Cat}^{\mathrm{rex}(\kappa)} \rightleftharpoons \udl{\Pr}^L_{\kappa}\cocolon (-)^{\kappa} \,, \]
    which further restricts to inverse equivalences $\udl{\Cat}^{\mathrm{irex}(\kappa)} \simeq \udl{\Pr}^L_{\kappa}$.
\end{corollary}
\begin{proof}
    To see the adjunction, first note that $(-)^{\kappa}$ takes values in $\udl{\Cat}^{\mathrm{rex}(\kappa)}$ by \Cref{lem:filt-lim-commute}. Now, \Cref{prop:ind_filt} gives an equivalence
    $\udl{\Fun}^{\udl{\sm}(\kappa)}\prn{\AA, \CC}^{\simeq}\simeq \udl{\Fun}^L\prn{\udl{\ind}^{\kappa} (\AA), \CC}^{\simeq}$ for all $\AA\in \udl{\Cat}^{\mathrm{rex}(\kappa)}$ and ${\CC}$ cocomplete. It thus suffices to prove that  $\udl{\Fun}^{\udl{\sm}(\kappa)}\prn{\AA, \CC^{\kappa}}^{\simeq}$ is carried equivalently onto the subanima $\udl{\Fun}^L\prn{\udl{\ind}^{\kappa} (\AA), \CC}^{\simeq}$ of $\kappa$-compact-preserving functors. Since $\udl{\ind}^{\kappa}(\AA)^{\kappa}\simeq \AA^{\natural}$ by \Cref{lem:idem-bousfield}, we just have to argue that if $\udl{\ind}^{\kappa} (\AA) \rightarrow \CC$ has $\kappa$-compact values on $\AA\subset \udl{\ind}^{\kappa}\AA$, then it has $\kappa$-compact values on \textit{retracts} of $\AA$;
    this holds since functors preserve retracts and retracts preserve $\kappa$-compactness.

    To see that the adjunction restricts to an equivalence, we just need to argue that the (co)unit maps $\AA \rightarrow (\udl{\ind}^{\kappa}(\AA))^{\kappa}$ and $\udl{\ind}^{\kappa}(\CC^{\kappa})\rightarrow \CC$ are equivalences for $\AA\in \udl{\Cat}^{\mathrm{irex}(\kappa)}$ and $\CC\in\udl{\presentable}^L_{\kappa}$. But since this may be checked fibrewise and since all operations are implemented fibrewise by \Cref{lem:ind-fibrewise}, this reduces to the non-parametrised statement for which it is known, cf. \cite[Lem.~5.4.2.4]{lur09}.
\end{proof}

Through this, we may prove a useful corollary.
\begin{corollary}\label{cor:kappa-incl}
    Let $\kappa \leq \lambda$ be $\baseCat$-restrictable cardinals.
    There is an inclusion
    \[
        \udl{\Pr}^L_{\kappa} \subset \udl{\Pr}^L_{\lambda} \subset \udl{\Pr}^L.
    \]
\end{corollary}
This likely does not hold for accessibility without presentability;
indeed, it does not hold for accessibility of non-parametrised 1-categories, per \cite[Thm.~2.11]{Adamek}.
\begin{proof}[Proof of \Cref{cor:kappa-incl}]
    Composing left adjoints to the inclusions
    \[
        \wGCat\prn{\allcolim} \subset \wGCat\prn{\udl{\sm}(\lambda)} \subset \wGCat\prn{\udl{\sm}(\kappa)},
    \]
    we find that there is a natural equivalence of functors  $\wGCat\prn{\udl{\sm}(\kappa)} \longrightarrow \wGCat\prn{\allcolim}$
    \[
        \udl{\ind}^{\kappa}\prn{\AA} \simeq \udl{\ind}^{\lambda} \udl{\wpresheaf}_{\udl{\sm}(\kappa)}^{\udl{\sm}(\lambda)}\prn{\AA}.
    \]
    In particular, this restricts to  natural equivalences of functors of $\udl{\Pr}^L_{\kappa} \rightarrow \udl{\Pr}^L$
    \[
        \CC \simeq \udl{\ind}^{\kappa}\prn{\CC^{\kappa}} \simeq \udl{\ind}^{\lambda} \udl{\presheaf}_{\udl{\sm}(\kappa)}^{\udl{\sm}(\lambda)}\prn{\CC^{\kappa}}.
    \]
    This shows that $\udl{\Pr}^L_{\kappa}$ lands in the image of $\udl{\ind}^{\lambda}$, which is $\udl{\Pr}^L_{\lambda}$ by \Cref{cor:kappa-ind-adjunction}.
\end{proof}

\Cref{cor:kappa-incl} shows that  $\kappa \mapsto \udl{\Pr}^L_{\kappa} \hookrightarrow \wGCat\prn{\allcolim}$ constitutes a poset map
\[
    \mathrm{P}\colon \prn{\Card[\NN]}^{\triangleright} \longrightarrow \mathrm{Sub}\prn{\wGCat\prn{\allcolim}},
\]
whose cocone point is the inclusion $\udl{\Pr}^L \hookrightarrow \wGCat\prn{\allcolim}$.
\begin{proposition}\label{prop:prL-filt}\label{prop:choose-kappa}
    $P$ is a (large) colimit diagram.
\end{proposition}
\begin{proof}
    By the description of \cite[Prop.~I.2.65,Cor.~I.2.67]{nat-thesis} (and filteredness of total orders), we must show that  
     every presentable $\baseCat$-category is $\kappa$-accessible for some $\NN$-normable $\kappa$ and every left adjoint of presentable $\baseCat$-categories preserves $\kappa$-compact objects for some $\NN$-normable $\kappa$;
    equivalently, every right adjoint of presentable $\baseCat$-categories is $\kappa$-accessible.
    That this holds for Martini-Wolf's ``sound doctrines'' is \cite[Cor.~2.4.1.8,Prop.~2.4.3.3]{mwpresentable};
    that it is true for uncountable $\NN$-normable cardinals is then  \Cref{rmk:cofin}.
\end{proof}

We are now ready to conclude the theorem.

\begin{proof}[Proof of \Cref{thm:prL-colims}.]
    First note that the inclusion $\udl{\Pr}^L_{\kappa}\subset \wGCat\prn{\allcolim}$ is the image of 
     \[
    \udl{\ind}^{\kappa}(-)\simeq\udl{\presheaf}_{\udl{\sm}(\kappa)}^{\allcolim}(-)\colon \udl{\Cat}^{\mathrm{rex}(\kappa)}\rightarrow \wGCat(\allcolim),
    \]
    by \Cref{cor:ind-of-compact};
    this functor is $\NN$-normed by \Cref{prop:nested-indexing,prop:relative_cocompletion_through_presentable}, so its image $\udl{\Pr}^L_{\kappa}\subset \wGCat\prn{\allcolim}$ is $\NN$-norm closed.
    Hence \Cref{prop:choose-kappa} expresses $\udl{\Pr}^L \subset \wGCat\prn{\allcolim}$ as a sequential colimit of $\NN$-norm closed $\baseCat$-subcategories, implying that it is norm-closed. 
    \Cref{prop:transf_examples} then completes the proof.
\end{proof}
As an immediate consequence of \Cref{thm:prL-colims,cor:NAlg-in-CatK} and the adjoint functor theorem \cite[Prop.~2.4.3.1]{mwpresentable}, we may characterise $\NN$-normed algebras in $\udl{\presentable}^L$:
\begin{corollary}\label{cor:prL-nalg}
    The $\baseCat$-functor induced by the lax $\NN$-normed inclusion 
    \[ \NAlg{\udl{\Pr}^L} \longrightarrow \NAlg{\wGCat[\times]} \simeq \wGCat_{\NN-\otimes} \]
    is a $\baseCat$-subcategory inclusion onto the distributive $\NN$-normed $\baseCat$-categories with presentable underlying $\baseCat$-categories and $\NN$-normed left adjoint $\baseCat$-functors.
\end{corollary}

\begin{remark}\label{rmk:formula_for_presentable_tensors}
    After \Cref{rmk:mw-vs-us}, \cite[Prop.~2.6.2.11]{mwpresentable} generalises Lurie's tensor product formula to our setting; namely, we acquire a natural equivalence 
    \[
        \CC\otimes \DD\simeq \udl{\func}^{\mathrm{lim}}(\CC\op,\DD)
    \] where the right-hand side denotes the full subcategory of limit-preserving functors.
\end{remark}

To finish, we connect this work to small $\baseCat$-categories: 

\begin{theorem}\label{thm:ind}
    Let $\kappa$ be $\NN$-normable. The idempotent-completion Bousfield localisation $(-)^{\natural}\colon \udl{\Cat}^{\mathrm{rex}(\kappa)}\rightarrow \udl{\Cat}^{\mathrm{irex}(\kappa)}$ refines uniquely to an $\NN$-normed functor. 
    
    The induced $\NN$-norms on the codomain satisfy the following universal property: for $w \colon W \to V$ a map of orbit-finite $\baseCat$-\animae{}, $\CC\in (\udl{\Cat}^{\mathrm{irex}(\kappa)})^W$, and $\DD\in (\udl{\Cat}^{\mathrm{irex}(\kappa)})^V$, the lax structure $V$-functor $w_\ast \CC\rightarrow w_{\otimes}\irex \CC$ induces an equivalence
    \[ \udl{\Fun}^{\mathrm{rex}(\kappa)}(w_{\otimes}^{\mathrm{irex}(\kappa)} \CC,\DD) \longrightarrow \udl{\Fun}^{\udl{\sm}(\kappa)-\partial_w}(w_\ast \CC,\DD). \] 

    Furthermore, the $\NN$-normed functor $\udl{\ind}^{\kappa}\colon \udl{\Cat}^{\mathrm{rex}(\kappa)}\rightarrow \udl{\presentable}^L_{\kappa}$ factors uniquely through $(-)^{\natural}$, inducing inverse $\NN$-normed equivalences $\udl{\ind}^{\kappa}\colon \udl{\Cat}^{\mathrm{irex}(\kappa)}\simeq \udl{\presentable}^L_{\kappa}\cocolon (-)^{\kappa}$.
\end{theorem}
\begin{proof}
    That the Bousfield localisation $(-)^{\natural}$ refines to an $\NN$-normed functor is simply because $(-)^{\natural}$-equivalences are the morphisms inverted by $\udl{\ind}^{\kappa}\colon \udl{\Cat}^{\mathrm{rex}(\kappa)}\rightarrow \wGCat(\allcolim)$, which admits an $\NN$-normed structure by \Cref{prop:nested-indexing}.
    For the last statement, just note by \Cref{thm:prL-colims} that the $\NN$-normed functor $\udl{\ind}^{\kappa}$ factorises as the composite $\udl{\Cat}^{\mathrm{rex}(\kappa)}\twoheadrightarrow\udl{\Cat}^{\mathrm{irex}(\kappa)} \rightarrow \udl{\presentable}^L_{\kappa}$ of $\NN$-normed functors, and the last map is an equivalence by \Cref{cor:kappa-ind-adjunction}.

    For the statement about the universal property, first recall by construction that this is the universal property of the $\NN$-norms in $\udl{\Cat}^{\mathrm{rex}(\kappa)}$. Next, observe that $\mathrm{idem}\in\Cat$  is filtered (hence sifted), so $\infl{}{\baseCat}\mathrm{idem}$-distributivity is reduced to a fibrewise $\mathrm{idem}$-colimit preservation  by \Cref{prop:dist-lemmas}. But then all functors  preserve $\mathrm{idem}$-colimits since these are absolute colimits. 
\end{proof}

\begin{example}
    If $\baseCat$ is atomic orbital, we claim that all regular cardinals $\kappa$ are $\baseCat$-restrictable;
    indeed the characterisation of \Cref{prop:kappa-restricts} reduces this to the observation that
    $U \times_V X \in \FinTN[\TT]{}$ is finite for all orbits $U,V,X$, which is clear.
    It is, however, not clear that $\kappa$ is always $\baseCat$-normable, and we suspect that this is not the case.
\end{example}
We withhold more examples until \Cref{subsection:the_equivariant_case}.

\section{Modes}\label{section:presentable_categories}
We now finally turn to the theory of parametrised modes, proving 
\Cref{alphThm:invert_object_idempotent,alphThm:compact_lie_norms,alphThm:GSymmon_small_cats} by putting together, as ingredients, all the preceding theory.
These are instances of idempotent normed  algebras,  so we begin with a study of idempotence in its own right.

\subsection{Idempotent normed algebras}\label{subsection:modes}
Throughout this section, fix $\CC$ an $\NN$-normed $\TT$-category.

\begin{definition}\label{def:idem}
  We say that an $\EEE{0}$-algebra $\unit \rightarrow Y$ in $\CC^{\baseCat}$ is \mdef{idempotent} if the map 
  \[
    \id \otimes \eta\colon Y \simeq Y \otimes \unit \longrightarrow Y \otimes Y
  \]
  is an equivalence.
  Moreover, we say that $Y$ is \mdef{$\NN$-idempotent} if, for all $\NN$-admissible maps of orbit-finite $\baseCat$-\animae{} $w\colon W \rightarrow V$, the map
    \[
        \id \otimes w_{\otimes}W^\ast\eta\colon V^*Y\simeq V^*Y\otimes w_{\otimes}W^\ast \unit \longrightarrow V^*Y\otimes w_{\otimes}W^*Y
    \]
    in $\udl{\sC}^V$ is an equivalence.
    
    On the other hand, an $\NN$-normed algebra $X \in \NAlg[\TT]{\CC}$is said to be \mdef{idempotent} if, for all  $w\colon W \rightarrow V$ as above,
    the $(w \sqcup \id_V \colon W \sqcup V \to V)$-multiplication map
    \[
        \mu\colon (w_{\otimes}W^*X) \otimes V^*X\longrightarrow V^*X
    \]
    is an equivalence. 
    Write $\NAlgidem{\CC} \subset \NAlg{\CC}$ for the full $\baseCat$-subcategory of these.
\end{definition}
In particular, this allows us to define $\NN$-modes.
\begin{definition}
    An \mdef{$\NN$-mode} is an idempotent $\NN$-normed algebra in $\udl{\presentable}^L$.
    We say that a left adjoint $\eta\colon \GSpaces \rightarrow \CC$ \mdef{witnesses $\CC$ as an $\NN$-mode} if it is the unit map for an $\NN$-mode.
\end{definition}

For ordinary symmetric monoidal structures, \cite{lur17} uses the localisation of the unit to identify
smashing localisations and idempotent commutative algebra objects, each embedding fully faithfully into $\EEE{0}$-algebras;
this time, $\NN$-normedness of smashing $\TT$-localisations will be recorded by the existence of (essentially unique) normed algebra structures, which will now arise essentially uniquely by $\NN$-idempotence of the underlying $\EEE{0}$-algebra.
This begins by noting that idempotent $\EEE{0}$-algebras are closed under norms. 
\begin{lemma}
    Let $w\colon W \rightarrow V$ be an $\NN$-admissible map between orbit-finite $\TT$-\animae{}. If $\unit \rightarrow Y$ is an idempotent $\EEE{0}$-algebra in $W^*\udl{\sC}$, then so is $\unit \simeq w_{\otimes}\unit \rightarrow w_{\otimes}Y$ in $V^*\udl{\sC}$.
\end{lemma}
\begin{proof}
  This map is the image of $\unit \rightarrow Y$ under the  symmetric monoidal functor 
$w_{\otimes}$, and symmetric monoidal functors preserve idempotence of $\EEE{0}$-algebras.
\end{proof}

We may now prove our structural result concerning idempotent $\NN$-normed algebras.
\begin{proposition}\label{prop:idem-Nalg}
    Let $\CC$ be an $\NN$-normed $\TT$-category.
    The composite
    \[
        U\colon \NAlgidem{\CC} \subset \NAlg{\CC}  \xlongrightarrow{\;\;\; \fgt \;\;\;} \Alg{\EEE{0}}{\CC}
    \]
    is fully faithful with image those $\eta\colon \unit \rightarrow Y$ satisfying the following equivalent conditions:
    \begin{enumerate}[label={(\alph*)}]
      \item \label[condition]{cond:unit-idem} the $\mathbb{E}_0$-algebra $\eta\colon \unit\rightarrow Y$ is $\NN$-idempotent;
      \item \label[condition]{cond:modules-idem} the $\mathbb{E}_0$-algebra $\eta\colon \unit\rightarrow Y$ is idempotent and, for all $\NN$-admissible maps $w\colon W \rightarrow V$, $V^*Y$ admits the structure of a module over the unique idempotent commutative algebra structure on $\unit\rightarrow w_{\otimes}W^*Y$;
        \item the $\mathbb{E}_0$-algebra $\eta\colon \unit\rightarrow Y$ is idempotent and the smashing localisation $L = -\otimes Y\colon \CC \rightarrow \Mod{Y}{\CC}$ is $\NN$-normed; or
        \item the $\mathbb{E}_0$-algebra $\eta\colon \unit\rightarrow Y$ is idempotent and refines to the unit of an $\NN$-normed algebra structure on $Y$.
    \end{enumerate}
    Moreover, $\myuline{\mathrm{NAlg}}_{\NN}^{\mathrm{idem}}(\CC)^{\TT} \subset \NAlg[\TT]{\CC}$ is a symmetric monoidal full subposet satisfying
    \begin{align}
      Y' \leq Y &\iff Y \in \Mod{Y'}{\CC}^{\baseCat};\label{eq:relation-mod}\\
      \Mod{Y' \otimes Y}{\CC} 
        &\;\;\; = \;\;\; 
        \Mod{Y'}{\CC} \cap \Mod{Y}{\CC} \subset \CC.
      \label{eq:tensor-mod}
    \end{align}
\end{proposition}

A subtle point is that, while the $\baseCat$-category $\Mod{Y}{\CC}$ describes modules over $V^* Y$ by definition, it also describes honest modules over $Y$:
\begin{lemma}\label{rmk:cocart-sections}
    The restriction functors $\mathrm{Mod}_{Y}prn{\CC^\baseCat)} \rightarrow \Mod{V^*Y}{\CC^V} \simeq \Mod{Y}{\CC}^V$ induce an equivalence
    \[
        \mathrm{Mod}_{Y}\prn{\CC^\baseCat}\xrightarrow{\;\;\;\; \sim \;\;\;\;\;} \Mod{Y}{\CC}^{\baseCat}
    \]
\end{lemma}
\begin{proof}
    The left-hand side is equivalent to the fibres of the pullback functor on categories of algebras over $Y$, applied to the map of operads $\EEE{\infty} \simeq \EEE{\infty} \times_{\EEE{1}} \EEE{1} \rightarrow \EEE{\infty} \times_{\EEE{1}} \mathrm{LMod}$;
    but by \cite[Cor.~I.4.38]{nat-thesis} these are the $\baseCat$-objects in the fibre of the associated map of inflated parametrised operads, whose algebra categories are simply defined fibrewise by operad algebras, and hence match the right-hand side.  
\end{proof}

Using this, we may prove \eqref{eq:relation-mod}.
\begin{lemma}\label{lem:NAlg-idem}
    Given $X,X' \in \NAlgidem{\CC}$ whose underlying $\EEE{0}$-algebra satisfies (a) through (e), we have
    \[
       \Map_{\NAlg[\baseCat]{\CC}}(X,X') \simeq  \Map_{\Alg[\baseCat]{\EEE{0}}{\CC}}(X,X') \simeq \begin{cases}
           * & X' \in \Mod{X}{\CC}^{\baseCat}\\
           \emptyset & \mathrm{otherwise}. 
       \end{cases}
    \]
\end{lemma}
\begin{proof}
    After \Cref{rmk:cocart-sections}, the proof of \ha{Prop.}{4.8.2.9} generalises verbatim and contains the second equivalence. 
    It follows that the mapping anima in normed algebras is empty whenever there are not even maps between the underlying $\EEE{0}$-algebras.
    
    We are left with showing that, if $X^\prime$ is contained in the subcategory $\udl{\module}_X(\CC)^{\baseCat} \subset \CC^{\baseCat}$, then the anima of normed algebra maps $X \rightarrow X^\prime$ is contractible. 
    We may compute this mapping anima in the full subcategory $\NAlg{\udl{\module}_X(\CC)}$;
    note that $X$ is the tensor unit in $\udl{\module}_X(\CC)$, and hence it is initial in $\NAlg{\udl{\module}_X(\CC)}$ by \cite[Cor.~II.2.16]{nat-thesis} and \cite[Obs.~1.58]{tensor} (the latter observation extending to our generality verbatim).
\end{proof}
 
\begin{proof}[Proof of \Cref{prop:idem-Nalg}]
    We first prove that the above conditions are equivalent.
    That (a) $\iff$ (b) is obvious in light of the characterisation of modules over idempotent algebras \ha{Prop.}{4.8.2.10}.
    For showing (a) $\implies$ (c), let $w\colon W \rightarrow V$ be an $\NN$-admissible map of orbit-finite $\baseCat$-\animae{} and $\varphi\colon X \rightarrow Z$ an $L^W$-equivalence.
    We want to show that $V^*Y\otimes w_{\otimes}\varphi$ is an equivalence. By the hypothesis on $Y$, we know 
    \[
        V^*Y\otimes w_{\otimes}\varphi\simeq V^*Y\otimes w_{\otimes}W^*Y\otimes w_{\otimes}\varphi\simeq V^*Y\otimes w_{\otimes}(W^*Y\otimes \varphi),
    \]
    whence follows the desired conclusion. 
    
    Next,  (c) $\implies$ (a) is clear since $\unit\rightarrow W^*Y$ is an $L$-equivalence by $\NN$-idempotence of $\eta$.
    Now, (c) $\implies$ (d) because $Y$ is the image of the normed algebra $\unit$ under the $\NN$-normed localisation $L$ (see \cite[\S~I.3.4-6]{nat-thesis}).

    To finish identifying conditions, suppose (d).
    \cite[Cor.~III.1.8]{nat-thesis} shows that postcomposition along the $\NN$-normed $\TT$-functor $\CAlg\prn{\CC} \rightarrow  \CC$ under the pointwise structure provides an equivalence
    \[
        \NAlg{\myuline{\mathrm{CAlg}}(\CC)} \xrightarrow{\;\;\;\;\; \sim \;\;\;\;\;} \NAlg{\CC},
    \]
    so a general normed algebra in $\CC$ admits refinements of its normed multiplication map $w_{\otimes}  W^* X \rightarrow V^* X$
    to a commutative algebra map.
    In particular, $V^*X$ is given a module structure over $w_{\otimes} W^* X$;
    this specialises to (b).
   
    Now, the fact that $U$ is fully faithful is follows from the mapping anima formula in \Cref{lem:NAlg-idem}; indeed, the morphism of mapping \animae{} induced by $U$ is always an endomorphism of $*$ or $\emptyset$, and
    these are all equivalences. 
    
    To describe the essential image, first observe that $UY$ certainly satisfies condition (a) since the composition $V^*Y\simeq V^*Y\otimes \unit \xrightarrow{\id\otimes \eta} V^*Y\otimes w_{\otimes}W^*Y\xrightarrow{\mu} V^*Y$ is the identity map, and so the first map is an equivalence if and only if the second is. On the other hand, if an $\EEE{0}$-algebra $\unit\rightarrow Y$ satisfies any of the conditions, then in particular, it refines to an $\NN$-normed algebra by (d). And then by (a)  and the above argument, it is idempotent. 
    
    Finally, \eqref{eq:relation-mod} and \eqref{eq:tensor-mod} follow directly from the  non-parametrised statements, which are shown in the proof of \cite[Prop. 4.8.2.9]{lur17} and in \cite[\href{https://www.sas.rochester.edu/mth/sites/doug-ravenel/otherpapers/CSY2.pdf\#thm.5.1.8}{Prop.~5.1.8}]{csy20}. 
\end{proof}
This and \Cref{cor:prL-nalg} yield an immediate corollary about parametrised mode theory.
\begin{corollary} \label{cor:contractibility-monoidal-structures}
    If $\eta\colon \GSpaces \rightarrow \CC$ is an idempotent $\EEE{0}$-algebra in $\GPr[L]$, then the following conditions are equivalent.
    \begin{enumerate}[label={(\alph*)}]
      \item the map $\eta$ witnesses $\CC$ as an $\NN$-mode,
      \item the map $\eta$ refines to the unit of a distributive $\NN$-normed structure, and
      \item the smashing localisation on $\GPr[L]$ classified by $\eta$ attains an $\NN$-normed structure.
    \end{enumerate}
    In this case, the anima of distributive $\NN$-normed structures on $\CC$ with unit  $\eta$ is contractible.
\end{corollary}

\begin{remark}  
\Cref{thm:prL-colims} identifies the underlying $\TT$-category of the unit in $\udl{\presentable}^L$, but we may even further describe the entire  unit $\NN$-normed structure: \Cref{prop:distributivity_of_cartesian_symmon} establishes that the cartesian structure is distributive with unit $*$ classified by the identity $\GSpaces = \GSpaces$, so the unit and cartesian structure agree under \Cref{cor:contractibility-monoidal-structures}.
\end{remark}

\subsection{The pointed mode}\label{subsec:pointed_mode}
In this subsection, we show that  $\udl{\Spaces}_*$ is an $\NN$-mode in $\udl{\presentable}^L$ with unit $S^0 \coloneqq *_+$, classifying the full $\baseCat$-subcategory $\udl{\presentable}^L_*\subseteq \udl{\presentable}^L$ of pointed presentable parametrised categories. In particular, this will endow $\GBSpaces$ and $\udl{\presentable}^L_*$ with an $\NN$-normed structure.
As in the non-parametrised situation, the key idea will be to view $\udl{\Spaces}_*$ as a localisation of a presheaf $\baseCat$-category:
\begin{construction} \label{const:pointed}
    Consider $\Delta^{1}$ with the projectable class of cocones consisting of:
    \begin{itemize}
        \item the constant cocones on the two objects $0,1 \in \Delta^1$, and
        \item the cocone $\emptyset^\triangleright \simeq *\to \Delta^{1}$ picking $1$.
    \end{itemize}
    Passing to constant parametrised categories, we obtain a $\baseCat$-class of cocones $\udl{\R}$ on the parametrised category $\udl{\Delta}^1$ with $V$-cocones as described above. 
    The full $\baseCat$-subcategory 
    \[
        \udl{\presheaf}(\udl{\Delta}^1,\udl{\RR}) \subset \udl{\presheaf}(\udl{\Delta}^1)
    \]
    is the $\baseCat$-category of those $V$-presheaves $V^*\prn{\udl{\Delta}^1}{\op} \to V^*\GSpaces$ sending $1$ to the terminal object, i.e. the pointed anim\ae{}, under the description of \Cref{recollect:localisation_cocompletion}.
     
    Now let $\udl{\ast}$ denote the class of cocones on $\udl{\Delta}^0$ only containing the constant cocone the unique object. The inclusion  $0\colon \udl{\Delta}^0\rightarrow \udl{\Delta}^1$  carries $\udl{\ast}$ to $\udl{\RR}$, and we claim the induced map after applying $\udl{\presheaf}(-)$ to the morphism $(\udl{\Delta}^0,\udl{\ast})\rightarrow (\udl{\Delta}^1,\udl{\RR})$ in $\Catcocone$  identifies with $(-)_+ \colon \GSpaces \to \GBSpaces$; indeed, under the equivalence $\udl{\presheaf}(\udl{\Delta}^1,\udl{\RR}) \simeq \GBSpaces$ the image of $0$ under $\udl{\Delta}^1 \xrightarrow{\yoneda} \udl{\presheaf}(\udl{\Delta}^1) \xrightarrow{L} \udl{\presheaf}(\udl{\Delta}^1,\RR)$ is sent to $S^0$, since
    \[
    \myuline{\Map}_{\udl{\presheaf}(\udl{\Delta}^1,\udl{\RR})} (L\yoneda(0),X) \simeq X(0),
    \]
    so that $L\yoneda(0)$ corepresents the underlying $\baseCat$-anima $\baseCat$-functor, as does $S^0$.
\end{construction}

\begin{lemma}\label{lem:ordinary_univ_prop_pointed_anima}
    Let $\udl{\sC}$ be a cocomplete $V$-category  with initial $\baseCat$-object $\emptyset_{\sC}$. Pullback along the $V$-functor $\udl{\Delta}^1\rightarrow \udl{\Spaces}_*$ picking out the $V$-morphism $S^0\rightarrow \ast$ induces an equivalence 
    \[
        \udl{\func}^{\colim}(\udl{\Spaces}_*,\udl{\sC})\xrightarrow{\;\;\;\; \sim \;\;\;\;} \CC_{/\emptyset_{\sC}}.
    \]
\end{lemma}
\begin{proof} 
    The lemma follows by the string of equivalences
    \[\udl{\func}^{\colim}(\udl{\Spaces}_*,\udl{\sC})\simeq \udl{\func}^{\colim}\prn{ \udl{\presheaf}\prn{\udl{\Delta}^1,\udl{\RR}},\udl{\sC}}\xrightarrow[\sim]{\;\;\; \yonedaMap^* \;\;\;} \udl{\func}^{\udl{\RR}}\prn{\udl{\Delta}^1,\udl{\sC}}\simeq \CC_{/\emptyset_{\sC}};
    \]  
    for the last equivalence, taking $\udl{\RR}$ to colimits reduces to the non-constant $V$-cocone $1\colon \udl{*} \simeq \emptyset^\triangleright \rightarrow V^*\udl{\Delta}^1$ being taken to a colimit cocone, i.e. that $1$ is sent to $\emptyset_{V^*\sC}$.
\end{proof}
\begin{proposition}\label{lem:distributivity_pointed_cocompletionB}
    Let $w\colon W \rightarrow V$ be an $\NN$-admissible map of orbit-finite $\baseCat$-\animae{}. There exists an equivalence $w_{\otimes}W^*\GBSpaces  \xrightarrow{\sim} V^*\GBSpaces$  participating in a commuting square
    \begin{equation}\label{comm-diagram-dist-based}
   \begin{tikzcd}[ampersand replacement=\&,column sep=large,row sep=scriptsize]
	{w_\otimes W^\ast \GSpaces} \& {V^\ast \GSpaces} \\
	{w_\otimes W^\ast \GBSpaces.} \& {V^\ast \GBSpaces}
	\arrow["\sim"', from=1-1, to=1-2]
	\arrow["{w_\otimes W^\ast (-)+}"', from=1-1, to=2-1]
	\arrow["{V^\ast(-)_+}", from=1-2, to=2-2]
	\arrow["\sim"', from=2-1, to=2-2]
\end{tikzcd}\end{equation}
    where the top horizontal equivalence is the $\NN$-mode $w$-multiplication $V$-functor of $\GSpaces$.
    Hence, for any cocomplete $V$-category $\CC$ with initial object $\emptyset_{\sC}$, pullback along the $\baseCat$-functor  $\udl{\Delta}^1\rightarrow  w_*\GBSpaces$ classifying the unique $\baseCat$-map $ S^0\rightarrow \ast$ induces an equivalence
    \[(S^0\rightarrow \ast)^*\colon \udl{\func}^{\delta_w}(w_*\GBSpaces,\udl{\C})\xrightarrow{\;\;\; \;\;\sim \;\;\;\;\;}\CC_{/\emptyset_{\sC}}.\]
\end{proposition}
\begin{proof}
    As usual, everything in sight is restriction-stable and is reduced easily to the case that $V$ is an orbit, so we will assume without loss of generality $V \in \baseCat$ is terminal. We adopt the notation from \Cref{const:pointed}.

    Consider the following commutative square in $\prn{\Catcocone}^V$:
    \[ \begin{tikzcd}
        (\udl{\Delta}^0,\ast) \ar[r,"\sim"] \ar[d,"{1}"] & w_\otimes W^\ast (\udl{\Delta}^0,\ast) \simeq (\udl{\Delta}^0,\udl{\ast}) \ar[d,"{w_\ast W^\ast(1)}"] \\
        (\udl{\Delta}^1,\udl{\RR}) \ar[r, "\eta"]  & w_\otimes W^\ast (\udl{\Delta}^1,\udl{\RR})\simeq (w_*W^*\udl{\Delta}^1,w_*\udl{\RR})  
    \end{tikzcd} \]
    where $\eta\colon \udl{\Delta}^1\rightarrow w_*W^*\udl{\Delta}^1$ is the $w^*\dashv w_*$ adjunction unit, which may easily be checked to map the set of cocones in $\udl{\RR}$ to one in $w_*W^*\udl{\RR}$.
    We claim that the bottom arrow is sent to an equivalence by $\udl{\presheaf}(-)\colon \Catcocone\rightarrow \udl{\presentable}^L$, so that we obtain the diagram \eqref{comm-diagram-dist-based} under the $\NN$-normed $\baseCat$-functor $\udl{\presheaf}(-)$ of \Cref{thm:catK_symMon}.
    
    We must show that for any cocomplete $\CC$ that restriction induces an equivalence
    \[\eta^*\colon  \udl{\Fun}^{w_\ast W^\ast \RR}(w_\ast W^\ast \udl{\Delta}^1,\CC) \longrightarrow \udl{\Fun}^{\RR}(\udl{\Delta}^1,\CC). \]
    First note that the unit map $\eta$ is fully faithful: indeed, by the double coset formula on the indexed product $w_*W^*\udl{\Delta}^1$, the unit map $\eta\colon \udl{\Delta}^1\rightarrow w_*W^*\udl{\Delta}^1$ is given fibrewise by the diagonal maps $\Delta^1 \rightarrow \prn{\Delta^1}^n$, which are fully faithful. 
    Consequently, the right Kan extension $\eta_* \colon \udl{\Fun}(\udl{\Delta}^1,\CC) \rightarrow \udl{\Fun}(w_\ast W^\ast \udl{\Delta}^1,\CC)$ along $\eta$ is a fully faithful functor.  We claim that this restricts to an essentially surjective functor -- and hence an equivalence -- from $\udl{\Fun}^{\udl{\R}}(\udl{\Delta}^1,\CC)$ onto $\udl{\Fun}^{w_\ast W^\ast\udl{\R}}(w_\ast W^\ast\udl{\Delta}^1,\CC)$, providing an inverse to $\eta^*$.

    Recall again that $\udl{\Fun}^{\udl{\R}}(\udl{\Delta}^1,\CC)^V$ are precisely those $V$-functors which send $1\in\udl{\Delta}^1$ to the initial object of $\udl{\sC}^V$. By the pointwise formula for right Kan extension \Shaone{Thm.}{10.3}, the image $\eta_*\mid_{\udl{\Fun}^{\udl{\RR}}\prn{\udl{\Delta^1}, \CC}^V}$ thus consists precisely of those $V$-functors $F\colon w_* W^* \udl{\Delta}^1 \rightarrow \CC$ satisfying the following property:
    \begin{enumerate}[label=($\star$)]
        \item for all $U \rightarrow V$, $F^U$ sends all but the initial $U$-object $0\in \prn{w_*W^*\udl{\Delta}^1}^U$  to $\emptyset_{U^* \CC} \in \CC^U$. 
    \end{enumerate}
  Therefore, to obtain the sought-after essential surjectivity, we just need to show that $F$ lies in $\udl{\Fun}^{w_\ast W^\ast\udl{\R}}(w_\ast W^\ast\udl{\Delta}^1,\CC)^V$ if and only if it satisfies property ($\star$).

    We may understand this combinatorially:
    under the base-change equivalence
    \[
        \prn{w_* \udl{\Delta}^1}^U\simeq \prn{u^*w_*\udl{\Delta}^1}^U\simeq \prn{\overline{w}_*\overline{u}^*\udl{\Delta}^1}^U\simeq \prn{\prod_{i=1}^n\overline{w}_{i*}\overline{u}^*_i\udl{\Delta}^1}^U \simeq \prn{\Delta^1}^n,
    \]
    every $U$-cocone in $w_* \RR$ has a corresponding three-step partition $(A,N,O) \vdash \cbr{1,\dots,n}$ supplying a $\overline{w}_*$-preimage with $i$th coordinate
    \[
        \prn{\partial_{A,N,O}}_i = \begin{cases}
            1\colon * \simeq \emptyset^{\triangleright} \rightarrow \Delta^1 & i \in A;\\
            \mathrm{const(0)}\colon *^{\triangleright} \rightarrow \Delta^1 & i \in Z; \;\;\; \text{and}\\
             \mathrm{const(1)}\colon *^{\triangleright} \rightarrow \Delta^1 & i \in O.
        \end{cases}
    \]
    Now, $F$ takes the $\overline{w}_*$-image of $\partial_{A,N,O}$ to a colimit $U$-cocone precisely when
    \begin{equation}\label{eq:an_}
    (\overline{w}_*\udl{K})\tcocone\xrightarrow{\;\;\; \canonical \;\;\;} \overline{w}_*(\udl{K}\tcocone) \xlongrightarrow{\;\;\; \overline{w}_*\partial_{A,N,O} \;\;\;} \overline{w}_*\overline{u}^*\udl{\Delta}^1 \simeq u^* w_* \udl{\Delta}^1
    \xlongrightarrow{ \;\;\;  u^*F \;\;\;}u^*\udl{\D}
    \end{equation}
    is a colimit $U$-cocone, i.e. $F$ distributes $\partial_{A,N,O}$.
    We explicitly describe this in two cases:
    \begin{enumerate}[label={\textit{Case (\roman*).}}]
        \item If $A = \emptyset$, then $\partial_{A,N,O}$ is constant, so distribution is automatic.
        \item If $A \neq \emptyset$, then $\overline{w}_*\udl{K} \simeq \emptyset$ and the composition $\canonical \circ \overline{w}_*\partial_{A,N,O}$ picks out the object   
        $x_N \in \prn{w_*\udl{\Delta}^1}^U \simeq \prn{\Delta^1}^n$ whose $i$th coordinate is 0 if $i \in N$ and 1 otherwise. \Cref{lem:distributing_over_initial_object} shows that $F$ distributes $\partial_{A,N,O}$ if and only if $F\prn{x_N} \simeq \emptyset_{U^* \CC}$.
    \end{enumerate}
    As $(A,N,O)$ is quantified over all three-step partitions, Case (ii) appears precisely for the proper subsets $N \subsetneq \cbr{1,\dots,n}$;
    thus $F$ distributes all $\partial_{A,N,O}$  precisely when it satisfies property ($\star$), yielding the claim and the first part of the proposition.

    The second part of the proposition is now an straightforward combination of the commuting square involving the universal distributive functors $\mu$
    \[
    \begin{tikzcd}
        \udl{\func}^L(w_{\otimes}W^*\GBSpaces,\CC)\dar["\mu^*", "\simeq"']\ar[rr,"\left(w_{\otimes}(-)_+\right)^*"] && \udl{\func}^L(w_{\otimes}W^*\udl{\Spaces},\CC)\dar["\mu^*","\simeq"']\\
        \udl{\func}^{\delta_w}(w_{\ast}W^*\GBSpaces,\CC) \ar[rr,"\left(w_{\ast}(-)_+\right)^*"]&& \udl{\func}^{\delta_w}(w_{\ast}W^*\udl{\Spaces},\CC)
    \end{tikzcd}
    \]
    with \Cref{lem:ordinary_univ_prop_pointed_anima} and the first part of the proposition.
\end{proof}

With this universal property, we can now conclude that $\udl{\Spaces}_*$ is an $\NN$-mode.

\begin{corollary}\label{prop:pointed_modes}
    The object $S^0 \in \GBSpaces[\baseCat]$ witnesses $\GBSpaces$ as an $\NN$-mode, and $\GBSpaces$-modules in $\udl{\presentable}^L$ are precisely the pointed presentable categories. Consequently, $\GBSpaces\otimes -\colon \udl{\presentable}^L \rightarrow \udl{\presentable}^L_*$ is an $\NN$-normed Bousfield localisation between distributive $\NN$-normed $\baseCat$-categories. 
\end{corollary}
\begin{proof}
    We first show that $(-)_+\colon \udl{\Spaces}\rightarrow \GBSpaces$ is an idempotent $\mathbb{E}_0$ algebra classifying $\udl{\presentable}^L_*$. 
    Observe that for any $\DD\in\presentable^L$,  $\GBSpaces\otimes \DD$  is pointed since \Cref{rmk:formula_for_presentable_tensors}'s  full embedding 
    \[
        \GBSpaces\otimes\DD\simeq \udl{\func}^{\mathrm{lim}}(\DD\op,\GBSpaces) \subset \udl{\func}(\DD\op,\GBSpaces)
    \]
    contains the zero object $\overline{0} \in \udl{\func}(\DD\op,\GBSpaces)$.
    On the other hand, for any $\DD\in\udl{\presentable}^L_*$, the functor $S^0\otimes \id_{\DD}\colon \DD \simeq \udl{\Spaces}\otimes \DD  \rightarrow \GBSpaces\otimes \DD$ is an equivalence, since for any other pointed $\EE \in\udl{\presentable}^L$, restriction is  the equivalence
    \begin{equation}\label{eqn:tearing_off_pointed_anima}
        (\id_{\DD}\otimes S^0)^*\colon \udl{\func}^L(\GBSpaces\otimes \DD,\EE)\simeq \udl{\func}^L(\DD,\udl{\func}^L(\GBSpaces,\EE))\xrightarrow[\sim]{\;\; (S^0)^* \;\;}\udl{\func}^L(\DD,\EE)
    \end{equation}
    where the last equivalence is by virtue of the universal property from \Cref{lem:ordinary_univ_prop_pointed_anima}. 
    Specialising this to the case $\DD=\GBSpaces$ shows that $(-)_+\colon \GSpaces\rightarrow \GBSpaces$ is idempotent, and, these three points imply that $\DD\in\udl{\presentable}^L$ is pointed if and only if the map $(-)_+\otimes \id\colon \DD \simeq \udl{\Spaces}\otimes \DD\rightarrow \GBSpaces\otimes \DD$ is an equivalence, so that $\udl{\module}_{\GBSpaces}(\udl{\presentable}^L)\simeq \udl{\presentable}^L_*$. 

    Next, we argue that $(-)_+$ is, in fact, $\NN$-idempotent.
    Given $w\colon W\rightarrow V$  an $\NN$-admissible map, we must show that the $\prn{\udl{\Pr}^L_*}^V$-map 
    \[
    \id_{\GBSpaces}\otimes (-)_+\colon V^*\GBSpaces \rightarrow V^*\GBSpaces\otimes w_{\otimes }W^*\GBSpaces
    \]
    is an equivalence. 
    Given $\EE\in \prn{\udl{\presentable}^L_{*}}^V$, combining the square \eqref{comm-diagram-dist-based} from \Cref{lem:distributivity_pointed_cocompletionB} with the equivalence of \eqref{eqn:tearing_off_pointed_anima} identifies the effect of $\udl{\func}^L(-,\EE)$ on $\id_{\GBSpaces}\otimes (-)_+$ with
    \[
        \left((-)_+\right)^*\colon \udl{\func}^L(V^*\GBSpaces,\EE)\longrightarrow \udl{\func}^L(V^*\udl{\Spaces},\EE);
    \]
    we have already established that this is an equivalence, so the result follows by Yoneda's lemma in $\prn{\udl{\presentable}^L_{*}}^V$.
    
    Finally, the last statement is an immediate consequence of the preceding paragraphs in conjunction with \Cref{thm:prL-colims,cor:contractibility-monoidal-structures}, combined with distributivity-preservation of $\NN$-normed Bousfield localisations under \Cref{prop:transf_examples}.
\end{proof}

Next, we record the fact that we may explicitly compute norms in $\GBSpaces$ as certain colimits of indexed products.
\begin{proposition}\label{prop:formula_for_norms_in_pointed_anima}
    Let $w\colon W \rightarrow V$ be an $\NN$-admissible map and $(x\colon \ast\rightarrow X)\in W^*\udl{\Spaces}_*$. Then $w_{\otimes}X\in V^*\udl{\Spaces}_*$ is computed as the total cofibre of the cubical diagram
    \[w_*\udl{\Delta}^1\xlongrightarrow{w_*x=w_*(\ast\rightarrow X)} \udl{\Spaces}_*.\]
\end{proposition}
\begin{proof}
    We show that this is a formal consequence of the fact that $(-)_+\colon \udl{\Spaces}\rightarrow \udl{\Spaces}_*$ is $\baseCat$-normed. To wit, first note that $X\in W^*\udl{\Spaces}_*$ may be written as $\cofib(S^0\rightarrow X_+)$, that is, it is a cofibre of a map coming from $\udl{\Spaces}$ via the functor $(-)_+$. We thus get 
    \[
        w_{\otimes}X\simeq {w}_{\otimes}\cofib(S^0\rightarrow X_+) \simeq \colim_{w_*\udl{\Lambda}^2_0}\
        {w}_{\otimes}(S^0\rightarrow X_+)\simeq \colim_{w_*\udl{\Lambda}^2_0}\big({w}_{*}(\ast\rightarrow X)\big)_+
    \] where the second equivalence is by distributivity of ${w}_{\otimes}$ and the third is since $(-)_+$ is $\NN$-normed.  
    By \cref{prop:norming_pushouts},  this may be computed as the pushout
    \[
    \begin{tikzcd}[row sep=scriptsize]
        \colim_{w_*\udl{\Delta}^1\backslash 1}w_{\otimes}(S^0\rightarrow X_+) \rar\dar \ar[dr, phantom, "\ulcorner" very near end]&  {w}_{\otimes}(X_+) \dar \\
        0 \rar & {w}_{\otimes}X.
    \end{tikzcd}
    \] 
    The top horizontal arrow may be rewritten as $\colim_{w_*\udl{\Delta}^1\backslash 1}\big({w}_{*}(\ast\rightarrow X)\big)_+ \rightarrow ({w}_{*}X)_+$ since $(-)_+$ is $\NN$-normed; hence $w_{\otimes}X$ is indeed the total cofibre as claimed.
\end{proof}

\begin{corollary}\label{cor:fixed-of-norm}
    Let $w\colon W \rightarrow V$ be an $\NN$-admissible map with $V$ an orbit and fix $(x\colon \ast\rightarrow X)\in W^*\udl{\Spaces}_*$.
    Then, there exists a natural equivalence
    \[
        \prn{{w}_{\otimes} X}^V \simeq \bigwedge_{U \in \Orb(W)} X^U.
    \]    
\end{corollary}
\begin{proof}
    By combining the fixed-points-of-colimits formula of \Cref{lem:cofree-colims} with \Cref{prop:formula_for_norms_in_pointed_anima}, we find that $\prn{{w}_{\otimes} X}^V$ is the total cofibre of the cubical diagram
    \[
        \prod_{U \in \Orb(W)} \Delta^1 \simeq \prn{w_* \udl{\Delta}^1}^V \longrightarrow  \int\prn{w_* \udl{\Delta}^1} \longrightarrow \BSpaces
    \]
    which, upon unwinding definitions, is seen to correspond with the object $\prn{X^U}_{U\in\Orb(W)} \in \prod_{U\in\Orb(W)}\BSpaces$;
    another aplication of \Cref{prop:formula_for_norms_in_pointed_anima} implies the claim. 
\end{proof}

\subsection{Formal inversion modes}\label{subsec:formal_inversions}
We now turn to the matter of formal inversions. 
The main result of this subsection will be \Cref{thm:invert_object_idempotent}, which will demonstrate that formal inversion often preserves the property of being an $\NN$-mode. 
We begin by recalling Cnossen's formulation \cite[Def. 2.35]{Cnossen2023} of it in the parametrised setting, which is built upon Robalo's theory \cite[\textsection 2]{robaloKTheoryBridge}.

\begin{definition}\label{recall:formal_inversion}
    Let $\udl{\sC}\in\calg\prn{\udl{\presentable}^L}^{\baseCat}$ and $\invObj\subseteq \udl{\sC}$ be a $\baseCat$-set of objects. A map $\udl{\sC}\rightarrow \udl{\D}$ in $\calg\prn{\udl{\presentable}^L}^{\baseCat}$ is said to be a \mdef{formal inversion of $\invObj$} if for all $\udl{\E}\in\calg\prn{\udl{\presentable}^L}^{\baseCat}$, the  map $\Map_{\calg\prn{\udl{\presentable}^L}^{\baseCat}}(\udl{\D},\udl{\E})\rightarrow \Map_{\calg\prn{\udl{\presentable}^L}^{\baseCat}}(\udl{\sC},\udl{\E})$ on non-parametrised mapping \animae{} is a monomorphism onto the components of the morphisms $\udl{\sC}\rightarrow\udl{\D}$  \mdef{inverting objects in $\invObj$}, i.e. such that  $\udl{\sC}^V\rightarrow \udl{\E}^V$ sends objects in $\invObj^V$ to $\otimes$-invertible objects in $\udl{\E}^V$ for all $V\in \baseCat$.
\end{definition}
Observe that, by the universal property, if there exists a formal inversion $\udl{\sC}\rightarrow\udl{\D}$ for $\invObj$, then it must be unique;
in this case, we will use the fixed notation
\[
    \Sigma^\infty_{\invObj}\colon \CC \longrightarrow \CC\brk{\invObj^{-1}} \coloneqq \DD.
\]
The following notions single out favourable situations in which the formal inversions admit usable and concrete descriptions.
\begin{definition}
    Let $\udl{\sC}\in\calg\prn{\udl{\presentable}^L}^{\baseCat}$ and let $\invObj\subseteq \udl{\sC}$ be a $\baseCat$-set of objects. 
    
    We say that $\invObj$ is \mdef{generated in top-degree} if for every map $u\colon U \rightarrow V$ in $\baseCat$ and every $X\in \invObj^U$, there exists a $Z\in\udl{\sC}^W$ and $Y\in \invObj^V$ such that $X\otimes Z\simeq u^*Y\in\udl{\sC}^U$.

    We say that $\invObj$ \mdef{satisfies the 3-cyclic condition} if for all $V\in \baseCat$ and all $X\in \invObj^V$, the cyclic permutation automorphism $(1\: 2\: 3)\colon X^{\otimes 3}\rightarrow X^{\otimes 3}$ in $\udl{\sC}^V$ is homotopic to the identity.
\end{definition}

The work of Cnossen and Robalo now ensures the following.
\begin{lemma}\label{lem:abstract_formal_inversion}
    Let $\CC \in \calg\prn{\udl{\presentable}^L}^{\baseCat}$ and let $\invObj\subseteq \udl{\sC}$ be a $\baseCat$-set of objects. 
    \begin{enumerate}[label=(\arabic*)]
        \item  If $\invObj$ is generated in top-degree, then the $\baseCat$-functor $\udl{\sC}\rightarrow \DD$ which at level $V$ is given by a formal inversion $\Sigma^\infty_{\invObj^V}\colon \udl{\sC}^V\rightarrow \udl{\sC}^V[(\invObj^V)^{-1}]$ of $\invObj^V$ is a formal inversion of $\invObj$.
        \item If  $\invObj$ moreover satisfies the 3-cyclic condition, then for all orbits $V$, writing $I$ for the filtered poset of finite subsets of objects in $\invObj^V$ and $F^{\otimes}\in\udl{\sC}$ for the tensor product of a finite tuple of objects $F\in I$ in $\invObj^V$, the natural $V$-functor 
        \[
        \colim_{F\in I}\colim(V^*\udl{\sC}\xrightarrow{F^{\otimes}\otimes-}V^*\udl{\sC}\xrightarrow{F^{\otimes}\otimes-}\cdots) \longrightarrow V^*(\udl{\sC}[\invObj^{-1}]) \]
        is an equivalence.
    \end{enumerate}
\end{lemma}
\begin{proof}
    Part (1) is precisely  \cite[Prop. 2.46]{Cnossen2023}. For part (2), the first paragraph in the proof of \cite[Prop. 2.46]{Cnossen2023} together with \cite[Def. 2.37]{Cnossen2023} show that $V^*(\udl{\sC}[\invObj^{-1}])$ is equivalent to $V^*\udl{\sC}\otimes_{\udl{\sC}^V}(\udl{\sC}^V[(\invObj^V)^{-1}])$. But then by the 3-cyclicity hypothesis, we see by \cite[Rmk. 2.17 and Thm. 2.26]{robaloKTheoryBridge} (cf. also \cite[Prop. 1.6.3]{annalaIwasa} for a more recent and modern treatment) as well as the argument in \cite[\textsection 6.1]{hoyoisSixOperations} to deal with telescopically inverting multiple objects, that we obtain the analogous filtered colimit formula on $\udl{\sC}^V[(\invObj^V)^{-1}]\in \presentable^L$, and thus on $V^*(\udl{\sC}[\invObj^{-1}])$ as desired.
\end{proof}

\begin{definition}\label{defn:multiplicative_invertibles}
    Let $\MM\in \NAlg[\baseCat]{\udl{\presentable}^L}$ be an $\NN$-mode and $\invObj\subseteq \MM$ a small full $\baseCat$-subcategory. We say that $\invObj$ is \mdef{$\NN$-good} if it is generated in top-degree, satisfies the 3-cyclic condition, and is closed under $\NN$-norms. 
\end{definition}

\begin{remark}\label{rmk:inv_res}
    Part (1) of \Cref{lem:abstract_formal_inversion} constructs a natural equivalence  $W^* \prn{\CC\brk{\invObj^{-1}}} \simeq \prn{W^* \CC} \brk{W^* \invObj^{-1}}$
    under $W^* \CC$ when $\invObj$ is generated in top degree;
    hence we may unambiguously use the notation $W^* \CC\brk{\invObj^{-1}}$ for it.
    Moreover, $W^* \invObj \subset W^* \CC$ is generated in top degree, and
    if $\invObj$ satisfies the 3-cyclic condition, then so does $W^* \invObj$.
\end{remark}

\begin{theorem}\label{thm:invert_object_idempotent}
  Suppose $\invObj \subset \MM$ is an $\NN$-good full $\baseCat$-subcategory of an $\NN$-mode. 
  Then the composite $\GSpaces \to \MM \rightarrow \MM\brk{\invObj^{-1}}$ 
  exhibits the codomain as an $\NN$-mode
  such that $\CC\in\udl{\presentable}^L$ is an $\MM\brk{\invObj^{-1}}$-module if and only if it is an $\MM$-module and  $\invObj$ acts invertibly on $\CC$.
\end{theorem}
\begin{proof} 
    The hypotheses of the theorem are satisfied for the $\baseCat_{/V}$-subcategory $V^* \invObj \subset V^* \MM$ for all $V$ by \Cref{rmk:inv_res}, and the results for each $V^* \invObj \subset V^* \MM$ for all $V$ imply the results for  $\invObj \subset \MM$ by \Cref{rmk:cocart-sections}.
    Thus we may assume that $\baseCat$ has a terminal object.
    
    For $\NN$-idempotence, we are tasked with showing for any $\NN$-admissible map $w\colon W' \rightarrow V$ between orbit-finite $\baseCat$-\animae{} that the natural map 
    \[ 
        V^*(\MM[\invObj^{-1}])\longrightarrow V^*(\MM[\invObj^{-1}])\otimes {w}_{\otimes}W^*(\MM[\invObj^{-1}]) 
    \] 
    is an equivalence. 
    As usual, by basechanging to $V$, we may assume without loss of generality that $V$ is the terminal object, and so suppress the notation $V^*$ everywhere. 

    We factor the map in question as 
    \begin{align*}
        \MM[{\invObj^{-1}}]  
        & \xrightarrow{\;\;\; F_1\;\;\;} (\MM \otimes w_{\otimes} W^\ast \MM)[(\invObj \otimes w_{\otimes} W^\ast\invObj)^{-1}] \\
        &\xrightarrow{\;\;\; F_2 \;\;\;} \MM[\invObj^{-1}] \otimes (w_{\otimes} W^\ast \MM) [(w_{\otimes} W^\ast \invObj)^{-1}] \\
        & \xrightarrow{\;\;\; F_3 \;\;\;} \MM[\invObj^{-1}] \otimes w_{\otimes} \left(W^\ast\MM[\invObj^{-1}]\right)  
    \end{align*}
    where each map is induced by the universal property of formal inversion. We argue that each $F_i$ is an equivalence, after which we will have verified that $\MM\brk{\invObj^{-1}}$ is an $\NN$-mode:
    \begin{enumerate}[label={$F_\arabic*$.}]
        \item That $F_1$ is an equivalence follows before formal inversions by the assumed property that $\MM$ is a mode;
        moreover, since $\invObj$ is closed under $\NN$-norms, the multiplication equivalence $\MM \otimes w_{\otimes} W^* \MM \rightarrow \MM$ takes $\invObj \otimes w_{\otimes} W^* \invObj$ to a full subcategory of $\invObj$, and applying the unit in the second coordinate shows that it is all of $\invObj$; 
        this implies the equivalence after formal inversion.
        \item To see that $F_2$ is an equivalence, it suffices to note that mapping out of both sides represent distributive bifunctors out of $\MM\times w_{\otimes}W^\ast \MM$ which invert $\invObj\times w_{\otimes}W^\ast \invObj$, compatibly with $F_2$.
        \item To see that $F_3$ is an equivalence, note by \Cref{lem:abstract_formal_inversion} that we have a distinguished filtered (hence sifted) colimit presentation for $\NN$-good formal inversions;
        in particular, $F_3$ is the tensor product of the identity on $\MM \brk{\invObj^{-1}}$ with the assembly map
        \[
            \alpha\colon \colim_{F\in I}\colim\prn{w_{\otimes} W^* \MM\xrightarrow{F^{\otimes}\otimes-} \cdots} \longrightarrow w_{\otimes} \colim_{F\in I}\colim\prn{W^* \MM\xrightarrow{F^{\otimes}\otimes-}\cdots} 
        \]        
        \Cref{thm:prL-colims} verifies that the $\NN$-normed structure on presentable categories is distributive, so \Cref{prop:dist-lemmas} verifies that $w_{\otimes}$ preserves sifted colimits, so $\alpha$, and hence $F_3$, is an equivalence.
    \end{enumerate}
    
    We now classify the $\MM\brk{\invObj^{-1}}$-modules among $\MM$-modules. On the one hand, if an $\MM$-module $\CC$ is an $\MM\brk{\invObj^{-1}}$-module and $X \in \invObj^V$, then the $\MM$-module $X$-action $V$-endofunctor on $\CC$  agrees with that as an $\MM\brk{\invObj^{-1}}$-module, which is invertible because $X \in \MM\brk{\invObj^{-1}}^V$ is invertible.
    On the other hand, when $\invObj^V$ acts invertibly on $\CC$, the tensored unit $\baseCat$-functor admits a factorisation
     \[\begin{tikzcd}[ampersand replacement=\&, column sep=huge, row sep=scriptsize]
	{\colim_{F \in I} \colim\prn{\CC \xrightarrow{F^{\otimes} \otimes -} \cdots}} \& {\colim_{F \in I} \colim_{\mathbb{N}_{\leq}} \CC} \\
	{\colim_{F \in I} \colim\prn{\CC \otimes \MM \xrightarrow{F^{\otimes} \otimes -} \cdots}} \& \CC \\
	{\CC \otimes \colim_{F \in I} \colim\prn{\MM \xrightarrow{F^{\otimes} \otimes -} \cdots}} \& {\CC \otimes \MM\brk{\invObj^{-1}}}
	\arrow["{\text{\tiny $\CC$ is $\MM$-module}}"', "\sim",from=1-1, to=2-1]
	\arrow["{ \text{\tiny $\invObj^V$ acts invertibly}}"', "\sim",from=1-2, to=1-1]
    \arrow[from=2-2, to=1-2, "\sim", "{\text{\tiny filt implies weakly contractible}}"']
	\arrow["{\text{\tiny \ref{thm:prL-colims}}}"', "\sim",from=2-1, to=3-1]
	\arrow[from=2-2, to=3-2]
	\arrow["{\text{\tiny \ref{lem:abstract_formal_inversion}}}"', "\sim", from=3-1, to=3-2]
\end{tikzcd}
\]
where each arrow marked with $\sim$ is an equivalence for the indicated reason; the remaining arrows are equivalences by two-out-of-three, establishing the claim.
\end{proof}

\subsection{The equivariant case for compact Lie groups}\label{subsection:the_equivariant_case}
Henceforth fix $G$ a compact Lie group.
In this final subsection, we apply all the theory developed above to $G$-equivariant homotopy theory. Recall the notations of $(\baseCat,\NN)=(\Orb_G,\finind)$ from \Cref{example:orb_G}. 
We call $\Orbit{G}$-categories \mdef{G-categories}.
Recall that by Elmendorf's theorem, there is a Dwyer-Kan localisation $\topcat_\ast^G \to \GBSpaces[G]$ sending a based topological $G$-space to its underlying $G$-homotopy type.

The following definition appears in \cite[App.~C]{gepnerMeierTMF} and \cite[Def.~4.1]{Cnossen2023}:
\begin{definition}
    Let $\myuline{\repsphere} \subset \GBSpaces$ denote the full $G$-subcategory spanned by \mdef{representation spheres};
    that is, its $H$-objects are represented by  pointed topological $H$-spaces  of the form $S^V$ for $V$ a finite-dimensional orthogonal $H$-representation.

    The $G$-category of \mdef{genuine $G$-spectra} is defined to be the formal inversion 
    \[ \myuline{\spectra} = \GBSpaces\left[\myuline{\repsphere}^{-1}\right] \,. \]
    The universal left adjoint $G$-functor into this formal inversion is denoted by $\Sigma^\infty \colon \GBSpaces \to \myuline{\spectra}$. Precomposing this with $(-)_+ \colon \GSpaces \to \GBSpaces$ yields another left adjoint $G$-functor $\Sigma_+^\infty \colon \GSpaces \to \myuline{\spectra}$, called the \mdef{suspension $G$-spectrum}.
    The suspension $G$-spectrum of the $G$-equivariant point is called the \mdef{sphere $G$-spectrum} $\mathbb{S} \in \myuline{\spectra}^G$.
\end{definition}
As an ingredient to deal with compact Lie groups which are not finite, we will need some precise control on the multiplicative norms of representations and orbits. 
This is an unavoidably geometric question, and so since the methods are rather orthogonal to the article, we have collected all matters pertaining to this in  Appendix \ref{section:model_cats}.

The following is the main result of this section.
\begin{theorem}\label{thm:compact_lie_norms}
  Let $G$ be a compact Lie group. Then the $G$-category of genuine $G$-spectra $\myuline{\spectra}$ is presentable, and the sphere $G$-spectrum $\mathbb{S} \in \myuline{\spectra}^G$ renders it an $\finind$-mode. 
    Hence the anima of distributive $\finind$-normed structures on $\myuline{\spectra}$ compatible with the suspension $G$-spectrum $G$-functor $\Sigma^\infty_+\colon \udl{\Spaces}_G\rightarrow \myuline{\spectra}$ (i.e. with $G$-unit $\mathbb{S}$) is contractible.
\end{theorem}
To prove this, we first note the following standard observation.
\begin{lemma}\label{lem:inv-vs-hsphere}
    An object $X \in \myuline{\spectra}^G$ is $\otimes$-invertible if and only if it is compact and $\Phi^H X$ is a sphere for all closed subgroups $H \leq G$.
\end{lemma}
\begin{proof}
  The proof used in \cite[\S2]{krause20} works in this generality:
  since $\mathbb{S}$ is compact,
  all dualisable objects of $\myuline{\spectra}^G$ are compact, so assume without loss of generality that $X$ is compact. Then $X$ is $\otimes$-invertible if and only if the coevaluation map $X \otimes \mathbb{D}X \rightarrow \mathbb{S}$ is an equivalence;
  this is taken to coevaluation under the conservative symmetric monoidal functor $\prn{\Phi^H(-) \mid H \leq G}\colon  \myuline{\spectra}^G \longrightarrow \prod_{H \leq G} \spectra$, yielding the claim.
\end{proof}

We will use this and the following to identify $\myuline{\spectra}$ with another formal inversion.
\begin{lemma}\label{lem:identify-inversions}
  Suppose $\DD \subset \EE \subset \CC$ are $\TT$-sets of objects in a presentable $\baseCat^{\simeq}$-normed $\TT$-category $\CC$ whose  composite $\EE \hookrightarrow \CC \rightarrow \CC\brk{\DD^{-1}}$
  takes values in $\otimes$-invertible objects.
  Then, the induced $\baseCat^{\simeq}$-normed $\baseCat$-functor $\CC\brk{\EE^{-1}} \rightarrow \CC\brk{\DD^{-1}}$ is an equivalence.
\end{lemma}
\begin{proof}
  It suffices to note that a $\baseCat^{\simeq}$-normed $\baseCat$-functor $\CC \rightarrow \BB$ inverts $\EE$ if and only if it inverts $\DD$;
  the backwards direction follows by factorisation through $\CC\brk{\DD^{-1}}$.
\end{proof}
\begin{remark}\label{rmk:point-set-smash}
    We will use a point-set model for the binary smash product of pointed $G$-\animae{}: \Cref{prop:formula_for_norms_in_pointed_anima} identifies the smash product within a cofibre sequence
    \begin{equation}\label{eq:smash-cofibre}
        X \vee Y \longrightarrow X \times Y \longrightarrow X \wedge Y
    \end{equation}
    in $\GBSpaces[G]$.
    Presenting $\GBSpaces[G]$ as via the standard model structure on $\mathrm{Top}_{G,*}$, whenever the base-point inclusions $* \hookrightarrow X$ and $* \hookrightarrow Y$ are $G$-NDR pairs, so is $X \vee Y \hookrightarrow X \times Y$;
    hence it is a cofibration, and \eqref{eq:smash-cofibre} is the image of a cofibre sequence in $\mathrm{Top}_{G,*}$. 
\end{remark}

\begin{proof}[Proof of \Cref{thm:compact_lie_norms}]
    Let $\myuline{\hosphere}\subseteq \udl{\Spaces}_*$ be the full $G$-subcategory of compact objects in $\udl{\Spaces}_*$ all of whose fixed points are spheres and who satisfy the 3-cyclic condition.
    
    By \Cref{cor:norms-finite-cats},
    compact objects of $\udl{\Spaces}_*$ are closed under the $\finind$-norms; 
    moreover, symmetric monoidality of $w_{\otimes}$  shows that 3-cyclic objects of $\GBSpaces$ are closed under $\finind$-norms.
    Thus, to argue that $\myuline{\hosphere}$ is $\finind$-good in the sense of \Cref{defn:multiplicative_invertibles},  it suffices to show that the fixed-point-sphere condition is closed under $\finind$-norms.
    But this follows under \Cref{cor:fixed-of-norm} since spheres are closed under smash product;
    therefore \Cref{thm:invert_object_idempotent} shows that $\udl{\Spaces}_*[\myuline{\hosphere}^{-1}]$ is an $\finind$-mode. 

    Note that the objects of $\myuline{\repsphere}$ are compact;
    moreover, the point-set smash product formula \Cref{rmk:point-set-smash} gives the usual point-set model for the smash product of them, so \cite[Lem. C.5]{gepnerMeierTMF} shows that they satisfy the 3-cyclic condition;
    hence we have an inclusion $\myuline{\repsphere}\subseteq \myuline{\hosphere}$.
    By \Cref{lem:inv-vs-hsphere}, we know that homotopy spheres are already invertible in $\myuline{\spectra}$, so \Cref{lem:identify-inversions} shows that $\myuline{\spectra}\coloneqq \udl{\Spaces}_*[\myuline{\repsphere}^{-1}]\rightarrow \udl{\Spaces}_*[\myuline{\hosphere}^{-1}]$ is an equivalence,
    implying that $\myuline{\spectra}$ is a $\finind$-mode, as desired. 
\end{proof}

To finish, we would like to build a bridge from this to stability of small $G$-categories. 
\begin{remark} \label{example:Lie-group}
    For $G$ a compact Lie group, \Cref{cor:norms-finite-cats} demonstrates that $\aleph_0$ is $\finind$-normable. 
    Condition (a) of \Cref{prop:kappa-norms} shows that $\NN$-normability of regular $\kappa$ is equivalent to having larger cardinality than the minimal $\NN$-normable cardinal, so this implies that \textit{every} regular cardinal is $\finind$-normable.
    In particular, \Cref{thm:ind} applies for $\finind$-normed $G$-categories with no restriction on $\kappa$.
\end{remark}
Note that the $\aleph_0$-compact objects in the fibres $\myuline{\spectra}^H$ are precisely the dualisable $H$-spectra, so they are preserved by the (symmetric monoidal) restriction and norm functors;
\Cref{lem:res-of-compact} shows that these agree with the $\aleph_0$-compact $H$-objects, so that $\myuline{\spectra}^{\aleph_0} \subset \myuline{\spectra}$ is a $\finind$-normed $G$-subcategory.
We learn that $\myuline{\spectra}$ lifts to an idempotent $\finind$-normed algebra in $\udl{\presentable}^L_{\aleph_0}$. 
By \Cref{cor:contractibility-monoidal-structures}, we thus obtain $\finind$-normed structures on 
$\udl{\presentable}^L_{\mathrm{st}}\coloneqq \udl{\module}_{\myuline{\spectra}}(\udl{\presentable}^L)$ and 
$\udl{\presentable}^L_{\aleph_0, \mathrm{st}}\coloneqq \udl{\module}_{\myuline{\spectra}}(\udl{\presentable}^L_{\aleph_0})$.
This transfers:


\begin{lemma}\label{rmk:compact_mode}
    Fix $\kappa$ a $\baseCat^{\simeq}$-restrictible cardinal.
    Suppose $\MM$ is an idempotent commutative algebra in $\prn{\udl{\Pr}^L_{\kappa}}^{\baseCat}$ and $\invObj \subset \MM$ a $\baseCat$-set of $\kappa$-compact objects.
    Given $\CC \in (\udl{\Cat}^{\mathrm{irex}(\kappa)})^{\baseCat}$,  the following are equivalent:
    \begin{enumerate}[label={(\alph*)}]
        \item $\CC$ is an $\MM\brk{\invObj^{-1}}^{\kappa}$-module;
        \item $\udl{\ind}^{\kappa}(\CC)$ is an $\MM\brk{\invObj^{-1}}$-module;
        \item $\udl{\ind}^{\kappa}(\CC)$ is an $\MM$-module with invertible $\invObj$-action; and
        \item $\CC$ is an $\MM^{\kappa}$-module with invertible $\invObj$-action.
    \end{enumerate}
\end{lemma}
\begin{proof}
    (a) $\iff$ (b) is immediate from \Cref{thm:ind} and (b) $\iff$ (c) from \Cref{thm:invert_object_idempotent}.
   Conditions (c) and (d) are equivalent under \Cref{thm:ind} since the $\invObj$-action endofunctors on $\udl{\ind}^{\kappa}(\CC)$ are the $\udl{\ind}^{\kappa}$-image of those on $\CC$ by $\kappa$-compactness of $\invObj$.
\end{proof}

We write $\udl{\Cat}\irex$ for $\udl{\Cat}^{\mathrm{irex}(\aleph_0)}$ and $\udl{\cat}\perfect \subset \udl{\cat}\irex$ for the full $G$-subcategory of $G$-stable categories, i.e. those on which tensoring with representation spheres are automorphisms. We call these $G$-categories \mdef{perfect}. By \Cref{thm:compact_lie_norms,rmk:compact_mode}, we know that $\myuline{\spectra}^{\aleph_0}$ is an $\NN$-idempotent algebra and that it classifies precisely the property of $G$-stability. As such, we learn that  $\udl{\cat}\perfect = \udl{\module}_{\myuline{\spectra}^{\aleph_0}}(\udl{\cat}\irex)\hookrightarrow \udl{\cat}\irex$ is adjoint to  an $\finind$-normed smashing localisation. 
We summarise this situation  as follows:
\begin{corollary}\label{prop:catperf_G}
    Let $G$ be a compact Lie group. 
    Then, $\myuline{\spectra}^{\aleph_0} \otimes - \colon \udl{\Cat}\irex \rightarrow \udl{\Cat}\perfect$ is a $\finind$-normed smashing localisation with the following universal property:
    for a finite-index closed subgroup inclusion $K \leq H$  and perfect $K$-category $\udl{\A}$, the norm is a perfect $H$-category $w_{\otimes}\udl{\A}$ equipped with a $(w,\finsmall)$-distributive functor $\mu\colon w_*\udl{\A}\rightarrow w_{\otimes}\udl{\A}$ 
    such that, for all perfect $H$-categories $\udl{\B}$, restriction induces an equivalence
        \[
            \mu^*\colon \udl{\func}\rexact(w_{\otimes}\udl{\A},\udl{\B})\xrightarrow{\;\;\;\; \sim \;\;\;\;} \udl{\func}\distributeSmall{w}(w_*\udl{\A},\udl{\B}).
        \]
        Moreover, fibrewise Ind-completion underlies an $\finind$-normed equivalence
        \[
            \udl{\ind}\colon \udl{\cat}\perfect \xrightarrow{\;\;\;\;\; \sim \;\;\;\;\;} \udl{\Pr}^L_{st,\aleph_0} \subset \udl{\Pr}^L
        \]
        under the $G$-subcategory normed structure on $\udl{\Pr}^L_{st,\aleph_0}$.
\end{corollary}

To conclude, we give an example of an $\NN$-mode for $\Orbit{G}^{\simeq} \subsetneq \NN \subsetneq \finind$ which is obstructed from becoming a $\finind$-mode, following Hill's seminal work on transfer systems.
Fix a prime $q$, so that we may work $q$-locally.
Recall from \cite[Cor.~3.7]{Barthel_compact} that the product of geometric fixed points
\begin{equation*}\label{eq:total-fp}
    \prod_{(H) \in \mathrm{Sub}(G)/G} \Phi^H\colon {\spectra}^G_{(q)} \longrightarrow \prod_{H} \spectra_{(q)}
\end{equation*}
restricts to an \textit{embedding} from the poset of thick $\otimes$-ideals of compact objects in the former to the latter.
Devinatz--Hopkins--Smith's thick subcategory theorem \cite[Thm.~7]{dhs} then identifies these with dimension functions under the pointwise poset structure:
\[
    \ell\prn{-}\colon \mathrm{Thick}_{\otimes}\prn{{\spectra}_{(q)}^G}\hookrightarrow 
    C(G) = \mathbb{N}^{\mathrm{Ob} \prn{\mathrm{Sub}(G)/G}};
\]
these are easily further identified as the acyclic objects of finite localisations.

The proof of \cite[Cor.~II.3.85, Prop.~II.3.89]{nat-thesis} identifies $\NN$-normedness of finite localisations of $G$-spectra with $\NN$-compatibility of their acyclic objects, and \cite[Lem.~4.7]{Hill_chromatic} then explicitly characterises this via their dimension functions.
\cite[Cor.~II.3.65]{nat-thesis} recognises $\NN$-normed Bousfield localisations of $\NN$-modes as themselves $\NN$-modes;
hence, \textit{given} a chromatic finite localisation $\XX$ of $G$-spectra, we have acquired an $\NN$-mode $\GSpaces \rightarrow \udl{\XX}$, where $\NN$ is determined by combinatorial data of the function $\ell\prn{\XX}^{\bullet}\colon \mathrm{Sub}(G)/G  \rightarrow \mathbb{N}$.

In \cite[Cor.~1.4]{Barthel}, the authors characterise the image of $\mathrm{Thick}_{\otimes}\prn{{\spectra}^G_{(q)}}\subset C(G)$ when $G$ is abelian finite (and \cite[Thm.~1.3]{Barthel_compact} when $G$ is abelian compact Lie);
using this, we may illustrate normed chromatic modes via the following example.
\begin{example}
    Fix $G = C_{p^2}$.
    For $H \subset G$, denote by $T(k,H)$ the $G$-spectrum with
    \[
        T\prn{k,H}^{\Phi K} \simeq \begin{cases}
            T(k) & H = K\\
            0 & \mathrm{otherwise}
        \end{cases}
    \]
    with trivial Weyl group action.
    For any $k \in \mathbb{N}$ the chromatic finite localisation 
    \[
        \myuline{\spectra}_{(q)} \longrightarrow \udl{\XX} \coloneqq L^f_{T(k+1,e) \oplus T\prn{k, C_{p}} \oplus
        T\prn{k+1,C_{p^2}}}\myuline{\spectra}_{(q)}
    \]
    exhibits the codomain as an $\NN'$-mode if and only if $\NN' \subset \NN$, for $\NN$ the transfer system 
    \[\begin{tikzcd}[ampersand replacement=\&]
	{\NN\colon} \& {[C_{p^2}/e]} \& {[C_{p^2}/C_p]} \& {[C_{p^2}/C_{p^2}]}.
	\arrow[from=1-2, to=1-3]
	\arrow[curve={height=-16pt}, from=1-2, to=1-4]
\end{tikzcd}\]
    Indeed the $C_{p^2}$-dimension function $\prn{k+1,k,k+1}$  is realisable by a finite localisation under \cite[Thm.~1.3]{Barthel}, so it suffices to note that $\Phi^{J} \circ \mathrm{Nm}^{H}_{K}$ takes $\XX$-acyclic compact $K$-spectra to $L^f_{\ell(\XX)^J}$-acyclic spectra precisely when $K \hookrightarrow H \hookleftarrow J$ is not $C_{p} \hookrightarrow C_{p^2} \hookleftarrow C_p$.
\end{example}

\appendix

\section{Normed straightening and unstraightening}\label{sec:straightening_unstraightening}
For the duration of this section, we will freely use the language of algebraic patterns and fibrous patterns in the sense of  \cite{bhs25};
in particular, we will freely use the notion of \mdef{Segal fibration} to an algebraic pattern, i.e. a fibrous pattern which is cocartesian, and we will freely use the pattern structure on fibrous patterns. 
We will also refer to fibrous patterns over the obvious pattern structure on $\SFTN{\TT}{\NN}$ as \mdef{$\NN$-operads};
see \cite[\S~I.4,\S~II.1]{nat-thesis}.

\begin{notation}
    We write $\prn{\Cat[\baseCat]_{\NN-\otimes}}^{\cocart}_{/\CC} \subset \prn{\Cat[\baseCat]_{\NN-\otimes}}_{/\CC}$ for the (non-full) subcategory of $\NN$-normed cocartesian fibrations and functors preserving cocartesian lifts in the sense of \Cref{def:norm-cocart-fib}.
\end{notation}
\begin{theorem} \label{thm:normed-straightening}
    Let $\CC$ be an $\NN$-normed $\baseCat$-category. 
    There is an equivalence
    \[ 
        \mathrm{St}\colon \prn{\Cat[\baseCat]_{\NN-\otimes}}^{\cocart}_{/\CC} \xrightarrow{\;\;\;\;\; \sim \;\;\;\;\;} \Fun^{\mathrm{lax}-\NN-\otimes}(\CC,\GCat[\times]),
    \]
    natural in lax $\NN$-normed $\baseCat$-functors in $\CC$, such that 
    \begin{itemize}
        \item on underlying functors, $\mathrm{St}$ is given by the parametrised straightening, and
        \item for $\DD \to \CC$ an $\NN$-normed cocartesian fibration, $w \colon W \to V$ an $\NN$-admissible map of orbit-finite $\TT$-\animae{}, and $c \in \CC^W$, the lax structure map on the straightening $F \colon \CC \to \GCat$
        \[ 
            \ell_{w,c}\colon w_\ast F(c) \longrightarrow F(w_{\otimes}c) 
        \]
        has $W$-fixed points given by by restriction of the $w$-norm functor $w_{\otimes} \colon  \DD^W \to  \DD^V$ to the fibres over $c$ and $w_{\otimes} c$, respectively.
    \end{itemize}
\end{theorem}
We will need the following.
\begin{lemma}\label{lem:fact_cocart}
    Suppose $\pi\colon \A \rightarrow \B$ is a functor and $\B^L, \B^R \subset\B$ is a factorisation system
    Then, $\pi$ is a cocartesian fibration if and only if the following conditions hold:
    \begin{enumerate}
        \item $\pi$ has cocartesian lifts over $\B^L$, and
        \item The pullback $\pi|_{\B^R}$ of $\pi$ along $\B^R \to \B$ is a cocartesian fibration.
        \item for all commutative square in $\A$
        \[ \begin{tikzcd}
            A \ar[r] \ar[d] & B \ar[d] \\
            C \ar[r] & D
        \end{tikzcd} \] where the horizontal arrows are cocartesian lifts over $\B^L$ and the vertical arrows lie over $\B^R$: if the right arrow is $\pi|{_{\B^R}}$-cocartesian, then so is the left one.
    \end{enumerate}
\end{lemma}
\begin{proof}
    For the \enquote{only if}-direction, the first two conditions follow from pullback stability of cocartesian fibrations, while the last condition follows from the dual of \cite[Prop.~2.4.17]{lur09}.
    
    For the \enquote{if}-direction, we use the dual of \cite[Prop.~4.4]{hhln23a} (which does not use the adequate condition stated in there), to see that there is a factorisation system on $\A$ where the right maps are the ones sent to $\B^R$ and the left maps are cocartesian lifts over $\B^L$. Conditions $1$ and $2$ precisely translate to the functor being the dual notion of an op-Gray fibration as defined in \cite[Def.~5.14]{hhln23a} (replacing cartesian by cocartesian).

    The proposition now precisely translates to \cite[Lem.~5.17]{hhln23a}. The result therein is stated for adequate factorisation system but the adequate condition is not used in the proof.
\end{proof}

This allows us to prove the following.
\begin{lemma}\label{lem:two-cocarts}
    An $\NN$-normed $\baseCat$-functor $\pi\colon \DD \rightarrow \CC$ is an $\NN$-normed cocartesian fibration if and only if the unstraightening $\int \DD \rightarrow \int \CC$ is a (non-parametrised) cocartesian fibration.
\end{lemma}
\begin{proof}
    Note that $\int \CC\in \Cat$ admits a factorisation system whose lefthand part is cocartesian arrows over $\SFTN{\TT}{\NN}$ and righthand parts are the arrows in fibres  under \cite[Prop.~2.1.2.5]{lur17}.

    The restriction of $\pi$ to the preimage of $\int \CC^R$ is given by the disjoint union of $\DD^W \to \CC^W$ over all orbit-finite $\TT$-\animae{} $W$. 
    We have $\CC^W \simeq \prod_{U_i \in \mathrm{Orbits}(W)} \CC^{U_i}$, naturally in $\CC$, so that the functor $\DD^W \to \CC^W$ identifies with a product of various evaluations of the original functor. It therefore is a cocartesian fibration if the original $\baseCat$-functor is a cocartesian fibration,
    whence  condition 2 of \Cref{lem:fact_cocart} with the underlying $\baseCat$-functor being cocartesian.
    Condition 1 of \Cref{lem:fact_cocart} is identified with $\NN$-normedness under \kerodon{01TK}.
    Unwinding definitions identifies our closure hypothesis with condition 3 of  \Cref{lem:fact_cocart}, yielding the claim.
\end{proof}

\begin{notation}
     We write $U\CC$ for the underlying $\TT$-category of an $\NN$-normed $\baseCat$-category $\CC$.
\end{notation}

\begin{proof}[Proof of \Cref{thm:normed-straightening}]
    We will temporarily denote cocartesian unstraightenings over $\baseCat^{\op}$ as $\int_{\baseCat}$, and leave unstraightening over $\SFTN{\TT}{\NN}$ undecorated.
    Consider the following commutative diagram:
\[
\adjustbox{max width=\textwidth}{
  \begin{tikzcd}[ampersand replacement=\&]
	\& P \& {\mathrm{Seg}_{\int \CC}\prn{\Cat}} \& {\Fun^{\mathrm{lax}-\NN-\otimes}\prn{\CC, \GCat[\times]}} \\
	\& {\prn{\Cat}^{\cocart}_{/\int \CC}} \& {\Fun\prn{\int \CC, \Cat}} \\
	{\prn{\Cat[\baseCat]}^{\cocart}_{/U\C}} \& {\prn{\Cat}^{\cocart}_{/\int_{\baseCat} U\CC}} \& {\Fun\prn{\int_{\baseCat} U \CC, \Cat}} \& {\Fun\prn{U\CC, \GCat}}
	\arrow[from=1-2, to=1-3, "\sim"]
	\arrow[hook, from=1-2, to=2-2]
	\arrow["\lrcorner"{anchor=center, pos=0.125}, draw=none, from=1-2, to=2-3]
	\arrow[hook', from=1-3, to=2-3]
	\arrow["\sim"', from=1-4, to=1-3]
	\arrow[from=1-4, to=3-4]
	\arrow["\sim", from=2-2, to=2-3]
	\arrow[from=2-2, to=3-2]
	\arrow[from=2-3, to=3-3]
	\arrow["\sim", from=3-1, to=3-2]
	\arrow["\sim", from=3-2, to=3-3]
	\arrow["\sim", from=3-3, to=3-4]
\end{tikzcd}}
\] 

where the upper right equivalence is \cite[Thm.~II.2.1,Thm.~I.4.1,Prop.~I.3.4]{nat-thesis}, the lower left by \cite[Rmk.~7.4]{Sha23}, and the lower right by \cite[Thm.~7.8]{bdgns16}.
We must show that the pullback full subcategory $P \subset \prn{\Cat}^{\cocart}_{/\int \CC}$ of cocartesian fibrations $\pi\colon \int \DD \rightarrow \int \CC$ whose unstraightening is a Segal  $\int \CC$-object corresponds precisely with the $\NN$-symmetric monodial cocartesian fibrations;
then, by examining the above diagram and noting that the bottom horizontal composite is (by definition) parametrised straightening, our equivalence will satisfy the first point.

In the language of algebraic patterns, we find ourself with a triangle
\[\begin{tikzcd}[ampersand replacement=\&,column sep=large,row sep=tiny]
	{\int \DD} \&\& {\int \CC} \\
	\& {\SFTN{\TT}{\NN}}
	\arrow["{\pi}", from=1-1, to=1-3]
	\arrow["{\pi_{\DD}}"', from=1-1, to=2-2]
	\arrow["{\pi_{\CC}}", from=1-3, to=2-2]
\end{tikzcd}\]
such that $\pi_{\CC}$ is a cocartesian fibrous pattern.
By cancellation for fibrous patterns \cite[Prop.~4.1.16, Cor.~3.3.24]{bhs25}, we find that $\pi$ is a cocartesian fibrous pattern (i.e. Segal fibration \cite[Obs.~4.1.10,Lem.~4.1.15]{bhs25}) if and only if the following conditions hold:
\begin{enumerate}[label={(\alph*)}]
    \item $\pi$ is a cocartesian fibration, and
    \item $\pi_{\DD}$ is fibrous pattern.
\end{enumerate}  
    But by composing cocartesian fibrations, condition (b) may be replaced with $\pi_{\DD} $ being a Segal fibration.
    This identifies $P$ with the $\NN$-normed $\TT$-functors $\DD \rightarrow \CC$ becoming cocartesian fibrations under $\int$;
    these are precisely the $\NN$-normed cocartesian fibrations by \Cref{lem:two-cocarts}.

    We are left with the second point.
    Recall that the ``Segal object'' embedding, considered with codomain all functors, factors as follows:\[
        \Fun^{\mathrm{lax-\NN-\otimes}}\prn{\CC,\cofree\prn{\B}^\times} \rightarrow \Fun\prn{\int \CC, \int \cofree\prn{\B}^\times} \xrightarrow{\text{fixed pts}} \Fun\prn{\int \CC, \B}.
    \]
    Under the first functor, the image of the active arrow $a_{w,c}\colon c \rightarrow w_{\otimes c}$  induces the $W$-arrow $\ell_{w,c}\colon w_* F(c) \rightarrow F\prn{w_{\otimes}  c}$ under the universal property of cocartesian arrows;
    this arrow is further taken to the fixed points $\ell_{w,c}^W\colon \prn{w_* F(c)}^W \rightarrow F(w_\otimes  c)^W$ under the second functor.
    Writing $\widetilde F$ for the Segal object associated with $F$, $\ell_{w,c}^W$ is identified with the cocartesian transport functor of $\mathrm{Un}_{\int \C}(\widetilde F)$ from the fibre over $c$ along $a_{w,c}$.

    By compositionality of cocartesian lifts, this is moreover restricted from cocartesian transport over $w$ of the composite functor $\int \DD \rightarrow \int \CC \rightarrow \SFTN{\TT}{\NN}$, i.e. the $w$-tensor functor $w_{\otimes}\colon \DD^W \rightarrow \DD^V$, as desired.
\end{proof}

We will use this for a crucial construction, powered by the following.
\begin{lemma} \label{lem:pullback-partially-lax}
    Fix a cospan of lax $\NN$-normed functors $\CC \xrightarrow{F} \EE \xleftarrow{G} \DD$ such that $F$ is an $\NN$-normed cocartesian fibration.
    This diagram admits a pullback $\NN$-normed  $\baseCat$-category $\CC \times_{\EE} \DD$, whose underlying $\baseCat$-category is $U\CC \times_{U\EE} U\DD$;
    moreover, the resulting lax $\NN$-normed $\baseCat$-functor $\CC \times_{\EE} \DD \rightarrow \DD$  is an $\NN$-normed cocartesian fibration.
    
    Let $(c,d, F(c) \simeq G(d))$ be a $W$-object in $\CC \times_{\EE} \DD$ and $w \colon W \to V$ an $\NN$-admissible map. The $\NN$-normed structure on $F$ and lax $\NN$-normed structure on $G$ yield an $\EE^V$-morphism  
    \[
    F(w_{\otimes }c) \simeq w_{\otimes} F(c) \simeq w_{\otimes} G(d) \to G(w_{\otimes}d).
    \]
    We denote its $F$-cocartesian lift with domain $w_{\otimes}c$ by $f\colon w_{\otimes} c \to \widetilde{w_{\otimes}c}$;
    then $f$ is the lax structure map of  $\CC \times_{\EE} \DD \rightarrow \CC$, and 
     the norm of $(c,d,F(c) \simeq G(d))$ is given by $(\widetilde{w_{\otimes}}c,w_{\otimes}d , F(\widetilde{w_{\otimes}}c)\simeq G(w_{\otimes}d))$. 
\end{lemma}
\begin{proof}
    Consider the pullback $\CC \times_{\EE} \DD$ in $\NN$-operads. 
    Pullback-stability of cocartesian fibrations \kerodon{01UF} immediately implies that $\CC \times_{\EE} \DD \rightarrow \DD$ unstraightens to a cocartesian fibration, so that the pullback is $\NN$-normed and the functor an $\NN$-normed cocartesian fibration.
    The remaining claims follow straightforwardly by explicitly unwinding \kerodon{01UF}.
\end{proof}

\begin{construction} \label{con:free-normed-cocart}
    Let $\DD \to \CC$ be an oplax $\NN$-normed $\baseCat$-functor. 
    We will construct a lax $\NN$-normed structure on the $\baseCat$-functor
    \begin{align*}
        \CC & \longrightarrow \GCat \\
        c & \longmapsto \DD \times_{\CC} \overslice{\CC}{c}
    \end{align*}
    such that, for an $\NN$-admissible map $w \colon W \to V$ and a $W$-object $(d,F(d) \to c)$ in the opposite-unstraightening $\udl{\mathcal{P}}^W$, the norm is given by
    \[ w_{\otimes}\left( d,F(d) \to c \right) = \left( w_\otimes d, F(w_{\otimes}(d)) \to w_{\otimes} F(d) \to c)) \right) \]
    where the composite is obtained by composing the structure map from the oplax structure on $F$ with $w_{\otimes}$ applied to the original morphism.
    
    To do so, consider the diagram of $\NN$-normed $\baseCat$-categories
    \[ \begin{tikzcd}[row sep=scriptsize]
        \udl{\mathcal{P}} \ar[r] \ar[d] \arrow[rd, "\lrcorner" very near start, phantom] & \Arrow(\CC{\op}) \ar[d,"{t}"] \ar[r,"{s}"] & \CC{\op} \\
        \DD{\op} \ar[r,"{F{\op}}"] & \CC{\op}
    \end{tikzcd} \]
    where the square is the pullback provided in \Cref{lem:pullback-partially-lax}.    
    We claim that the top composite $\PP \to \Arrow(\CC^{\op}) \to \CC^{\op}$ lax $\NN$-normed functor is in fact $\NN$-normed. To see this, we use the description of the lax structure of $\PP \to \Arrow(\CC^{\op})$ from \Cref{lem:pullback-partially-lax}. The lax structure maps in $\Arrow(\CC^{\op})$ are $t$-cocartesian lifts. These lift are mapped to equivalences under the $\NN$-normed functor $s$, so that the composite functor is $\NN$-normed.
    
    Next, we claim that the opposite functor $\PP{\op} \to \CC$ is an $\NN$-normed cocartesian fibration. Fibrewise this is true by \cite[Thm.~1.2]{GepnerHaugsengNikolaus2017} and the fibrewise cocartesian arrows are those sent to equivalences by the left arrow in the diagram. As the left arrow is $\NN$-normed, we conclude that cocartesian arrows are indeed closed under restriction and norms, as desired.
    The description of the lax $\NN$-normed structure follows from combining the descriptions in \Cref{thm:normed-straightening} and \Cref{lem:pullback-partially-lax}.
\end{construction}

\section{Siftedness, distributivity, and Beck-Chevalley conditions}\label{sec:unicity}
\stoptocwriting
We will now compare our notion of distributivity to that found in \cite{elmantoHaugseng}, and hence \cite{llp25}.
This is founded upon the following Beck-Chevalley conditions.
\begin{construction}
Given a pair of arrows $A \xrightarrow{m} X \xrightarrow{x} Y$ in $\GSpaces[\TT]$ such that $x$ is $\NN$-admissible, we acquire an essentially unique \mdef{distributivity diagram}
\begin{equation}\label{eq:distributivity diagram}
\begin{tikzcd}[ampersand replacement=\&, row sep=tiny, column sep=large]
		\&\& {R} \& {Q} \\
	{\Delta_{m,x} } \& A \\
	\&\& X \& Y
	\arrow["{h}", from=1-3, to=1-4]
	\arrow["\epsilon"', from=1-3, to=2-2]
	\arrow["{\psi'}"', from=1-3, to=3-3]
	\arrow["\lrcorner"{anchor=center, pos=0.125}, draw=none, from=1-3, to=3-4]
	\arrow["\psi = x_* m", from=1-4, to=3-4]
	\arrow["\coloneqq"{description}, draw=none, from=2-1, to=2-2]
	\arrow["m"', from=2-2, to=3-3]
	\arrow["x", from=3-3, to=3-4]
\end{tikzcd}
\end{equation}
exhibiting $\psi$ as image of $m$ under the right adjoint of $x^\ast$, i.e. inducing an equivalence
\[
    \Map_{/Y}\prn{\varphi, \psi} \longrightarrow \Map_{/X}\prn{x^* \varphi, \psi^{\prime}} \xrightarrow{\;\;\; \epsilon^* \;\;\;} \Map_{/X}\prn{x^* \varphi, m}
\]
for all $\varphi\colon P \rightarrow Y$.
Given $\CC$ an $\NN$-normed $\baseCat$-category, we acquire a commutative diagram
\begin{equation}\label{eq:adjointable}
\begin{tikzcd}[ampersand replacement=\&, column sep=tiny]
	\& {\CC^Q} \&\&\& {\CC^{Y}} \\
	{\Psi_{m,x} \coloneqq} \& {\CC^{R}} \\
	\& {\CC^{A}} \&\&\& {\CC^{X}}
	\arrow["{{\psi^*}}"', from=1-5, to=1-2]
	\arrow["{{h_{\otimes}}}", from=2-2, to=1-2]
	\arrow["{{\epsilon^*}}", from=3-2, to=2-2]
	\arrow["{{x_{\otimes}}}"', from=3-5, to=1-5]
	\arrow["{{m^*}}", from=3-5, to=3-2]
\end{tikzcd}
\end{equation}
Given left adjoints $\psi_! \dashv \psi^*$ and $m_! \dashv m^*$, we acquire a Beck-Chevalley transformation
\[
    \beta_{m,x}\colon \psi_!h_{\otimes}\epsilon^* \longrightarrow x_{\otimes} m_!.
\]
\end{construction}
We may now give Elmanto-Haugseng's notion of distributivity. 
\begin{definition}\label{def:dist-EH}
    Given an $\NN$-compatible collection of $\TT$-anima $\udl{\mathcal{H}} \subset \GSpaces$,
    an $\NN$-normed $\TT$-category $\CC$ is \mdef{$\udl{\mathcal{H}}$-distributive in the sense of Elmanto-Haugseng} if 
    \begin{enumerate}[label =(\alph*)]
        \item \label{cond:fibrewise-sifted} $\CC$ fibrewise admits sifted colimits and its restriction and $\NN$-tensor functors preserve them;
        \item \label{cond:M-cocomplete} $\CC$ admits $\udl{\mathcal{H}}$-colimits; and
        \item \label{cond:BC} for every distributivity diagram $\Delta_{m,x}$ such that $m$ has fibres lying in $\udl{\mathcal{H}}$, the Beck-Chevalley transformation $\beta_{m,x}$ is an equivalence.
    \end{enumerate}
    If $\udl{\mathcal{H}} = \GSpaces[\TT]$, we simply say that $\CC$ is \mdef{distributive in the sense of Elmanto-Haugseng}.
\end{definition}

\begin{remark}
    Condition (b) can be unwound to \textit{existence} of left adjoints to pullback along the maps with fibres in $\mathcal{H}$ and the Beck-Chevalley conditions for restrictions thereof;
    hence it is equivalent to conditions (2) and (3) of \llp{Def.}{3.12}.
    In particular, this agrees with the stipulation that $\CC$ is $\udl{\mathcal{H}}$-distributive in the sense of \cite{llp25}.
    Moreover, in the case $\NN = \TT$, i.e. $\TT$ is an atomic orbital category, 
    it follows by unwinding definitions that a $\TT$-normed $\baseCat$-category is $\TT$-distributive in the sense of \cite{llp25} if and only if it is distributive in the sense of Elmanto-Haugseng.
\end{remark}

Our main result about this notion is the following, which compares it to ours.
To state it,  we refer to images of sifted categories under the constant-value functor $\infl{}{\baseCat}\colon \Cat \rightarrow \GCat[\baseCat]$ as \mdef{inflated-sifted categories}, and write $\infl{}{\TT} \Cat_{\mathrm{sifted}} \subset \GCat$ for the collection of inflated-sifted categories
\begin{proposition}\label{prop:two-distributivities}
    Suppose $\CC$ is an $\NN$-normed $\TT$-category and $\udl{\mathcal{H}} \subset \GSpaces[\TT]$ an $\NN$-compatible collection of $\baseCat$-anima.
    Then, $\CC$ is $\udl{\mathcal{H}} \cup \infl{}{\TT}\Cat_{\mathrm{sifted}}$-distributive if and only if it is $\udl{\mathcal{H}}$-distributive in the sense of Elmanto-Haugseng;
    moreover, $\CC$ is distributive if and only if it is distributive in the sense of Elmanto-Haugseng.
\end{proposition}

\subsection{Sifted distributivity}
\label{subsec:sift}
In certain cases, we will find distributivity to be equivalent to preservation of certain parametrised colimits.
This requires a notion of parametrised siftedness:
\begin{definition}
    Given $w\colon W \rightarrow V$ a map of orbit-finite $\TT$-\animae{} and $\udl{K}$ a $V$-category, we say that $\udl{K}$ is \mdef{$w$-sifted} if the diagonal $V$-functor
    $\Delta^w\colon \udl{K} \longrightarrow w_* w^* \udl{K}$
    is cofinal.
    More generally, if $\NN \subset \TT$ is an orbital subcategory, we say that $\udl{K}$ is \mdef{$\NN$-sifted} if $U^* \udl{K}$ is $w$-sifted for all maps $U \rightarrow V$ and $\NN$-admissible maps $w\colon W \rightarrow U$.
\end{definition}
\begin{example} 
If $\NN = \TT$, these are the \textit{$\TT$-sifted $\TT$-categories} of \Shatwo{Def.}{8.14}.
\end{example}
The proof of \Shatwo{Thm.}{8.15}, for only $\NN$-admissible diagonals, shows the following.
\begin{proposition}\label{prop:sifted-cofinal}\label{ex:fibrewise-sifted}
    $\udl{K} \in \GCat[V]$ is $\NN$-sifted if and only if it is fibrewise sifted and, for all $\NN$-maps $w\colon W \rightarrow U$ and maps $U \rightarrow V$, the restriction functor $K^U \rightarrow K^W$ is cofinal.
\end{proposition}
\begin{example}\label{ex:infl_sift}
    Inflated sifted categories are $\NN$-sifted for all $\NN$ by \Cref{prop:sifted-cofinal}.
\end{example}
These satisfy the following pleasant property:
\begin{proposition}\label{prop:dist-lemmas}
    Let $w \colon W \to V$ be an $\NN$-admissible map of orbit-finite $\TT$-\animae{}, $\CC$ a $W$-category, $\DD$ a $V$-category, and $F\colon w_\ast \CC \rightarrow \DD$ a $w$-multifunctor.
  \begin{enumerate}[label=(\arabic*)]
    \item the $w$-multifunctor $F$ distributes inflated-sifted colimits if and only if the $V$-functor $F$ fibrewise preserves sifted colimits.
    \item \label{st:sifted-distribute} Suppose $\overline{m}\colon \udl{K}^{\triangleright} \rightarrow w_* \CC$ is a $V$-colimit cocone such that $\udl{K}$ is $w$-sifted.
      Then, the $V$-functor  $F$ preserves the colimit $\overline{m}$ if and only if $F$ distributes the mate $W$-functor $\widetilde p\colon w^* \udl{K}^{\triangleright} \rightarrow \CC$.
    \item \label{st:sifted-generate} Suppose $\KK\subset\GCat$ contains all inflated-sifted categories, $\udl{\mathcal{H}} \subset \KK$ is a collection of $\TT$-categories which generate $\KK$ under inflated-sifted colimits, and $F$ is $\udl{\mathcal{H}}$-distributive and preserves inflated-sifted colimits.
    Then, $F$ is $\KK$-distributive.
  \end{enumerate}
\end{proposition}
\begin{proof} 
  For (2), we closely mirror \Shatwo{Prop.}{8.19}:
  indeed, given a $V$-colimit diagram $\overline m\colon \udl{K}^{\triangleright} \rightarrow w_* \CC$,  its mate $\widetilde{p} \colon \prn{w^* \udl{K}}^{\triangleright} \rightarrow \CC$ is a $U$-colimit diagram.
  Consider the diagram
\[\begin{tikzcd}[ampersand replacement=\&, row sep=small, column sep=large]
	{\udl{K}^{\triangleright}} \& {\prn{w_* w^* \udl{K}}^{\triangleright}} \& {w_* \prn{w^* \udl{K}}^{\triangleright}} \& {w_* \CC} \& \DD \\
	{\udl{K}^{\triangleright}} \&\&\& {w_* \CC} \& \DD
	\arrow["\prn{\Delta^w}^{\triangleright}", from=1-1, to=1-2]
	\arrow[equals, from=1-1, to=2-1]
	\arrow["\can", from=1-2, to=1-3]
	\arrow["{w_* \widetilde{p}}", from=1-3, to=1-4]
	\arrow["{F}", from=1-4, to=1-5]
	\arrow[equals, from=1-4, to=2-4]
	\arrow[equals, from=1-5, to=2-5]
	\arrow["{\overline{m}}", from=2-1, to=2-4]
	\arrow["{F}", from=2-4, to=2-5]
\end{tikzcd}\]
    That the lefthand rectangle commutes is a routine application of the triangle identity. Now, $w$-siftedness of $\udl{K}$ is defined so that the diagonal functor $\Delta^w$ is cofinal, so the composite is a $V$-colimit (i.e. $F$ preserves the colimit $\overline{m}$) if and only if the composite $\prn{w_* w^* \udl{K}}^{\triangleright} \rightarrow \DD$ is a $V$-colimit;
    this is clearly equivalent to $F$ distributing $\widetilde p$.
  
    Statement (1) is a straightforward application of this, as non-parametrised sifted diagrams are $\NN$-sifted by \Cref{ex:infl_sift}.

    For statement (3), we equivalently show that the collection of categories for which $F$ is distributive is closed under inflated-sifted colimits. 
    Let $\udl{K} \to \CC$ be a diagram and assume that $\udl{K} \simeq \colim_{i \in \JJ} \udl{K}_i$ is a colimit of \animae{} $\udl{K}_i\in \udl{\mathcal{H}}$ over which $F$ is distributive and $\JJ$ is inflated-sifted. 

    For a diagram $\partial\colon \udl{K}\rightarrow \CC$, we know by the colimit decomposition of \cite[Cor. A.7]{kaifSil} that $\colim_{\udl{K}}\partial\simeq \colim_{i\in \JJ}\colim_{\udl{K}_i}\partial_i\in \CC$. 
    On the other hand, since the cartesian $\NN$-normed structure on $\udl{\cat}$ is distributive by \Cref{prop:distributivity_of_cartesian_symmon}, we have 
    \[
        w_*\udl{K}\simeq w_*\colim_{i\in \JJ}\udl{K}_i\simeq \colim_{i\in w_*\JJ}w_*\udl{K}_i\xleftarrow{\;\;\;\; \sim \;\;\;\;} \colim_{i\in \JJ}w_*\udl{K}_i,
    \] 
    where the left-pointing equivalence is induced by the diagonal map $\JJ\rightarrow w_*w^*\JJ$, which is an equivalence by $w$-siftedness of $\JJ$ under \Cref{ex:infl_sift}. 

    Viewing $\colim_{\udl{K}}\partial\in \udl{\sC}^W\simeq (w_*\udl{\sC})^V$ as a global object of $w_*\CC$, we must show that the assembly map $\colim_{w_*\udl{K}}Fw_*\partial \rightarrow F\colim_{\udl{K}}\partial$ in $\udl{\D}$ is an equivalence;
    but this factors as
    \[\colim_{w_*\udl{K}}Fw_*\partial\simeq \colim_{i\in \JJ}\colim_{w_*\udl{K}_i}Fw_*\partial_i \rightarrow \colim_{i\in\JJ}F\colim_{\udl{K}_i}\partial_i\rightarrow F\colim_{i\in\JJ}\colim_{\udl{K}_i}\partial_i\simeq F\colim_{\udl{K}}\partial\] where the second map is an equivalence by assumption that $F$ is $\udl{\mathcal{H}}$-distributive, and the third map is an equivalence by assumption that $F$ preserves inflated-sifted colimits, completing the proof.
\end{proof}

\subsection{The comparison with Elmanto-Haugseng distributivity}
We now prove \Cref{prop:two-distributivities}, starting by simplifying condition (c).
\begin{lemma}\label{lem:EH-distributivity-orbital}
    Suppose $\CC$ satisfies conditions (a) and (b) and $\beta_{m,x}$ is an equivalence whenever $Y$ is an orbit.
    Then, $\CC$ is $\udl{\mathcal{H}}$-distributive in the sense of Elmanto-Haugseng.
\end{lemma}
\begin{proof}
Given a map from an orbit $i\colon U \rightarrow Y$ and a diagram $A \xrightarrow{m} X \xrightarrow{x} Y$ with $m$  having fibres in $\mathcal{H}$ and $w$ being $\NN$-admissible, we acquire a diagram 
    \[
    \begin{tikzcd}[ampersand replacement=\&]
	\&\&\& {R_U} \&\& {Q_U} \\
	\& {A_U} \& R \&\& Q \\
	A \&\&\& {X_U} \&\& U \\
	\&\& X \&\& Y
	\arrow[from=1-4, to=2-2]
	\arrow[from=1-4, to=2-3]
	\arrow[from=1-4, to=3-4]
	\arrow[from=1-4, to=1-6]
	\arrow[from=1-6, to=2-5]
	\arrow["{\psi_U=x_{U *} m_U}", from=1-6, to=3-6]
	\arrow[from=2-2, to=3-1]
	\arrow["{m_U}"'{pos=0.9}, from=2-2, to=3-4]
	\arrow["m"', from=3-1, to=4-3]
	\arrow[from=3-4, to=4-3]
	\arrow["{x_U}"{pos=0.8}, from=3-4, to=3-6]
	\arrow[from=3-6, to=4-5]
	\arrow["x"', from=4-3, to=4-5]
    \arrow[from=2-3, to=2-5, crossing over]
	\arrow["{\psi=x_* m}"{pos=0.2}, from=2-5, to=4-5, crossing over]
    \arrow[from=2-3, to=3-1, crossing over]
	\arrow[from=2-3, to=4-3, crossing over]
    \end{tikzcd}
    \]
    where every square is a pullback diagram and the front and back house-shaped faces are distributivity diagrams;
    so in particular, we acquire a natural transformation $\Delta_{m_U,x_U} \rightarrow \Delta_{m,x}$.
    This yields a natural transformation $\Psi_{m,x} \rightarrow \Psi_{m_U,x_U}$, i.e.
    \[
        \begin{tikzcd}[ampersand replacement=\&, row sep=small]
	\& {\CC^{Q_U}} \&\&\& {\CC^{U}} \\
	{\CC^{Q}} \&\&\& {\CC^{Y}} \\
	\& {\CC^{R_U}} \\
	{\CC^{R}} \\
	\& {\CC^{A_U}} \&\&\& {\CC^{X_U}} \\
	{\CC^{A}} \&\&\& {\CC^{X}}.
	\arrow[from=1-5, to=1-2, "\psi_U^*"']
	\arrow[from=2-1, to=1-2]
	\arrow[from=2-4, to=1-5, "i^*"]
	\arrow[from=3-2, to=1-2, "h_{U,\otimes}"{pos=0.2}]
	\arrow["{{{h_{\otimes}}}}", from=4-1, to=2-1]
	\arrow[from=4-1, to=3-2]
	\arrow[from=5-2, to=3-2, "\epsilon_U^*"]
	\arrow[from=5-5, to=1-5, "x_{U,\otimes}"']
	\arrow[from=5-5, to=5-2, "m_U^*"'{pos=0.7}]
	\arrow["{{{\epsilon^*}}}", from=6-1, to=4-1]
	\arrow[from=6-1, to=5-2, "j^*"]
	\arrow[from=6-4, to=5-5]
	\arrow["{{{m^*}}}", from=6-4, to=6-1]
    \arrow["{{{x_{\otimes}}}}"', from=6-4, to=2-4, crossing over]
    \arrow["{{{\psi^*}}}"'{pos=0.3}, from=2-4, to=2-1, crossing over]
\end{tikzcd}
    \]
    We claim that the two Beck-Chevalley conditions are directly related:
    \begin{equation}\label{eq:horrible-BC-diagram}
        \begin{tikzcd}[ampersand replacement=\&, row sep=tiny, column sep=small]
	{  i^* \psi_! h_{\otimes} \epsilon^*} \& {i^*\psi_! h_{\otimes} \epsilon^* m^* m_!  } \& { i^*\psi_! \psi^* x_{\otimes} m_!} \& {i^* x_{\otimes} m_!} \\
	{ \psi_{U,!} h_{U,\otimes} \epsilon_U^* j^*} \& {\psi_{U,!} h_{U,\otimes} \epsilon_U^* m_U^* m_{U,!}j^*} \& { \psi_{U,!} \psi_U^* x_{U,\otimes} m_{U,!} j^*} \& {x_{U,\otimes} m_{U,!}j^* }
	\arrow[from=1-1, to=1-2]
	\arrow["\simeq"{marking, allow upside down}, draw=none, from=1-1, to=2-1]
	\arrow["\simeq"{marking}, draw=none, from=1-2, to=1-3]
	\arrow["\simeq"{marking, allow upside down}, draw=none, from=1-2, to=2-2]
	\arrow[from=1-3, to=1-4]
	\arrow["\simeq"{marking, allow upside down}, draw=none, from=1-3, to=2-3]
	\arrow["\simeq"{marking, allow upside down}, draw=none, from=1-4, to=2-4]
	\arrow[from=2-1, to=2-2]
	\arrow["\simeq"{marking}, draw=none, from=2-2, to=2-3]
	\arrow[from=2-3, to=2-4]
\end{tikzcd}
    \end{equation}
    Here, the top composite is $i^* \beta_{m,w}$ and the bottom is $\beta_{m_U,w_U} j^*$;
    commutativity follows by commutativity of \eqref{eq:horrible-BC-diagram} together with the fact that $\CC$ admits $\udl{\mathcal{H}}$-indexed colimits.
    
    By assumption, each $\beta_{m_U,w_U}j^*$ is an equivalence, implying that $i^* \beta_{m,w}$ is an equivalence for each $U \rightarrow X$.
    The Segal condition that the $i$ assemble to an equivalence 
    \[
        \CC^X \xrightarrow{\;\; \sim\;\;} \lim_{U \rightarrow X}\CC^U
    \]
    then implies that $\beta_{m,x}$ is an equivalence, completing the result.
\end{proof}

We may now conclude our comparison.
\begin{proof}[Proof of \Cref{prop:two-distributivities}]
    First suppose $\CC$ is $\udl{\mathcal{H}} \cup \infl{e}{\TT}\Cat_{\mathrm{sifted}}$-distributive.
    By assumption, $\CC$ satisfies \Crefenum{def:dist-EH}{cond:M-cocomplete} and the cocompleteness part of \Crefenum{def:dist-EH}{cond:fibrewise-sifted}, and \Crefenum{prop:dist-lemmas}{st:sifted-distribute} shows that it satisfies the full \Crefenum{def:dist-EH}{cond:fibrewise-sifted}.
    Moreover, \Cref{rmk:assembly-map} shows that $\beta_{m,x}$ is an equivalence whenever $X$ is an orbit, so $\CC$ is $\udl{\mathcal{H}}$-distributive in the sense of Elmanto-Haugseng by \Cref{lem:EH-distributivity-orbital}.
    
    Now, suppose $\CC$ is $\udl{\mathcal{H}}$-distributive in the sense of Elmanto-Haugseng.
    By assumption, $\CC$ admits $\udl{\mathcal{H}} \cup \infl{e}{\TT}\Cat_{\mathrm{sifted}}$-colimits, and by \Crefenum{prop:dist-lemmas}{st:sifted-distribute}, the indexed tensor multifunctors of $\CC$ distribute $\infl{e}{\TT}\Cat_{\mathrm{sifted}}$-colimits.
    Moreover, \Cref{rmk:assembly-map} shows that the indexed tensor multifunctors of $\CC$ distributive $\udl{\mathcal{H}}$-colimits.
    The result then follows by \Crefenum{prop:dist-lemmas}{st:sifted-generate}.

    For the remaining statement, set $\udl{\mathcal{H}} = \GSpaces$.    
    In light of \Crefenum{prop:dist-lemmas}{st:sifted-distribute}, it suffices to prove that $\GSpaces$ generates $\GCat$ under inflated-sifted colimits.
    But this follows by \Shatwo{Thm.}{8.6} by noting that their \textit{corepresentable $\TT$-left fibrations} are $\TT$-anima, as are finite coproducts of them.
\end{proof}
\resumetocwriting
\section{Parametrised size issues}\label{sec:size}
The purpose of this subsection is to work out some technical details about size issues in the parametrised setting in order to connect to and leverage the work of Martini--Wolf.
First, a familiar fact from topos theory:
\begin{lemma}\label{lem:small-vs-cpct}
    A $\baseCat$-anima is $\kappa$-compact if and only if it is a retract of a $\kappa$-small object;
    if $\kappa$ is uncountable, then this is equivalent to it being $\kappa$-small.
\end{lemma}
\begin{proof}
    The first claim is \cite[Prop.~5.3.4.17]{lur09}, and the second follows from the fact that idempotent-splittings are $\aleph_1$-small colimits. 
\end{proof}
\begin{proposition}\label{prop:kappa-restricts}
     If $\kappa$ is a regular cardinal, the following are equivalent.
    \begin{enumerate}[label={(\alph*)}]
        \item Given a cospan of orbits $U \rightarrow V \leftarrow X$, the pullback $U \times_V X$ is $\kappa$-small.
        \item The inclusions $\Spaces[V,\sm(\kappa)] \hookrightarrow \Spaces[V]$
        glue together to a $\baseCat$-subcategory.
        \item The inclusions $\Cat[V,\sm(\kappa)] \hookrightarrow \Cat[V]$ glue together to a $\baseCat$-subcategory.
        \item The subcategory $\Spaces[\baseCat,\sm(\kappa)] \subset \Spaces[\baseCat]$ is closed under pullbacks.
    \end{enumerate}
    In this case, $\Spaces[V,\kappa] \hookrightarrow \Spaces[V]$ and $\Cat[V,\kappa] \hookrightarrow \Cat[V]$ form $\baseCat$-subcategories.
\end{proposition}
\begin{proof}
    Obviously (c) $\implies$ (b) $\implies$ (a) $\impliedby$ (d).
    To see (a) $\implies $(c), fix some $u\colon U \rightarrow V$ and note that the subcategory $\C \coloneqq \Cat[V] \times_{\Cat[U]} \Cat[U, \udl{\sm}(\kappa)]$ is closed under $\kappa$-small nonparametrised colimits;
    (a) and the observation that $U \rightarrow V$ has $\kappa$-small domain if and only if it is small in $\Spaces[V]$ by colimit-creation of the unslicing functor $\Spaces[V] \rightarrow \Spaces[\baseCat]$ shows that $\C$ contains orbits.
    Hence $\C$ contains all $\kappa$-small categories, i.e. (c).

    Now, assume (b).
    We will prove (d), so fix a cospan $A \rightarrow Y \leftarrow Z$ of $\kappa$-small $\baseCat$-\animae{}.
    Now, parametrised starightening-unstraightening over the projection $A \times_Y Z \rightarrow Z$ exhibits this pullback as a $\kappa$-small colimit of pullbacks of the form $U^* A \times_{V^*Y} \udl{*}$;
    $V^* A$ and $V^* Y$ are $\kappa$-small by assumption, so we are reduced to showing $A \times_Y *$ is $\kappa$-small in the case that $\baseCat$ has a terminal object $V$.
    Applying this again reduces us to showing that $* \times_Y *$ is $\kappa$-small;
    this is a mapping anima in the category of $V$-objects of $Y$, which is $\kappa$-small by assumption that $Y$ itself (hence $Y^V$) is $\kappa$-small, so we are done.

    The remaining claim follows from \Cref{lem:small-vs-cpct}, since functors preserve retracts.
\end{proof}
We now note that the notions of \mdef{$\kappa$-smallness, $\kappa$-filteredness}, and \mdef{$\kappa$-compactness} agree with those of \cite[Def.~2.2.2.1]{mwpresentable}, \cite[Def.~2.1.1.1]{mwpresentable}, and \cite[Def.~2.1.5.1]{mwpresentable},  with internal class $U$ being that determined by the $\baseCat$-category $\GCat[\sm(\kappa)]$.
Now, using the exact same arguments as \Cref{prop:kappa-restricts}, we acquire the following.
\begin{proposition}\label{prop:kappa-norms}
     If $\kappa$ is a regular cardinal, the following are equivalent.
    \begin{enumerate}[label={(\alph*)}]
        \item Given a diagram of orbit-finite $\baseCat$-\animae{} $U \xrightarrow{u} V \xleftarrow{w} W \leftarrow X$, the $U$-anima $u^* w_* X$ is $\kappa$-small.
        \item The inclusions $\Spaces[V,\sm(\kappa)] \hookrightarrow \Spaces[V]$
        form an $\NN$-normed $\baseCat$-subcategory.
        \item The inclusions $\Cat[V,\sm(\kappa)] \hookrightarrow \Cat[V]$ form an $\NN$-normed $\baseCat$-subcategory.
    \end{enumerate}
    In this case, $\Spaces[V,\kappa] \hookrightarrow \Spaces[V]$ and $\Cat[V,\kappa] \hookrightarrow \Cat[V]$ form $\NN$-normed $\baseCat$-subcategories.
\end{proposition}
Now, we ask the reader to recall \Cref{def:Card,def:filt,def:small,def:kappa-normable}.
In particular, $\baseCat$-restrictability and $\NN$-normability are identified with the equivalent conditions of \Cref{prop:kappa-restricts,prop:kappa-norms}.
These sometimes appear in the literature:
\begin{corollary}\label{cor:everything-is-everything}
    A regular cardinal $\kappa$ is $\baseCat$-restrictable if and only if it is $(\baseCat,\baseCat)$-regular in the sense of \cite[Def.~5.20]{brink2026globalinftycategoriesglobalthom}.
    Moreover, $\kappa$ is uncountable and $\baseCat^{\simeq}$-normable if and only if it is $\Spaces[\baseCat]$-regular in the sense of \cite[Def.~2.2.2.3]{mwpresentable}, in which case $\GCat[\myuline{\mathrm{rex}}(\kappa)] \subset \GCat$ determines Martini-Wolf's sound doctrine $\Cat[\kappa]_{\Spaces[\baseCat]}$. 
\end{corollary}
\begin{proof}
     Definition 2.2.2.3 in \cite{mwpresentable} obviously implies ours, and ours obviously implies their condition (1).
     In our case, their condition (2) is satisfied for all $\kappa$, since presheaf categories are compactly generated.
     
     Their condition (3) obviously implies $\baseCat^{\simeq}$-normability; compactness is closed under summand inclusion, and applying $W = \emptyset$ shows that $\baseCat^{\simeq}$-normability of $\kappa$-implies that $*$ is $\kappa$-small, so pullback-closure of $\kappa$-smallness implies finite limit closure, completing the identification with Martini-Wolf soundness.
     The identification with the sound doctrine follows by \Cref{lem:small-vs-cpct}.

     Similarly, the comparison with Brink-Lenz follows immediately from the pullback characterisation in \Cref{prop:kappa-restricts}.
\end{proof}
Now, $\NN$-normable cardinals (hence $\baseCat$-restrictable cardinals) are abundant:
\begin{proposition}\label{prop:cofinal_embedding}\label{rmk:treg-is-corep}\label{rmk:cofin}
    Let $\upsilon$ be the large regular cardinal of small sets.
    The embedding $\Card[\NN] \subset \Card$ is cofinal and $\upsilon$-small supremum-closed, and each are $\upsilon$-filtered.

    Moreover, if $\mathrm{Sound}_{\Spaces[\baseCat]}$ denotes Martini-Wolf's poset of sound doctrines for the topos $\Spaces[\baseCat]$, then \Cref{cor:everything-is-everything} produces a cofinal full embedding $\Card[\NN]_{\geq \aleph_1} \subset \mathrm{Sound}_{\Spaces[\baseCat]}$.
\end{proposition}
\begin{proof}
    First note that $\Card$ is $\upsilon$-filtered, since it has small suprema by regularity of $\upsilon$.

    We first prove that there exists an $\NN$-normable cardinal.
    Fix a diagram of orbit-finite $\baseCat$-\animae{} $U \xrightarrow{u} V \xleftarrow{w} W \xleftarrow{x} X$ with $w$ $\NN$-admissible.
    Since the yoneda-image generates $\GSpaces[U]$,  $u^* w_* X$ can be expressed as the colimit of a small sequence of small colimit diagrams;
    taking the suprema of the cardinalities of the indexing diagrams yields a cardinality $\overline{\kappa}_{u,v,x}$, and cofinality of regular cardinals in $\Card$ allows us to choose regular $\kappa_{u,w,x}$ above it.
    Choosing this uniformly across the small set of triples $(u,v,w)$, we take a supremum
    \[
        \kappa \coloneqq \sup_{u,w,x} \kappa_{u,w,x}.
    \]
    By definition, this is an $\NN$-normable regular cardinal, so $\Card[\NN]$ is indeed nonempty.

    Note that if we have a relation of regular cardinals $\lambda \geq \kappa$ such that $\kappa$ is $\NN$-normable, $\lambda$ is  $\NN$-normable as well by condition (a) of \Cref{prop:kappa-norms}.
    Now, since $\Card$ is well-ordered, $\Card[\NN]$ has a minimum element $\kappa_{\NN}$;
    thus we have
    \[
        \Card[\NN] = \mathrm{Card}^{\mathrm{reg}}_{\geq \kappa_{\NN}} \subset \Card.
    \]
    Consequently, cofinality and supremum-closure of $\Card[\NN]$ are reduced to $\mathrm{Card}^{\mathrm{reg}}$, where they are well-known.
    $\upsilon$-filteredness of $\Card[\NN]$ then follows from that of $\Card$.

    What remains is cofinality of $\Card[\NN] \subset \mathrm{Sound}_{\Spaces[\baseCat]}$;
    applying cancellation for cofinality \kerodon{02N6} to the composition $\Card[\NN] \subset \Card[\baseCat^{\simeq}] \subset \mathrm{Sound}_{\Spaces[\baseCat]}$ shows that its first inclusion is cofinal, so it suffices to prove this claim when $\NN \simeq \baseCat^{\simeq}$.
    Indeed, under the identification with $\Spaces[\baseCat]$-regularity of \Cref{cor:everything-is-everything}, this is simply \cite[Rmk.~2.2.2.22]{mwpresentable}.
\end{proof}
We conclude with the usual ``$\kappa$-small and $\kappa$-filtered'' generation of cocompleteness.
\begin{proposition}\label{lem:filt+small}
    Fix $\kappa$ a $\baseCat$-restrictable cardinal.
    There is an equality of $\baseCat$-subcategories
    \begin{align*}
    \wGCat\prn{\udl{\allcolim}} =
    \wGCat\prn{\infl{}{\baseCat} (\filt(\kappa)) \cup \udl{\sm}(\kappa)} \subset \wGCat.
    \end{align*}
\end{proposition}
\begin{proof}
    By the usual argument, we immediately pass to the case that $\baseCat \simeq \baseCat_{/V}$.
    
    Let $\KK \subset \GCat$ be the maximum full $\baseCat$-subcategory such that 
    \[
    \wGCat\prn{\infl{}{\baseCat} (\filt(\kappa)) \cup \udl{\sm}(\kappa)} = \wGCat\prn{\KK};
    \]
    that is, $\KK$ consists of the diagrams whose cocompleteness (resp. cocontinuity) is guaranteed by $\kappa$-small and fibrewise $\kappa$-filtered cocompleteness (cocontinuity). 
    We are tasked with proving that $\KK = \GCat$.

    By \cite[Prop.~5.4.1]{MartiniWolf2024} we have $\GCatcoco{\allcolim} = \GCatcoco{\infl{}{\baseCat}\Cat \sqcup \GSpaces}$, so we must show 
    \[
        \infl{}{\baseCat}\Cat, \GSpaces \subset \KK.
    \]
    For the former, note that $\infl{}{\baseCat}$ of $\kappa$-small categories are $\kappa$-small since it is cocontinuous and semds $ * \mapsto \udl{*} \simeq V$, which is $\kappa$-small;
    applying \cite[Prop.~5.4.5]{MartiniWolf2024} then reduces the inclusion to the non-parametrised case, which is well-known.
    
    Now, it follows by \cite[Cor.~A.7]{kaifSil} that $\KK$ is closed under $\KK$-indexed colimits, so to conclude, it suffices to show that  $\GSpaces[V] \simeq \presheaf\prn{\baseCat}$ is generated under non-parametrised colimits by $\kappa$-small $V$-categories;
    indeed, it is so generated by orbits.
\end{proof}

\section{Equivariant norms and finiteness}\label{section:model_cats}\label{appendix:point-set_norms}
The purpose of this appendix is to prove the following result.
\begin{proposition}\label{cor:norms-finite-cats}
    Let $G$ be a compact Lie group. The fully faithful inclusions 
    
    \begin{center}
        \begin{tabular}{c c c c}
         $\GCat[\mathrm{fin}]\subset \GCat,$ & $  \GSpaces[\mathrm{fin}]\subset \GSpaces, $ & $ \GBSpaces[\mathrm{fin}]\subset \GBSpaces, $   \\
         $\GCat[\aleph_0]\subset \GCat, $ & $  \GSpaces[\aleph_0]\subset \GSpaces, $ & $  \GBSpaces[\aleph_0]\subset \GBSpaces, $ 
    \end{tabular}
    \end{center}
    of finite (resp. compact) objects refine to inclusions of $\finind$-normed categories.
\end{proposition}
We will need two crucial results. Firstly, we need the equivariant finiteness of cubes, and this is supplied by \cite[Prop. C.0.5]{hilmanRamzi}. Secondly, we need finiteness of coinduced homogeneous anim\ae{}, i.e. the following.
\begin{proposition}\label{prop:fin-norms-homogeneous}
    The $G$-categories of the form $w_* H/K$ are finite, given $w\colon G/H \rightarrow G/G$ a $\finind$-map.
\end{proposition}

It is at this point that we must use geometry.
For $G$ a compact Lie group, we let $\topcat^G$ denote the 1-category of compactly generated, weak Hausdorff topological \animae{} with an action of $G$. 
This 1-category can be equipped with a model structure in which the fibrations and weak equivalences are detected on fixed points -- that is, 
the Dwyer-Kan localisation at the maps $f\colon X \rightarrow Y$ such that, for all closed subgroups $H \leq G$, the map $X^H \rightarrow Y^H$ is a weak equivalence is the $\infty$-category of $G$-\animae{}: 
\[ \gamma^G \colon \topcat^G \longrightarrow \Spaces[G] \,; \]
see \cite[Thm.~6.3]{piacenza} or \cite[App. B]{schwede2018global} or  \cite[Thm.~3.17]{stephan}, noting that Quillen equivalences of simplicial model categories with all objects fibrantly localise to equivalences of $\infty$-categories via the dual to \cite[Cor.~A.3.1.10]{lur09}, and this simplicial enrichment may be obtained by base-change of Schwede's topological model structure along the monoidal right Quillen equivalence $\mathrm{Sing}\colon  \topcat \rightarrow \mathrm{SSet}$ via \cite[Rmk.~A.3.2.6]{lur09}.

We will henceforth use boldface font for topologically enriched 1-categories.
\begin{lemma}\label{lem:res-model}  
  For every closed subgroup $H\leq G$, the following diagram commutes:
\[\begin{tikzcd}[ampersand replacement=\&,column sep=scriptsize]
	{\topcat^G} \& {\Fun_{\topcat}\prn{\mathbf{B}G, \mathbf{Top}}} \&\& {\Fun_{\topcat}\prn{\mathbf{O}_G\op, \mathbf{Top}}} \&\& {\Spaces[G]} \\
	{\topcat^H} \& {\Fun_{\topcat}\prn{\mathbf{B}H, \mathbf{Top}}} \&\& {\Fun_{\topcat}\prn{\mathbf{O}_H\op, \mathbf{Top}}} \&\& {\Spaces[H]}
	\arrow["\simeq"{description}, draw=none, from=1-1, to=1-2]
	\arrow["{{\res{G}{H}}}"', from=1-1, to=2-1]
	\arrow["{{i_1^*}}"', from=1-2, to=2-2]
	\arrow["{{\iota^*_G}}"', from=1-4, to=1-2]
	\arrow["L", from=1-4, to=1-6]
	\arrow["{{i_0^*}}"', from=1-4, to=2-4]
	\arrow["{{\res{G}{H}}}"', from=1-6, to=2-6]
	\arrow["\simeq"{marking, allow upside down}, draw=none, from=2-1, to=2-2]
	\arrow["{{\iota_H^*}}"', from=2-4, to=2-2]
	\arrow["L", from=2-4, to=2-6]
\end{tikzcd}\]
    where $\iota_H\colon \mathbf{B}H \hookrightarrow \mathbf{O}_H$ is the inclusion,
    $i_0 = \mathrm{ind}_{H}^{G}\colon \mathbf{O}_H \rightarrow \mathbf{O}_G$ induction, and $i_1$ is the restriction of induction to to $\mathbf{B}H \subset \mathbf{O}_H$, i.e. the classifying topological functor of the inclusion $H \subset G$. 
    Moreover, the injective model structure on $\Fun_{\topcat}\prn{\mathbf{O}_G\op, \mathbf{Top}}$ exists and $i_0^*$ is left-Quillen for it.
\end{lemma}
\begin{proof}
    The lefthand and righthand squares are essentially by definition, and the middle by functoriality of pullback, since we have a natural equivalence of topological functors $i_0 \circ \iota_H \simeq \iota_G \circ i_1$.
    That $i_0^*$ is left-Quillen for the injective model structure as long as it exists is obvious (see \cite[Prop.~A.2.8.7]{lur09}), so we are left with existence.
    But such a structure exists so long as cofibrations in $\mathbf{Top}$ are closed under $\times Y$ for all $Y \in \mathbf{Top}$ by \cite[Rmk.~4.5]{Moser19}.
    This is true for a general monoidal model category, such as compactly generated weak Hausdorff topological spaces (see \cite[after Cor.~2.2]{hovey1998monoidalmodelcategories}).
\end{proof}
\begin{remark}
    Elmendorf's theorem \cite[Thm.~3.17]{stephan} precisely states that $\iota_G^*$ is a right Quillen equivalence for the projective model structure on its domain, and hence connects $\iota_G^*$ underlies a zigzag of Quillen equivalences for the injective model structure as well.
\end{remark}
The projective and injective model structures both present the category of topological functors $\mathbf{O}_G\op \rightarrow \mathbf{Top}$, and hence the presheaf category on $\Orbit{G}$;
in particular, the above Quillen adjunction presents the adjunction $\res{G}{H} \dashv \coind{H}{G}$.

\begin{lemma}\label{lem:coind-model}
    Let $\bMap_H(G,-)\colon \topcat^H \rightarrow \topcat^G$ be the functor giving continuous $H$-equivariant maps their compact-open topology with precomposition $G$-action.
    Then, $\bMap_H(G,-)$ is homotopical and there exists a natural equivalence \[ L_{weq} \bMap_H(G,X) \simeq \coind{H}{G} X \,.\]
\end{lemma}
\begin{proof}
    We may take right adjoints to \Cref{lem:res-model} and acquire the diagram 
    \[\begin{tikzcd}[ampersand  replacement=\&,column sep=scriptsize]
	{\topcat^G} \& {\Fun_{\topcat}\prn{\mathbf{B}G, \mathbf{Top}}} \&\& {\Fun_{\topcat}\prn{\mathbf{O}_G\op, \mathbf{Top}}} \&\& {\Spaces[G]} \\
	{\topcat^H} \& {\Fun_{\topcat}\prn{\mathbf{B}H, \mathbf{Top}}} \&\& {\Fun_{\topcat}\prn{\mathbf{O}_H\op, \mathbf{Top}}} \&\& {\Spaces[H]}
	\arrow["\simeq"{description}, draw=none, from=1-1, to=1-2]
	\arrow["{{\iota_{G*}}}", from=1-2, to=1-4]
	\arrow["{{i_{1*}}}"', from=2-2, to=1-2]
	\arrow["L", from=1-4, to=1-6]
	\arrow["{{i_{0*}}}"', from=2-4, to=1-4]
	\arrow["{{\bMap_H\prn{G,-}}}", from=2-1, to=1-1]
	\arrow["\simeq"{marking, allow upside down}, draw=none, from=2-1, to=2-2]
	\arrow["{{\iota_{H*}}}", from=2-2, to=2-4]
	\arrow["L", from=2-4, to=2-6]
	\arrow["{{\coind{G}{H}}}", from=2-6, to=1-6]
    \end{tikzcd}\]
    If $\bMap_H(G,-)$ is homotopical, we are done.
    We may apply the double coset formula:
    \begin{align*}
        \bMap_H(G,-)^K
        &\simeq \prn{\bres{G}{K} \bMap_H(G,-)}^K\\
        &\simeq \prn{\prod_{KgH \in K \backslash G / H} \bMap_{K^g \cap H}\prn{K, C_g^* \bres{H}{K \cap ^gH} -}}^K\\
        &\simeq \prn{\prod_{KgH \in K \backslash G / H} \bMap_{K^g \cap H}\prn{K, C_g^* \bres{H}{K \cap ^gH} -}}^K\\
        &\simeq \prod_{KgH \in K \backslash G / H} (-)^{K \cap^g H};
    \end{align*}
    this is homotopical since fixed points preserve weak equivalences and $\topcat$ is a monoidal model category;
    since $\iota_{G*}$ preserves and reflects weak equivalences, this implies that $\bMap_H(G,-)$ is homotopical.
\end{proof}
\begin{proof}[Proof of \Cref{prop:fin-norms-homogeneous}]
    Since $G$-\animae{} are a colocalising subcategory of $G$-categories, it is equivalent to prove that they are finite as $G$-\animae{};
    that is, we are tasked with proving that $w_* H/K$ is in the finite-colimit closure of $\cbr{G/J \mid J \leq G}$.
    But it is equivalent to $\bMap_H(G,H/K)$ by \Cref{lem:coind-model}, which attains a canonical structure as a closed $G$-manifold by \cite[\S~6]{tD75}, and hence admits a $G$-CW structure up to homeomorphism \cite[Cor.~7.2]{ilman83};
    this must be finite, since our manifold is compact, giving the desired finiteness. 
\end{proof}
We moreover have the following.
\begin{lemma}\label{lem:res-homogeneous}
    If $H,J \leq G$ are closed subgroups, then the $H$-category $\res{G}{H} G/J$ is finite.
\end{lemma}
\begin{proof}
    Since $G/J$ is a compact $G$-manifold, its restriction is a compact $H$-manifold;
    by the same argument of \Cref{prop:fin-norms-homogeneous}, this implies that it is finite.
\end{proof}
\begin{proof}[Proof of \Cref{cor:norms-finite-cats}]
    We first argue the case for $G$-categories. Finite (resp. compact) $G$-categories are the finite colimit-closure (resp. finite colimit- and retracts-closure) of $G/H$.
    Let $\CC \subset \GCat[G]$ be the collection of finite (resp. compact) $G$-categories which are taken to finite (resp. compact) $G$-categories under all restrictions and $\finind$-norms.
    
    This contains $\mathbf{O}_G$ by \Cref{prop:fin-norms-homogeneous,lem:res-homogeneous} and compatibility of distributive finite products with finiteness. Since retracts are clearly preserved by the norms, it suffices to prove that $\CC$ is closed under pushouts, i.e. $\Lambda_0^2$-indexed colimits.
    But $\Lambda_0^2$ is a poset with (trivial) $H$-action, so $w_* \Lambda_0^2$ is finite by \cite[Prop. C.0.5]{hilmanRamzi}. The $\finind$-normed $G$-category 
  $\prn{\GCat[G]}^\times$ is distributive under \Cref{prop:distributivity_of_cartesian_symmon}, so we indeed find for $p \in \Fun_S(\infl{e}{S}\Lambda_0^2, \CC)$ that  $w_{\otimes} \colim_{\Lambda_0^2} \simeq \colim_{w_* \Lambda_0^2} w_{\otimes} \circ p$
    is a finite colimit of finite objects, hence finite.

    Finally, to deduce the other cases, simply note that we have compact-preserving $\finind$-normed left adjoint functors  $\GCat \xlongrightarrow{|-|} \GSpaces \xlongrightarrow{(-)_+} \GBSpaces$
    all of which are essentially surjective on compact generators. Since the finite (resp. compact) objects in these respective categories are all also finite colimit-closures (resp. finite colimit- and retract-closures) of the compact generators, we obtain the desired conclusions for each of these $G$-categories.
\end{proof}

\DeclareFieldFormat[article,incollection,inproceedings]{title}{\mkbibemph{#1}}

{\emergencystretch=1.5em 
\renewcommand{\emph}[1]{\textit{#1}} 
\printbibliography[heading=bibintoc]}

@article {Blumberg-op,
    AUTHOR = {Blumberg, Andrew J. and Hill, Michael A.},
     TITLE = {Operadic multiplications in equivariant spectra, norms, and
              transfers},
   JOURNAL = {Adv. Math.},
  FJOURNAL = {Adv. Math.},
    VOLUME = {285},
      YEAR = {2015},
     PAGES = {658--708},
}

@article{Hill_chromatic,
  AUTHOR = {Hill, Michael A.},
     TITLE = {Equivariant chromatic localizations and commutativity},
   JOURNAL = {J. Homotopy Relat. Struct.},
  FJOURNAL = {Journal of Homotopy and Related Structures},
    VOLUME = {14},
      YEAR = {2019},
    NUMBER = {3},
     PAGES = {647--662},
}

@misc{stefanich,
      title={Higher sheaf theory I: Correspondences}, 
      author={Germán Stefanich},
      year={2020},
      eprint={2011.03027},
      archivePrefix={arXiv},
      primaryClass={math.AT},
}

@article {piacenza,
    AUTHOR = {Piacenza, Robert J.},
     TITLE = {Homotopy theory of diagrams and {CW}-complexes over a
              category},
   JOURNAL = {Canad. J. Math.},
  FJOURNAL = {Canadian Journal of Mathematics. Journal Canadien de
              Math\'ematiques},
    VOLUME = {43},
      YEAR = {1991},
    NUMBER = {4},
     PAGES = {814--824},
}

@article{bhs25,
   title={Envelopes for algebraic patterns},
   volume={25},
   number={9},
   journal={Algebr. Geom. Topol.},
   author={Barkan, Shaul and Haugseng, Rune and Steinebrunner, Jan},
   year={2025},
   pages={5319–5388} }

@article{krause20,
  AUTHOR = {Krause, Achim},
     TITLE = {The {P}icard group in equivariant homotopy theory via stable
              module categories},
   JOURNAL = {J. Topol.},
  FJOURNAL = {Journal of Topology},
    VOLUME = {18},
      YEAR = {2025},
    NUMBER = {2},
     PAGES = {Paper No. e70020, 41},
}

@misc{kerodon,
    title        = {Kerodon},
    author       = {Jacob Lurie},
    howpublished = {\url{https://kerodon.net}},
    shorthand    = {kerodon} 
}

@misc{nat-thesis,
  title = {Equivariant higher algebra},
  note={thesis},
  author = {Natalie Stewart},
  year = {2026}
}

@misc{tensor,
      title={On tensor products with equivariant commutative operads}, 
      author={Natalie Stewart},
      year={2025},
      eprint={2504.02143},
      archivePrefix={arXiv},
      primaryClass={math.AT},
}

@article{stephan,
   title={On equivariant homotopy theory for model categories},
   volume={18},
   number={2},
   journal={Homology, Homotopy and Applications},
   author={Stephan, Marc},
   year={2016},
   pages={183–208} 
}

@article{Moser19,
   AUTHOR = {Moser, Lyne},
     TITLE = {Injective and projective model structures on enriched diagram
              categories},
   JOURNAL = {Homology Homotopy Appl.},
  FJOURNAL = {Homology, Homotopy and Applications},
    VOLUME = {21},
      YEAR = {2019},
    NUMBER = {2},
     PAGES = {279--300},
}

@article {ilman83,
    AUTHOR = {Illman, S\"oren},
     TITLE = {The equivariant triangulation theorem for actions of compact
              {L}ie groups},
   JOURNAL = {Math. Ann.},
  FJOURNAL = {Mathematische Annalen},
    VOLUME = {262},
      YEAR = {1983},
    NUMBER = {4},
     PAGES = {487--501},
}

@article {tD75,
    AUTHOR = {tom Dieck, Tammo},
     TITLE = {Orbittypen und \"aquivariante {H}omologie. {II}},
   JOURNAL = {Arch. Math. (Basel)},
  FJOURNAL = {Archiv der Mathematik},
    VOLUME = {26},
      YEAR = {1975},
    NUMBER = {6},
     PAGES = {650--662},
}

@book{schwede2018global,
  title={Global homotopy theory},
  author={Schwede, Stefan},
  year={2018},
  publisher={Cambridge University Press}
}

@misc{Barwick-intro,
      title={Parametrized higher category theory and higher algebra: A general introduction}, 
      author={Clark Barwick and Emanuele Dotto and Saul Glasman and Denis Nardin and Jay Shah},
      year={2016},
      eprint={1608.03654},
      archivePrefix={arXiv},
      primaryClass={math.AT},
}

@misc{bhm22,
      title={Norms for compact Lie groups in equivariant stable homotopy theory}, 
      author={Andrew J. Blumberg and Michael A. Hill and Michael A. Mandell},
      year={2022},
      eprint={2212.11404},
      archivePrefix={arXiv},
      primaryClass={math.AT},
}

@misc{norms,
    title={Dimension-shifting norms in equivariant homotopy theory for compact Lie groups},
    author={Branko Juran and Natalie Stewart},
    note={in preparation}
}

@article {bh21,
    AUTHOR = {Bachmann, Tom and Hoyois, Marc},
     TITLE = {Norms in motivic homotopy theory},
   JOURNAL = {Ast\'{e}risque},
  FJOURNAL = {Ast\'{e}risque},
    NUMBER = {425},
      YEAR = {2021},
     PAGES = {ix+207},
}

@book{kelly1982basic,
  AUTHOR = {Kelly, Gregory Maxwell},
     TITLE = {Basic concepts of enriched category theory},
    SERIES = {London Mathematical Society Lecture Note Series},
    VOLUME = {64},
 PUBLISHER = {Cambridge University Press, Cambridge-New York},
      YEAR = {1982},
     PAGES = {245},
}

@article{roos,
  AUTHOR = {Roos, Jan-Erik},
     TITLE = {Introduction \`a{} l'\'etude de la distributivit\'e{} des
              foncteurs {$\varprojlim$} par rapport aux {$\varinjlim$} dans
              les cat\'egories des faisceaux (topos)},
   JOURNAL = {C. R. Acad. Sci. Paris},
  FJOURNAL = {Comptes Rendus Hebdomadaires des S\'eances de l'Acad\'emie des
              Sciences},
    VOLUME = {259},
      YEAR = {1964},
     PAGES = {969--972},
}

@article{HHR,
    AUTHOR = {Hill, M. A. and Hopkins, M. J. and Ravenel, D. C.},
     TITLE = {On the nonexistence of elements of {K}ervaire invariant one},
   JOURNAL = {Ann. of Math. (2)},
  FJOURNAL = {Annals of Mathematics. Second Series},
    VOLUME = {184},
      YEAR = {2016},
    NUMBER = {1},
     PAGES = {1--262},
}

@article{greenleesMayMU,
  AUTHOR = {Greenlees, J. P. C. and May, J. P.},
     TITLE = {Localization and completion theorems for {$MU$}-module
              spectra},
   JOURNAL = {Ann. of Math. (2)},
  FJOURNAL = {Annals of Mathematics. Second Series},
    VOLUME = {146},
      YEAR = {1997},
    NUMBER = {3},
     PAGES = {509--544},
}

@misc{hovey1998monoidalmodelcategories,
      title={Monoidal model categories}, 
      author={Mark Hovey},
      year={1998},
      eprint={math/9803002},
      archivePrefix={arXiv},
      primaryClass={math.AT},
}

@article{MartiniWolf2024,
  AUTHOR = {Martini, Louis and Wolf, Sebastian},
     TITLE = {Colimits and cocompletions in internal higher category theory},
   JOURNAL = {High. Struct.},
  FJOURNAL = {Higher Structures},
    VOLUME = {8},
      YEAR = {2024},
    NUMBER = {1},
     PAGES = {97--192},
}

@misc{bdgns16,
      title={Parametrized higher category theory and higher algebra: Expos\'e I -- Elements of parametrized higher category theory}, 
      author={Clark Barwick and Emanuele Dotto and Saul Glasman and Denis Nardin and Jay Shah},
      year={2016},
      eprint={1608.03657},
      archivePrefix={arXiv},
      primaryClass={math.AT},
}

@misc{cll23,
      title={Parametrized stability and the universal property of global spectra}, 
      author={Bastiaan Cnossen and Tobias Lenz and Sil Linskens},
      year={2023},
      eprint={2301.08240},
      archivePrefix={arXiv},
      primaryClass={math.AT},
}

@article {csy20,
    AUTHOR = {Carmeli, Shachar and Schlank, Tomer M. and Yanovski, Lior},
     TITLE = {Ambidexterity and height},
   JOURNAL = {Adv. Math.},
  FJOURNAL = {Adv. Math.},
    VOLUME = {385},
      YEAR = {2021},
     PAGES = {Paper No. 107763, 90},
}

@article {hoyoisSixOperations,
    AUTHOR = {Hoyois, Marc},
     TITLE = {The six operations in equivariant motivic homotopy theory},
   JOURNAL = {Adv. Math.},
  FJOURNAL = {Advances in Mathematics},
    VOLUME = {305},
      YEAR = {2017},
     PAGES = {197--279},
}

@misc{NS,
      title={Parametrized and equivariant higher algebra}, 
      author={Denis Nardin and Jay Shah},
      
      year={2022},
      eprint={2203.00072},
      archivePrefix={arXiv},
      primaryClass={math.AT},
}

@article{Cnossen2023,
  author  = {Cnossen, Bastiaan},
  title   = {Twisted ambidexterity in equivariant homotopy theory},
  journal = {J. Topol.},
  volume  = {19},
  number  = {2},
  year    = {2026}
}

@misc{QuinnZhu,
      title={Multiplicative Equivariant Thom Spectra \& Structured Real Orientations}, 
      author={Ryan Quinn and Qi Zhu},
      year={2026},
      eprint={2512.15573},
      archivePrefix={arXiv},
      primaryClass={math.AT},
}

@misc{carrickSlice,
      title={Slice spectral sequences through synthetic spectra}, 
      author={Christian Carrick},
      year={2025},
      eprint={2510.19501},
      archivePrefix={arXiv},
      primaryClass={math.AT},
}

@article{quigleyShahParametrisedTate,
  AUTHOR = {Quigley, J. D. and Shah, Jay},
     TITLE = {On the parameterized {T}ate construction},
   JOURNAL = {J. Topol.},
  FJOURNAL = {Journal of Topology},
    VOLUME = {18},
      YEAR = {2025},
    NUMBER = {1},
}

@misc{horevKlangZou,
      title={Equivariant nonabelian Poincar\'e duality and equivariant factorization homology of Thom spectra}, 
      author={Asaf Horev and Inbar Klang and Foling Zou},
      year={2024},
      eprint={2006.13348},
      archivePrefix={arXiv},
      primaryClass={math.AT},
      note ={with Appendix by Jeremy Hahn and Dylan Wilson}
}

@article{levyThomSpectra,
  title={{Eilenberg Mac Lane spectra as $p$-cyclonic Thom spectra}},
  author={Levy, Ishan},
  year = {2022},
  journal = {J. Topol.},
    volume = {15},
    number ={2},
pages ={878-895}
}

@article{elmantoHaugseng,
  title={{On distributivity in higher algebra I: The universal property of bispans}},
  author={Elmanto, E. and Haugseng, R.},
  year = {2023},
  journal = {Compos. Math.},
    volume = {159 (11)},
pages ={2326--2415}
}

@article {Sha23,
    AUTHOR = {Shah, Jay},
     TITLE = {Parametrized higher category theory},
   JOURNAL = {Algebr. Geom. Topol.},
  FJOURNAL = {Algebraic \& Geometric Topology},
    VOLUME = {23},
      YEAR = {2023},
    NUMBER = {2},
     PAGES = {509--644},
}

@misc{annalaIwasa,
      title={Motivic spectra and universality of $K$-theory}, 
      author={Toni Annala and Ryomei Iwasa},
      year={2025},
      eprint={2204.03434},
      archivePrefix={arXiv},
      primaryClass={math.AG},
}

@misc{brink2026globalinftycategoriesglobalthom,
      title={Global $\infty$-categories and global Thom spectra}, 
      author={Emma Brink and Tobias Lenz},
      year={2026},
      eprint={2608.28504},
      archivePrefix={arXiv},
      primaryClass={math.AT},
      url={https://arxiv.org/abs/2608.28504}, 
}

@article{robaloKTheoryBridge,
  AUTHOR = {Robalo, Marco},
     TITLE = {{$K$}-theory and the bridge from motives to noncommutative
              motives},
   JOURNAL = {Adv. Math.},
  FJOURNAL = {Advances in Mathematics},
    VOLUME = {269},
      YEAR = {2015},
     PAGES = {399--550},
}

@article{llp25,
  AUTHOR = {Lenz, Tobias and Linskens, Sil and P\"utzst\"uck, Phil},
     TITLE = {Norms in equivariant homotopy theory},
   JOURNAL = {Forum Math. Sigma},
  FJOURNAL = {Forum of Mathematics. Sigma},
    VOLUME = {14},
      YEAR = {2026},
     PAGES = {Paper No. e54, 96},
}

@misc{nardinThesis,
      title={Stability and distributivity over orbital $\infty$-categories},
      note={thesis},
      author={Nardin, Denis},
      year={2017},
      URL={https://dspace.mit.edu/handle/1721.1/112895?show=full}
}

@book {lur09,
    AUTHOR = {Lurie, Jacob},
     TITLE = {Higher topos theory},
    SERIES = {Annals of Mathematics Studies},
    VOLUME = {170},
 PUBLISHER = {Princeton University Press, Princeton, NJ},
      YEAR = {2009},
     PAGES = {xviii+925},
}

@article {dhs,
    AUTHOR = {Hopkins, Michael J. and Smith, Jeffrey H.},
     TITLE = {Nilpotence and stable homotopy theory. {II}},
   JOURNAL = {Ann. of Math. (2)},
  FJOURNAL = {Annals of Mathematics. Second Series},
    VOLUME = {148},
      YEAR = {1998},
    NUMBER = {1},
     PAGES = {1--49},
}

@book {Adamek,
    AUTHOR = {Ad\'amek, Ji\v r\'i{} and Rosick\'y, Ji\v r\'i},
     TITLE = {Locally presentable and accessible categories},
    SERIES = {London Mathematical Society Lecture Note Series},
    VOLUME = {189},
 PUBLISHER = {Cambridge University Press, Cambridge},
      YEAR = {1994},
     PAGES = {xiv+316},
      ISBN = {0-521-42261-2},
   MRCLASS = {18Axx (18-02)},
  MRNUMBER = {1294136},
MRREVIEWER = {J.\ R.\ Isbell},
       DOI = {10.1017/CBO9780511600579},
       URL = {https://doi.org/10.1017/CBO9780511600579},
}

@online{lur17,
    author={Lurie, Jacob},
    title={Higher Algebra},
    url={https://people.math.harvard.edu/~lurie/papers/HA.pdf},
    year={2017}
}

@article {Rubin,
    AUTHOR = {Rubin, Jonathan},
     TITLE = {Combinatorial {$N_\infty$} operads},
   JOURNAL = {Algebr. Geom. Topol.},
  FJOURNAL = {Algebraic \& Geometric Topology},
    VOLUME = {21},
      YEAR = {2021},
    NUMBER = {7},
     PAGES = {3513--3568},
}

@article{Barthel_compact,
   title={On the Balmer spectrum for compact Lie groups},
   volume={156},
   number={1},
   journal={Compos. Math.},
   publisher={Wiley},
   author={Barthel, Tobias and Greenlees, J. P. C. and Hausmann, Markus},
   year={2019},
   pages={39–76} }

@article {Barthel,
    AUTHOR = {Barthel, Tobias and Hausmann, Markus and Naumann, Niko and
              Nikolaus, Thomas and Noel, Justin and Stapleton, Nathaniel},
     TITLE = {The {B}almer spectrum of the equivariant homotopy category of
              a finite abelian group},
   JOURNAL = {Invent. Math.},
  FJOURNAL = {Inventiones Mathematicae},
    VOLUME = {216},
      YEAR = {2019},
    NUMBER = {1},
     PAGES = {215--240},
}

@article{Barrero_2023,
   title={Global $N_\infty$‐operads},
   volume={55},
   number={3},
   journal={Bull. Lond. Math. Soc. },
   author={Barrero, Miguel},
   year={2023},
   pages={1275–1301} 
}

@article{Barrero_ab,
   title={Global transfer systems of abelian compact Lie groups},
   volume={151},
   number={7},
   journal={Proc. Am. Math. Soc.},
   author={Barrero, Miguel},
   year={2023},
   pages={3169–3182} 
}

@misc{mwpresentable,
	title={Presentability and topoi in internal higher category theory}, 
      author={Louis Martini and Sebastian Wolf},
      year={2025},
      eprint={2209.05103},
      archivePrefix={arXiv},
      primaryClass={math.CT},
}

@misc{shahPaperII,
  title={Parametrized higher category theory II: Universal constructions}, 
      author={Jay Shah},
      year={2022},
      eprint={2109.11954},
      archivePrefix={arXiv},
      primaryClass={math.CT},
}

@article{GepnerHaugsengNikolaus2017,
   AUTHOR = {Gepner, David and Haugseng, Rune and Nikolaus, Thomas},
     TITLE = {Lax colimits and free fibrations in {$\infty$}-categories},
   JOURNAL = {Doc. Math.},
  FJOURNAL = {Documenta Mathematica},
    VOLUME = {22},
      YEAR = {2017},
     PAGES = {1225--1266},
}

@article {hhln23a,
    AUTHOR = {Haugseng, Rune and Hebestreit, Fabian and Linskens, Sil and
              Nuiten, Joost},
     TITLE = {Two-variable fibrations, factorisation systems and
              {$\infty$}-categories of spans},
   JOURNAL = {Forum Math. Sigma},
  FJOURNAL = {Forum of Mathematics. Sigma},
    VOLUME = {11},
      YEAR = {2023},
     PAGES = {Paper No. e111, 70},
}

@article {gepnerMeierTMF,
     AUTHOR = {Gepner, David and Meier, Lennart},
     TITLE = {On equivariant topological modular forms},
   JOURNAL = {Compos. Math.},
  FJOURNAL = {Compositio Mathematica},
    VOLUME = {159},
      YEAR = {2023},
    NUMBER = {12},
     PAGES = {2638--2693},
}

@article{nikolausScholze,
  AUTHOR = {Nikolaus, Thomas and Scholze, Peter},
     TITLE = {On topological cyclic homology},
   JOURNAL = {Acta Math.},
  FJOURNAL = {Acta Mathematica},
    VOLUME = {221},
      YEAR = {2018},
    NUMBER = {2},
     PAGES = {203--409},
}

@article{cnossennorms,
  title={Normed equivariant ring spectra and higher {T}ambara functors},
  author={Cnossen, Bastiaan and Haugseng, Rune and Lenz, Tobias and Linskens, Sil},
  journal={arXiv preprint arXiv:2407.08399},
  year={2024}
}

@misc{hilmanRamzi,
  title={Equivariant localizing motives and multiplicative norms on algebraic K-theory}, 
      author={Kaif Hilman and Maxime Ramzi},
      year={2026},
      eprint={2603.15873},
      archivePrefix={arXiv},
      primaryClass={math.KT},
}

@article{lnp25,
  title={Global homotopy theory via partially lax limits},
  author={Linskens, Sil and Nardin, Denis and Pol, Luca},
  journal={Geometry \& Topology},
  volume={29},
  number={3},
  pages={1345--1440},
  year={2025},
}

@misc{kaifSil,
  title={Parametrised functor calculus: excision, spheres, and semiadditivity}, 
      author={Kaif Hilman and Sil Linskens},
      year={2024},
      eprint={2410.22327},
      archivePrefix={arXiv},
      primaryClass={math.AT},
      url={https://arxiv.org/abs/2410.22327}, 
}

@article{kaifPresentable,
  AUTHOR = {Hilman, Kaif},
     TITLE = {Parametrised presentability over orbital categories},
   JOURNAL = {Appl. Categ. Structures},
  FJOURNAL = {Applied Categorical Structures. A Journal Devoted to
              Applications of Categorical Methods in Algebra, Analysis,
              Computer Science, Logic, Order and Topology},
    VOLUME = {32},
      YEAR = {2024},
    NUMBER = {3},
}

@misc{kaifNoncommMotives,
 title={Parametrised noncommutative motives and equivariant cubical descent in algebraic K-theory}, 
      author={Kaif Hilman},
      year={2026},
      eprint={2202.02591},
      archivePrefix={arXiv},
      primaryClass={math.KT},
}

\end{document}